\documentclass[a4paper,10pt,reqno]{amsart}

\usepackage{amsmath}
\usepackage{amsthm}
\usepackage{amsfonts}
\usepackage{amssymb}
\usepackage{mathrsfs}
\usepackage[shortlabels]{enumitem}
\usepackage{graphicx}
\usepackage{subfigure}
\usepackage{color}
\usepackage[dvipsnames]{xcolor}
\usepackage{stmaryrd}

\usepackage{mathtools}
\usepackage[font=small]{caption}
\mathtoolsset{showonlyrefs=true}

\newtheorem{theorem}{Theorem}[section]
\numberwithin{theorem}{section}
\newtheorem{prop}[theorem]{Proposition}
\newtheorem{lemma}[theorem]{Lemma}

\theoremstyle{definition}
\newtheorem{definition}[theorem]{Definition}
\newtheorem{rem}[theorem]{Remark}
\newtheorem{ass}{Assumption}

\newtheorem{assOm}{Assumption}

\numberwithin{equation}{section}

\newcommand{\N}{\mathbb{N}}

\newcommand{\R}{\mathbb{R}}
\newcommand{\Rd}{\mathbb{R}^d}
\newcommand{\C}{\mathbb{C}}

\newcommand\e{\mathrm{e}}
\newcommand\I{\mathrm{i}}
\newcommand\re{\operatorname{Re}}
\newcommand\im{\operatorname{Im}}

\newcommand{\DD}{\mathcal D}
\newcommand{\rd}{\mathrm{d}}

\newcommand\dom{\mathcal D}

\newcommand\lbar\overline

\newcommand\eps\varepsilon
\renewcommand\epsilon\varepsilon
\renewcommand\rho\varrho
\newcommand\al\alpha
\newcommand\lm\lambda

\newcommand\ds\displaystyle

\newcommand\p\partial

\newcommand{\cf}{\emph{cf.}}
\newcommand{\ie}{{\emph{i.e.}}}
\newcommand{\eg}{{\emph{e.g.}}}
\newcommand{\dist}{\operatorname{dist}}
\newcommand{\sgn}{\operatorname{sgn}}
\newcommand{\supp}{\operatorname{supp}}
\newcommand{\loc}{\mathrm{loc}}
\newcommand{\LRd}{L^2\left(\Rd,\C\right)}

\newcommand{\LjlocRd}{L^1_{\loc}\left(\Rd,\C\right)}

\newcommand{\Llam}{{L^2\left(\Omega_n,\C\right)}}
\newcommand{\CcRdR}{C_0^{\infty}\left(\Rd,\R\right)}
\newcommand{\CcRdC}{C_0^{\infty}\left(\Rd,\C\right)}
\newcommand{\WotRd}{{W^{1,2}\left(\R^d,\C\right)}}
\newcommand{\WttRd}{{W^{2,2}\left(\R^d,\C\right)}}
\newcommand{\Wotlam}{{W^{1,2}\left(\Omega_n,\C\right)}}
\newcommand{\DnDR}{\Delta_n^{\rm DR}} 
\newcommand{\Num}{\operatorname{W}} 
\newcommand{\Tm}{T_{\min}}
\newcommand{\Trm}{T_{0,\min}}
\newcommand{\an}{a_{\nabla}}
\newcommand{\bn}{b_{\nabla}}
\newcommand{\aU}{a_U}
\newcommand{\bU}{b_U}

\renewcommand{\restriction}{\!\upharpoonright\!}
\renewcommand\d{\,{\rm d}}
\newcommand{\Lla}{{{\mathcal L}_\lambda(T)}}

\author{Sabine B\"ogli}
\address[S.\ B\"ogli]{
School of Mathematics,
Cardiff University,
21–-23 Senghennydd Road, Cardiff CF24 4AG, UK}
\email{sabine.boegli@math.unibe.ch}

\author{Petr Siegl}
\address[P.\ Siegl]{
Mathematisches Institut, 
Universit\"{a}t Bern,
Alpeneggstr.\ 22,
3012 Bern, Switzerland
\& On leave from Nuclear Physics Institute ASCR, 25068 \v Re\v z, Czech Republic}
\email{petr.siegl@math.unibe.ch}

\author{Christiane Tretter}
\address[C.\ Tretter]{
Mathematisches Institut, 
Universit\"{a}t Bern,
Sidlerstrasse 5,
3012 Bern, Switzerland
\& Matematiska institutionen, Stockholms universitet, 10691 Stockholm, Sweden
}
\email{tretter@math.unibe.ch}

\thanks{
}

\usepackage[update,prepend]{epstopdf} 

\title[Approximations of spectra of Schr\"odinger operators]{Approximations of spectra of Schr\"odinger operators with complex potentials on $\R^d$}

\begin{document}



\date{\today}

\begin{abstract}
We study spectral approximations of Schr\"odinger operators $T=-\Delta+Q$ with complex potentials on~$\Omega=\Rd$, or exterior domains $\Omega\subset\R^d$, 
by domain truncation. Our weak assumptions cover wide classes of potentials $Q$ for which $T$ has discrete spectrum, of approximating domains~$\Omega_n$, and of 
boundary conditions on $\partial \Omega_n$ such as mixed Dirichlet/Robin type. In parti\-cular, $\re Q$ need not be bounded from below and $Q$ may be singular. 
We prove generalized norm resolvent convergence and spectral exactness, \ie~approximation of \emph{all} eigenvalues of $T$ by those of the truncated operators $T_n$
\emph{without} spectral pollution.  Moreover, we estimate the eigenvalue convergence rate and prove convergence of pseudospectra. 
Numerical computations for several examples, such as complex harmonic and cubic oscillators for $d=1,2,3$, illustrate our results. 
\end{abstract}

\maketitle

\section{Introduction}
Although domain truncation is one of the most commonly used techniques for approximating partial differential operators on unbounded domains, 
it is a major challenge to guarantee its reliability, even if the spectrum is purely discrete.
Not only may the approximation produce \emph{spurious limits} that are no true eigenvalues. 
It may also happen that some true eigenvalues are \emph{not approximated}, in particular for non-selfadjoint operators.
While very recent research and applications show that there is particular interest in  Schr\"odinger operators on unbounded domains with complex potentials \cite{Almog-2008-40, Almog-toappear, Hansen-Wong,Brown-Marletta-2001}, 
there are no general spectral convergence results for domain truncation for this basic class of operators. 

The aim of the present paper is to fill this gap and prove \emph{spectral exactness}, \ie~the absence of the two unwanted phenomena described above, for wide classes of Schr\"odin\-ger operators $T\!=\!-\Delta+ Q$ in $L^2\left(\Omega,\C\right)$ where $\Omega$ is $\R^d$ or an exterior domain in $\R^d$. Our assumptions on the potential, the domains $\Omega_n$ approximating~$\Omega$, and the conditions on the artificial boundaries $\partial\Omega_n$ are very weak. 
For the complex-valued potential $Q$ we only require $|Q(x)| \rightarrow \infty$ as $|x| \rightarrow \infty$ 
and some mild assumptions guaranteeing that $T$ has discrete spectrum; in particular, $\re Q$ need not be bounded from below and $Q$ may be singular. 
For the  approximating operators $T_n=-\Delta+Q$ in $\Llam$ we require no regularity of the bounded domains $\Omega_n$ exhausting $\Omega$ as $n\to\infty$
for Dirichlet conditions on~$\partial \Omega_n$, and only low regularity for mixed Dirichlet-Robin conditions.
Moreover, we establish estimates for the convergence rate of the approximate eigenvalues and convergence of pseudospectra. Our abstract results are illustrated  by numerical computations for several examples of different potentials, dimensions, domains, and boundary conditions.

The notion of \emph{spectral exactness} 
was first introduced in \cite{Bailey-1993} for regular approximations of singular selfadjoint Sturm-Liouville problems by interval truncation.
It means that a sequence of approximating operators $\{T_n \} _n$ has the following two  properties, \cf \ \eg~\cite{Brown-Marletta-2001}:
\begin{enumerate}[\upshape i)]
\item \emph{spectral inclusion}: for every eigenvalue $\lambda \in \sigma(T)$ there exist $\lambda_n\in \sigma(T_n)$, $n\in\N$, with $\lambda_n \to \lambda$ as $n \to \infty$;
\item \emph{no spectral pollution}: if there exists a sequence of eigenvalues $\lm_n\in \sigma(T_n)$, $n\in\N$, with an accumulation point $\lambda \in \C$, then $\lambda \in \sigma(T)$. 
\end{enumerate}
For partial differential operators, results on spectral exactness in the literature are fragmented. Even in the case of Schr\"odinger operators,
explicit proofs of spectral exactness are either confined to selfadjoint or elliptic problems, in both cases restricted to potentials with real part bounded from below, 
\cf~\cite{Marletta-2010, Boulton-2013, Gannot-2014} and references therein, or they cover only the one-dimensional case, \cf~\cite{Combes-1983-52}, or they concern Galerkin approximations, \cf~\cite{Hansen-Wong}. 
Spectral exactness for domain truncation of \emph{non-selfadjoint} differential operators was studied \eg~in~\cite{Brown-Marletta-2001,Brown-Marletta-2003,Brown-2004-24},
where tests for spectral exactness in terms of boundary conditions were developed. 
However, the verification of the assumptions therein proved to be difficult and sometimes impossible, \cf~\cite[Ex.\ 1]{Brown-2004-24}.
Our new result yields spectral exactness also for this previously debated example, \cf~Subsection~\ref{subsec.ext}.

In general, spectral exactness is a major challenge for non-selfadjoint problems. 
In the selfadjoint case, it is well-known that generalized strong resolvent convergence implies spectral inclusion, and if the resolvents converge even in norm, 
then spectral exactness prevails, \cf~\cite[Thm.\ 9.24\ a),\ 9.26\ b)]{weid1} and also~\cite{Weidmann-survey} for a survey on related results.
Here ``generalized'' refers to the fact that the resolvents $(T_n-\lm)^{-1}$ and $(T-\lm)^{-1}$ do not act in the same space. 
In the non-selfadjoint case, norm resolvent convergence excludes spectral pollution, \cf~\cite[Sec.\ IV.3.1]{kato};
however, the approximation need not be spectrally inclusive, \cf~\cite[Ex.\ IV.3.8]{kato}. 
Moreover, in general, generalized strong resolvent convergence is not enough to guarantee spectral exactness even 
if all operators have compact resolvents, \cf~the Galerkin approximation in \cite[Ex.~5]{Boulton-2012-2} 
where a spurious eigenvalue was proved to exist.

In the present paper we establish spectral exactness by proving generalized norm resolvent convergence of $T_n$ to $T=-\Delta+Q$ in $\R^d$, or in exterior domains in $\R^d$, for domain truncation. Striving for minimal assumptions on the potential $Q$, we exploit the interplay between the different parts of the potential $Q$ if we decompose it as
\[
Q= Q_0 - U + W 
\]
where $Q_0$ with $\re Q_0 \ge 0$ is the ``regular" 
part, $-U \le 0$ is the ``non-positive" part, and $W$ is the ``singular" part.
More precisely, the required regularity of $Q_0$, and the way how we introduce the operators $T$ and $T_n$, depend on the sectoriality angle $\theta$ of $Q_0-U$:
%
\begin{enumerate}[\upshape I.]
\item If $\theta<\pi/2$, which requires $U\equiv 0$, 
we can allow for potentials with lower regularity and we use sectorial form techniques to introduce $T$ and $T_n$, \cf~Assumption~\ref{ass.q.sec}; 
\item If $\theta\geq \pi/2$, where $\re Q$ need not be bounded from below, 
we require more regularity and we use perturbation theory for $m$-accretive operator 
to introduce $T$ and $T_n$, \cf~Assumption~\ref{ass.q.non-sec}.

\end{enumerate}
The following two one-dimensional examples illustrate the difference between the two different asumptions:
\begin{equation*}
	\begin{aligned}
		&{\theta < \pi/2, \,\text{Assumption \ref{ass.q.sec}}: } \quad &&Q(x) = (1+\I) x^2 + \I \delta(x), \\
		&\theta \geq \pi/2, \,\text{Assumption \ref{ass.q.non-sec}}:  \quad &&Q(x)=\I x^3 - x^2 + \I x^{-\frac{1}{4}}.
	\end{aligned}
\end{equation*} 

In both cases, the resulting operators $T$ and $T_n$ are quasi-sectorial in the sense of \cite[Sec.\ 2.8]{Haase-2006-169} and they coincide 
if both Assumptions \ref{ass.q.sec} and \ref{ass.q.non-sec} are satisfied. We emphasize that the formulation of our results is independent of the assumption that is satisfied.

The paper is organized as follows. 
In Section \ref{sec.op.Rd}, we establish the two different sets of assumptions on the potential $Q$, introduce the operator $T = -\Delta +Q$ in $\LRd$ 
in two different ways, and provide the necessary results on the 
operator domain, graph norm, and resolvent estimates for $T$ in both cases. 
In Section~\ref{sec.Tn}, we establish the assumptions on the truncated domains $\Omega_n$ and the boundary conditions on the artificial boundary $\partial \Omega_n$, introduce the corresponding
approximating operators $T_n$,  
and study their properties. 
In particular, 
we show that the sequence $\{T_n\}_n$ is uniformly quasi-sectorial, \cf~\cite[Sec.\ 2.1]{Haase-2006-169}, with semi-angle $<\pi/2$ in Case~I and with $\ge \pi/2$ in Case~II;
moreover, in the latter case 
we derive 
uniform resolvent estimates in the complementary sector in the left half-plane. 
In Section \ref{sec.conv}, 
employing results on discretely or collectively compact approximations, \cf~\cite{stummel1,anselonepalmer,Osborn-1975-29}, 
we prove our main theorem on generalized norm resolvent convergence of $T_n$ to $T$, \cf~Theorem~\ref{thm.norm.res.conv}.
In Section \ref{sec.sp.conv}, we use this result to establish spectral exactness 
and estimates on the convergence rate of the approximate eigenvalues, \cf~Theorems~\ref{thm.conv.spectrum} and \ref{thm.conv.rate},
as well as convergence of the pseudospectra of $T_n$ to those of $T$ in Attouch-Wets metric, which is a generalization of Hausdorff metric to unbounded subsets of $\C$, \cf~Theorem~\ref{thm.pseudo}. 
In Section~\ref{sec.ext.dom}, we show that all our theorems generalize to Schr\"odinger operators on exterior domains $\Omega\subset \R^d$ by sketching the necessary modifications in the assumptions and proofs. 
In the final Section~\ref{sec.ex}, we illustrate the abstract results by numerical computations for several examples 
of different potentials $Q$, dimensions $d$, domains $\Omega$, and boundary conditions on $\partial \Omega_n$, including complex cubic and harmonic oscillators.

Throughout this paper, we employ the following conventions.
The Euclidean norm in $\C^d$ is denoted by $|\cdot|$, the corresponding scalar product by $\langle\cdot,\cdot\rangle_{\C^d}$,  
and the Euclidean scalar product in $\R^d$ by a dot. 
A domain $\Omega \subset \R^d$ is an open connected subset; $\Omega$ is called exterior domain if $\R^d\setminus \Omega$ is compact. 
For a subset $\Omega\subset\R^d$, we tacitly view every function $f\!\in\! L^2(\Omega,\C)$ as an element of $L^2(\R^d,\C)$ by extending $f$ by zero outside $\Omega$; conversely, we view every $g\!\in\! L^2(\R^d,\C)$ with $g\restriction \R^d\backslash\Omega=0$ as an element of $L^2(\Omega,\C)$.
The norm and scalar product in $\LRd$ and $\Llam$ are denoted by $\|\cdot\|,$ $\|\cdot\|_n$ and $\langle\cdot,\cdot \rangle$, $\langle\cdot,\cdot \rangle_n$, respectively. All scalar products are linear in the first argument.
Partial derivatives, always understood in the weak sense, are denoted by $\partial_j$ and
we systematically abbreviate $\langle \nabla f, \nabla g \rangle\!:=\!\sum_{j=1}^d \langle \partial_j f, \partial_j g \rangle $, 
$ \|\nabla f\|\!:=\!\| |\nabla f |\|$. 

\section{Schr\"odinger operators with complex potentials on $\R^d$}
\label{sec.op.Rd}

In this section, we establish mild criteria for Schr\"odinger operators $T=-\Delta + Q$ in $\LRd$ with complex-valued potential to have compact resolvent and to qualify for our main result on spectral exactness, 
\cf~Assumption~\ref{ass.q.sec} or \ref{ass.q.non-sec}. Our criteria allow for potentials $Q$ of the form 
\begin{equation}\label{Q}
   Q= Q_0 - U + W, \quad \re Q_0 \ge 0, \ U \ge 0, 
\end{equation}
with real part possibly unbounded from below ($U\not\equiv 0$) and with singular part ($W\not\equiv 0$). The assumptions and construction of the operator $T$ are different for the case that $Q_0 - U$ is 
sectorial with semi-angle $\theta < \pi/2$ ($U\equiv 0$) or $\theta \ge \pi/2$ ($U\not\equiv 0$). The weaker sectoriality assumptions in the latter case necessitate more than the minimal regularity of 
$Q_0$ needed in the former case. 

We remark that if $Q$ satisfies both Assumptions \ref{ass.q.sec} and \ref{ass.q.non-sec}, then the operator $T$ resulting in both cases is the same.

\subsection{Semi-angle $\theta < \pi/2$}
\label{subsec.Rd.sec}

We define the operator $T = -\Delta + Q$ through sectorial forms, \ie~via the first representation theorem, \cf~\cite[Thm.\ VI.2.1]{kato}. The potential $Q$ is viewed as a form $q$ 
that splits into two parts, $q=q_0 + w$. 

The ``regular'' part $q_0$ is generated by $Q_0 \in \LjlocRd$.
The perturbation $w$ is assumed to be bounded outside a ball $B_R(0)$ and $\|\nabla\cdot\|^2$-bounded in $L^2(B_R(0),\C)$ as forms.

Since $w$ need not be closable, also forms representing  $\delta$-like distributions comply with our assumptions.

\begin{ass}\label{ass.q.sec}
The sesquilinear form $q$ decomposes as $q=q_0 + w$ where $q_0$ and $w$ have the following properties. 
The form $q_0$ is generated by $Q_0 \in L^{1}_{\rm \loc}(\Rd,\C)$, \ie~
\begin{equation}\label{q0.def}
q_0[\cdot]:= \int_{\R^d} Q_0 |\cdot|^2\,\rd x, 
\quad \dom(q_0):=\left\{ f\in L^2(\R^d,\C) \, : \, Q_0 |f|^2\in L^1(\R^d,\C) \right\},
\end{equation}
such that
%
\begin{enumerate}[label=(\ref{ass.q.sec}.\rm{\roman{*})}]
\item \label{ass.q.sec.i}
\emph{sectoriality of $Q_0$ with semi-angle $\theta \!<\! \pi/2$}:
there exist $c_0\!>\!0$ and $\theta \!\in\! [0,\pi/2)$ \vspace{-1mm} with
\begin{equation}\label{Q0.sec}
\begin{aligned}
\re Q_0 \geq c_0,
\quad 
|\im Q_0| \leq \tan \theta \, \re Q_0
;
\end{aligned}
\end{equation}
\item \label{ass.q.sec.unbdd} \emph{unboundedness of $Q_0$ at infinity}:
\begin{equation}
|Q_0(x)| \to \infty \ \ {\rm as} \ \ |x| \to \infty.
\end{equation}
\end{enumerate}
For the form $w$, there  exist
$R>r>0$ and $\zeta \in \CcRdR$ with
\begin{equation}
\supp \zeta \subset B_R(0), \quad 0 \leq \zeta \leq 1, \quad \zeta \restriction B_r(0)=1,
\end{equation}
and sesquilinear forms $w_1, w_2$ with $W_0^{1,2}(B_R(0),\C)\!\subset\! \dom(w_1)$, $\dom(w_2)\!=\!L^2(\R^d,\C)$~with
\begin{equation}\label{eq.splitting.w}
\forall\, f\in\dom(w):\quad \sqrt{\zeta}f \in W_0^{1,2}(B_R(0),\C), \quad w[f]=w_1[\sqrt{\zeta}f]+w_2[\sqrt{1-\zeta}f],
\end{equation}
and such that 
\begin{enumerate}[label=(\ref{ass.q.sec}.\rm{\roman{*})}]
\setcounter{enumi}{2}
\item\label{ass.q.sec.w}
\emph{$\|\nabla \cdot\|^2$}-boundedness of $w_1$ in $L^2(B_R(0),\C)$: 
there exist  $a_w \geq 0$, $b_w \in [0,1)$ so that, for every $f \in W_0^{1,2}(B_R(0),\C)$,
\begin{equation}
\begin{aligned}
\quad |w_1[f] | \leq  a_w \|f \|^2 + b_w \| \nabla f \|^2;
\end{aligned}
\end{equation}
\item\label{ass.q.sec.w.2} \emph{boundedness of $w_2$ outside $B_r(0)$}:
there exists $M_w \geq 0$ so that, for every $f\in L^2(\R^d,\C)$,
\begin{equation}\label{ass.w.eq}
|w_2[(1-\chi_r) f]| \leq M_w \|f\|^2,
\end{equation}
where $\chi_r$ is the characteristic function of $B_r(0)$.
\end{enumerate}
\end{ass}

\begin{rem}
Assumption \ref{ass.q.sec.i} can be weakened to
\begin{enumerate}[label=(\ref{ass.q.sec}.\rm{\roman{*}')}]
\item \label{ass.q.sec.i'}
\emph{quasi-sectoriality of $Q_0$ with  semi-angle $\theta < \pi/2$ and rotation angle $\beta \in (-\pi/2,\pi/2)$}: 
there exist $\beta \in (-\pi/2,\pi/2)$, $\mu \in \C$, and $\theta \!\in\! [0,\pi/2)$ with
 \begin{equation}\label{quasi.m.cond}
 \begin{aligned}
 |\arg (\e^{-\I \beta}(Q_0 - \mu))| \leq \theta.
 \end{aligned}
 \vspace{-2mm}
 \end{equation}
\end{enumerate}
Then all main results, \cf~Theorems \ref{thm.norm.res.conv}, \ref{thm.conv.spectrum},  \ref{thm.conv.rate}, and \ref{thm.pseudo},
continue to hold 
if
\begin{enumerate}[label=(\ref{ass.q.sec}.\rm{\roman{*}')}]
\setcounter{enumi}{2}
\item\label{ass.q.sec.w'}
Assumption~\ref{ass.q.sec.w} holds with $b_w \in [0,\cos \beta)$.
\end{enumerate}
Note that \ref{ass.q.sec.i}, \ref{ass.q.sec.w} are the special case $\beta=0$, $\mu=0$ of \ref{ass.q.sec.i'}, \ref{ass.q.sec.w'}. 
\end{rem}

\begin{prop}\label{prop.T.sec.def}
Let Assumption \ref{ass.q.sec} be satisfied. Then 
\begin{enumerate}[\upshape i)]
\item the form $t$ given by
\begin{equation}\label{t.def}
\begin{aligned}
t  := \|\nabla \cdot \|^2 + q_0+ w,
\quad
\dom(t)  := \WotRd\cap\dom(q_0),
\end{aligned}
\end{equation}
is densely defined, closed, sectorial, and $\CcRdC$ is a core of $t$;
\item the $m$-sectorial operator $T$ uniquely determined by $t$ has compact resolvent.
\end{enumerate}
\end{prop}
\begin{proof}
i) We write $t$ in the form $t=t_0+w$ with 
\begin{equation*}
t_0 :=\|\nabla \cdot \|^2 +  q_0, \quad \dom(t_0):=\dom(t).
\end{equation*} 
%
By \eqref{Q0.sec}, for every $f \in \dom(t)$,  
\begin{equation}\label{sec-t0}
\left|
\im t_0[f]
\right|
 =
\left|   
\im q_0[f] 
\right|
\leq \tan \theta \, \re t_0 [f].
\end{equation}
Thus $t_0$ is sectorial and 
closed being the sum of two closed sectorial forms, \cf~\cite[Thm.\ VI.1.31]{kato}. 
The space $\CcRdC$ is a core of $t_0$ since it is a core of $\re t_0$, \cf~\cite[Thm.\ 8.2.1]{Davies-1995} and \cite[Thm.\ VI.1.21]{kato}.

Let $\zeta$ be the function used in Assumption~\ref{ass.q.sec}. Note that $\|\zeta\|_{\infty}=1$.
By Assumption~\ref{ass.q.sec.w}, \ref{ass.q.sec.w.2}, for every $f \in \dom(t),$
\begin{equation}\label{w.est}
\begin{aligned}
|w[f]| &\leq |w_1[\sqrt{\zeta} f] | + |w_2[\sqrt{1- \zeta} f]| 
\\
&\leq a_w \|\sqrt{\zeta} f\|^2 + b_w \| \nabla(\sqrt{\zeta} f)\|^2 +  M_w \|f\|^2 
\\
& \leq b_w 
\left(
\|f \,\nabla\! \sqrt{\zeta} \| + \|\sqrt{\zeta} \,\nabla \!f\|
\right)^2
 + (a_w + M_w)\|f\|^2 
\\
& \leq b_w(1 + \varepsilon) \|\nabla f\|^2 + 
\left(
a_w + M_w + b_w
\left(
1+\frac{1}{\varepsilon} 
\right)\left\| \nabla \sqrt{\zeta} \right\|_{\infty}^2
\right)
\|f\|^2 
\\
& =: b_w(1 + \varepsilon) \|\nabla f\|^2  + C_{w,\varepsilon}\|f\|^2
\end{aligned}
\end{equation}
where $\varepsilon >0$ may be chosen so small that $b_w(1+\eps)<1$. 
Note that $ \|\nabla f\|^2 \leq \re t_0[f]$ by \eqref{Q0.sec}.
Thus the form $w$ is relatively bounded with respect to $\re t_0$, and therefore also with respect to $t_0$ by \eqref{sec-t0}, with relative bound smaller than $1$. 
Hence the~form $t$ is closed and sectorial with $\dom(t)=\dom(t_0)$,  $\CcRdC$ is a core of $t$, and $t$ uniquely determines an $m$-sectorial operator $T$, 
\cf~\cite[Thm.\ VI.3.4,\ VI.1.33,\ VI.2.1~i)]{kato}.

ii) The embedding $\big(\dom(t_0)$, $(\re t_0[\cdot]\!+\!\|\cdot\|^2)^{1/2}\big)\!\hookrightarrow\! L^2(\R^d\!,\C)$ is compact by Rellich's criterion \cite[Thm.\ XIII.65]{reedsimon},
Thus, by~\eqref{w.est} and the choice of $\varepsilon$, so is the embedding $\big(\dom(t),(\re t[\cdot]+(C_{w,\eps}\!+\!1)\|\cdot\|^2)^{1/2}\big)\hookrightarrow L^2(\R^d,\C)$.
Then, by~\cite[Thm.\ XIII.64, part $\rm (iv) \Rightarrow (i)$]{reedsimon}, the selfadjoint operator $\re T$ has compact resolvent and hence so does~$T$ 
due to~\cite[Thm.\ VI.3.3]{kato}.
\end{proof}

\begin{rem}[quasi-sectorial case with semi-angle $\theta<\pi/2$]
\label{quasi.m.rem}
For potentials $Q_0$ satisfying 
\ref{ass.q.sec.i'}, \ref{ass.q.sec.w'} instead of \ref{ass.q.sec.i}, \ref{ass.q.sec.w}, 
the form $t$ uniquely determines a quasi-$m$-sectorial operator $T$ with compact resolvent.
Here quasi-$m$-sectorial means that the operator $\e^{-\I\beta}(T\!-\!\mu)$ is $m$-accretive and its numerical range $\Num (\e^{-\I\beta}(T\!-\!\mu))$~satisfies
\begin{equation}
\Num (\e^{-\I\beta}(T-\mu))\subset\left\{z\in\C:\,|\arg(z)|\leq\theta\right\},
\end{equation}
\cf~\cite[Def.\ III.6.9]{edmundsevans}. 
In fact, one may show, analogously to Proposition~\ref{prop.T.sec.def}, that the shifted and rotated form 
\begin{equation}
\begin{aligned}
\widetilde t&:= \e^{-\I \beta}\|\nabla \cdot \|^2 + \int_{\R^d} \e^{-\I \beta} (Q_0-\mu) |\cdot|^2\,\rd x + \e^{-\I \beta} w,
\\
\dom(\widetilde t) & := \WotRd \cap \{f\in L^2(\R^d,\C)\, :\, \e^{-\I \beta} (Q_0-\mu) |f|^2 \in L^1(\R^d,\C)\},
\end{aligned}
\end{equation}
uniquely determines an $m$-sectorial operator $\widetilde{T}$ with compact resolvent and
$T:= \e^{\I \beta } \widetilde T + \mu$.
Note that $b_w < \cos \beta$ guarantees the relative boundedness of $ \e^{-\I \beta} w$ with respect to $\re (\e^{-\I \beta}\|\nabla \cdot \|^2)$ with relative bound smaller than $1$.
\end{rem}
 
\subsection{Semi-angle $\theta \geq \pi/2$}
\label{subsec.Rd.non.sec}

As in the previous case, we split $Q$ into a ``regular'' part $Q_0$ and perturbations. 
However, now the essential requirement is only $\re Q_0 \geq 0$, which prevents us from using sectorial form techniques.
Instead, we introduce an $m$-accretive operator $T_0=-\Delta + Q_0$ using \cite[Thm.\ VII.2.6,\ Cor.\ VII.2.7]{edmundsevans}. 
Then we add the qualitatively new, non-positive, part $- U$ (controlled by $\im Q_0$) and the singular perturbation $W$ 
(again bounded outside a ball $B_R(0)$, but inside now $\Delta$-bounded in $L^2(B_R(0),\C)$). 

\begin{ass}\label{ass.q.non-sec}
The function $Q \in L^2_{\rm loc}(\Rd,\C)$ decomposes as
\begin{equation}
Q = Q_0 - U + W
\end{equation}
where $\re Q_0 \geq 0$, $U \geq 0$, $U \re Q_0 =0$, $W \in L^2_{\rm loc}(\Rd,\C)$, and the following hold. 
\begin{enumerate}[label=(\ref{ass.q.non-sec}.\rm{\roman{*})}]
\item \label{ass.q.non-sec.reg}
\emph{regularity of $Q_0$ and $U$}:
$Q_0 \!\in\! W^{1,\infty}_{\rm \loc}(\Rd,\C)$, $U \!\in\! L^{\infty}_{\rm \loc}(\Rd,\R)$,
and there exist $\an$, $\bn$, $\aU$, $\bU$  $\geq 0$ such that 
\begin{equation}\label{ass.q.non-sec.rb} 
\begin{aligned}
|\nabla Q_0 |^2 \leq \an + \bn |Q_0|^2,  \quad 
U^2 \leq \aU  +  \bU |\im Q_0|^2;
\end{aligned}
\end{equation}
\item \label{ass.q.non-sec.unbdd} 
\emph{unboundedness of $Q_0$ at infinity}:
\begin{equation}
|Q_0(x)| \rightarrow \infty \ \  {\rm as} \ \  \ {|x| \rightarrow \infty}.
\end{equation}
\end{enumerate}
There exist $R>r>0$ such that
\begin{enumerate}[label=(\ref{ass.q.non-sec}.\rm{\roman{*})}]
\setcounter{enumi}{2}
\item\label{ass.q.non-sec.W}
\emph{$\Delta$-boundedness of $W$ in $L^2(B_R(0),\C)$}: 
there exist $a_W \geq 0$, $b_W \in [0,1)$ such that, for every $f \in W^{2,2}(B_R(0),\C) \cap W^{1,2}_0(B_R(0),\C)$,
\begin{equation}
\|W f \|^2 \leq a_W \|f \|^2 + b_W \|\Delta f\|^2;   
\end{equation}

\item\label{ass.q.non-sec.W.2}
\emph{boundedness of $W$ outside $B_r(0)$}: 
there exists $M_W \geq 0$ such that
\begin{equation}
\|(1-\chi_r) W \|_{\infty} \leq M_{W},
\end{equation}
where $\chi_r$ is the characteristic function of $B_r(0)$.
\end{enumerate}

\end{ass}

%
%

\begin{prop}\label{prop.T.def}
Let Assumption \ref{ass.q.non-sec} be satisfied.
Then  
\begin{enumerate}[\upshape i)]
\item the minimal operator
\begin{equation}\label{Tmin.def}
\begin{aligned}
\Tm  := -\Delta + Q, 
\quad 
\dom(\Tm) := \CcRdC,
\end{aligned}
\end{equation}
is closable with closure
%
\begin{equation}
\begin{aligned}
T  \!=\! -\Delta + Q,
\quad 
\dom(T)  \!=\! \WttRd \cap \{ f \!\in\! \LRd: Q_0 f \!\in\! \LRd \};
\end{aligned}
\hspace{-12mm}
\end{equation}
\item there exist $k,\,K> 0$ such that, for every $f\in \dom(T)$,
\begin{equation}\label{norm.equiv.2}
\begin{aligned}
& 
k\left(\|\Delta f\|^2+ \|Q_0 f\|^2 + \|f\|^2\right) 
\\
&
\leq \|T f\|^2+\|f\|^2
\\
&
\leq K\left(\|\Delta f\|^2+ \|Q_0 f\|^2 + \|f\|^2\right);
\end{aligned}
\end{equation}
\item the embedding $\left(\dom(T),(\|T\cdot\|^2 + \|\cdot\|^2)^{1/2}\right) \hookrightarrow \LRd$ is compact;
\item if, in addition, $\bU <1$, then the resolvent of $T$ is compact. Moreover, for every $b' \in (\max \{\bU,b_W \}, 1)$, there exists $a_{W\!-U}(b') \geq 0$ such that the sector
\begin{equation}\label{hyperbolaset}
\hspace{1.5cm}
{\mathcal R}(b')\!:=\!
\left\{
\lambda \!\in\! \C : \re \lambda \!<\! - \frac{a_{W\!-U}(b')}{1\!-\!\sqrt{b'}}, 
|\im \lm| \!<\!  \frac{1\!-\!\sqrt{b'}}{\sqrt{b'}} |\re \lm| \!-\! \frac{a_{W\!-U}(b')}{\sqrt{b'}}
\right\} 
\end{equation}
is a subset of $\rho(T)$ and, for all $\lambda \in {\mathcal R}(b')$,
\begin{equation}\label{Res.dec}
\|(T-\lambda)^{-1}\| \leq \frac{1}{(1-\sqrt{b'})|\re \lm| - \sqrt{b'} |\im \lm| - a_{W\!-U}(b')}.
\end{equation}
\end{enumerate}
\end{prop}

\begin{rem}
Apart from the estimate of the spectrum that follows from iv), there are others which may further narrow down the spectral enclosure, at least for a certain range of $b'\in(0,1)$. 
For example, an estimate similar to the one in the proof of Lemma \ref{lem.tau0.def} shows that there exists $\widetilde a(b') \ge 0$ such that
the hyperbolic \vspace{-2mm} region
\begin{equation}\label{hyperbolaset-old}
\widetilde {\mathcal R}(b'):=
\left\{
\lambda \in \C : \re \lambda < - \sqrt{\frac{2+b'}{1-b'}\, \widetilde a(b')},  |\im \lambda|^2 < \frac{1-b'}{2+b'} | \re \lambda |^2  - \widetilde a(b')
\right\} 
\end{equation}
is a subset of $\rho(T)$ and, with some $d(b') >0$, 
\begin{equation}\label{Res.dec-old}
\|(T-\lambda)^{-1}\| \leq \frac{d(b')}{|\re \lambda|}, \quad \lambda \in \widetilde {\mathcal R}(b').
\end{equation}
In fact, the semi-angle $\widetilde \vartheta\!=\!\arctan\sqrt{\frac{1-b'}{2+b'}}$ of the asymptotes of $\widetilde {\mathcal R}(b')$ is larger than the semi-angle $\vartheta=\arctan  \frac{1\!-\!\sqrt{b'}}{\sqrt{b'}}$ of the sector  ${\mathcal R}(b')$ for $b' \!\in\! (b_0,1)$ with some $b_0\!\in\!(0,1)$, \ie~for these $b'$ the set $\C\setminus \widetilde {\mathcal R}(b')$ gives a tighter spectral enclosure than $\C\setminus {\mathcal R}(b')$; here $b_0$ is a zero of a certain cubic polynomial, $b_0\sim 0.46$.
\end{rem}

\begin{rem}\label{rem:compl}
Proposition \ref{prop.T.def} can be used to slightly extend the completeness result  in \cite{Almog-toappear}. 
Define $b\!:=\!\max\{b_U,b_W\}$ and 
$\vartheta(b)\!:=\!\arctan 
\left(
\max \big\{\frac{1-\sqrt b}{\sqrt b}, \sqrt{\frac{1-b}{2+b}} \big\}
\right)$.~If the selfadjoint
operator $(-\Delta + |Q_0|+1)^{-1}$ in $\LRd$ belongs to the Schatten class ${\mathcal S}_p$  and
\begin{equation}\label{compl.p}
p< \frac{\pi}{2(\pi-\vartheta(b))}, 
\end{equation}
then the system of eigenfunctions and associated functions of $T$ is complete.

This follows from \cite[Cor.\ XI.9.31]{DS2} combined with the bound \eqref{Res.dec} and the fact that the resolvent of $T$ belongs to ${\mathcal S}_p$  if and only if so does $(-\Delta + |Q_0|+1)^{-1}$; the latter is a consequence of \eqref{norm.equiv.2} and the second resolvent identity.

An example which cannot be cast into the setting of \cite{Almog-toappear} is 
the one-dimensional operator 
\begin{equation}\label{ix3.per}
T_{\alpha,\beta}= -\frac{{\rm d}^2}{{\rm d} x^2} + \I |x|^\beta \sgn x - \alpha |x|^\beta, \quad \beta >2, \ \alpha \in [0,1),
\end{equation}
in $L^2(\R,\C)$ for the case $\alpha \neq 0$. 
Nonetheless, our results now imply that its system of eigenfunctions and associated functions is complete if 
\begin{equation}\label{compl.beta}
\beta > 2 \left(\frac \pi{\vartheta(\alpha^2)} - 1\right);
\end{equation}
in fact, here $b_U=\alpha^2$, $b_W=0$, and the eigenvalues $\{\mu_k\}_k$ of $-{\rm d}^2/{\rm d}x^2 + |x|^\beta$ satisfy $\mu_k k^{-\frac{2\beta}{\beta+2}} \to c >0$ as $k\to\infty$, 
see \eg~\cite{Titchmarsh-1954-5}, and hence \eqref{compl.p} is equivalent to \eqref{compl.beta}.
\end{rem}

The proof of Proposition \ref{prop.T.def} uses three technical lemmas which are proved first. 
\begin{lemma}\label{lem.tau0.def}
Let Assumption \ref{ass.q.non-sec} be satisfied and define 
\begin{equation}\label{Trm.def}
\Trm  := -\Delta + Q_0, 
\quad 
\dom(\Trm) := \dom(\Tm) = \CcRdC.
\end{equation}
Then, for every $\varepsilon_1>0$, there exists $C_1(\varepsilon_1) \geq 0$
such that, for every $f \in \dom(\Trm)$,
\begin{equation}\label{tau0.norm.est}
\|\Trm f \|^2 
\geq
(1-\varepsilon_1) 
\left(
\|\Delta f\|^2 + \|Q_0 f\|^2 
\right) 
- C_1(\varepsilon_1) \|f\|^2. 
\end{equation}
\end{lemma}
\begin{proof}
Let $\eps_1>0$.
For $f \in \dom(\Trm)$,
\begin{equation}\label{tau0.est}
\begin{aligned}
\|\Trm  f\|^2 
 =\|-\Delta f + Q_0 f \|^2 
=\|\Delta f\|^2 + \|Q_0 f\|^2  
+ 2 \re \langle -\Delta f, Q_0 f \rangle. 
\end{aligned}
\end{equation}
Using $\re Q_0  \geq 0$ and Assumption \ref{ass.q.non-sec.reg}, we obtain
\begin{equation}
\begin{aligned}
 2\re \langle -\Delta f, Q_0 f \rangle
& = 
2\re \langle \nabla f, f \nabla Q_0 + Q_0 \nabla f \rangle 
\geq 
2\re \langle \nabla f, f \nabla Q_0 \rangle 
\\
& \geq  
- 2 \| \nabla f \| \| f \nabla Q_0\| 
\geq
- \alpha \| f \nabla Q_0\|^2 - \frac 1 \alpha \| \nabla f \|^2 
\\
&
\geq
- \alpha \an \|f\|^2 - \alpha \bn \|  Q_0 f\|^2 -  \frac 1 \alpha \| \nabla f \|^2
\end{aligned}
\end{equation}
where $\alpha >0$ is arbitrary.  
Moreover, for every $\beta >0$,
\begin{equation}\label{nab.est}
\| \nabla f \|^2 
= 
\langle -\Delta f, f \rangle 
\leq
\left\|\Delta f\right\| \left\|f\right\|
\leq
\frac{\beta}{2}\left\|\Delta f\right\|^2+\frac{1}{2\beta}\left\|f\right\|^2.
\end{equation}
By inserting the above inequalities into~\eqref{tau0.est},  we obtain altogether
\begin{equation}
\|\Trm  f\|^2 
\geq
\left(
1- \frac{\beta}{2\alpha}
\right)
\|\Delta f\|^2
+
\left(
1-\alpha \bn 
\right)
\|Q_0 f\|^2 
-
\left(
\alpha \an + \frac{1}{2\alpha\beta} 
\right)
\|f\|^2.
\end{equation}
Now the claim follows if we choose $\alpha= \varepsilon_1/b_\nabla$ and $\beta=2\varepsilon_1^2/b_\nabla$.
\end{proof}

\begin{lemma}\label{lem.W.Delta}
Let Assumption \ref{ass.q.non-sec} be satisfied and let $\Trm$ be as in \eqref{Trm.def}. Then, for every $\varepsilon_2 > 0$, there exists $C_2(\varepsilon_2) \geq 0$ such that, for every $f \in \dom(\Trm)$, 
\begin{equation}
\begin{aligned}
\|Wf\|^2 
& \leq
(b_W + \varepsilon_2) \|\Delta f\|^2 + C_2 (\varepsilon_2) \|f\|^2.
\end{aligned}
\end{equation}
\end{lemma}
\begin{proof}
With the radii $R>r>0$ used in Assumption~\ref{ass.q.non-sec}, we fix $\eta \in C_0^{\infty}(B_R(0),\R)$ such that $0 \leq \eta \leq 1$ and $\eta \restriction B_r(0) = 1$.
Since $\eta f \in C_0^{\infty}(B_R(0),\C)$ for every $f \in \CcRdC$, it follows from Assumptions~\ref{ass.q.non-sec.W},~\ref{ass.q.non-sec.W.2} that
\begin{equation}
\begin{aligned}
\|Wf\| & \leq \|W \eta f\| + \|W (1-\eta) f\| 
\leq
\|W \eta f\| + M_W \|f\|,
\\
\|W \eta f\|^2 & \leq a_W\|\eta f\|^2 + b_W \|\Delta (\eta f)\|^2.
\end{aligned}
\end{equation}
Moreover, we have $\Delta(\eta f) = (\Delta \eta) f + 2 \nabla \eta . \nabla f  + \eta \Delta f$, and the proof can be completed by straightforward estimates using \eqref{nab.est}.
\end{proof}
\begin{lemma}\label{lem.norm.eq}
Let Assumption~\ref{ass.q.non-sec} be satisfied and let $\Tm$ be as in \eqref{Tmin.def}. Then there exist $k,\,K> 0$ such that, for every $f\in \dom(\Tm)$,
\begin{equation}\label{norm.equiv}
\begin{aligned}
& 
k
\left(
\|\Delta f\|^2+ \|Q_0f\|^2 + \|f\|^2
\right) 
\\
&
\leq 
\|\Tm f\|^2+\|f\|^2
\\
&
\leq 
K
\left(
\|\Delta f\|^2+ \|Q_0 f\|^2 + \|f\|^2
\right).
\end{aligned}
\end{equation}

\end{lemma}
\begin{proof}
The upper bound in \eqref{norm.equiv} is immediate from Assumption~\ref{ass.q.non-sec.reg} and Lemma~\ref{lem.W.Delta} as $\Tm=\Trm-U+W$. To show the lower bound, we start from
\begin{equation}\label{tau.est}
\begin{aligned}
\| (\Trm  - U + W) f \|^2 
& =
\|\Trm f \|^2 + \|U f\|^2 + \|W f\|^2 
-
2 \re \langle U f, W f \rangle
\\
&\quad 
+
2 \re \langle \Delta f, (U -W) f \rangle 
+
2 \re \langle Q_0 f, W f \rangle,
\end{aligned}
\end{equation}
where we used $ 2 \re \langle Uf, Q_0 f \rangle=0 $ since $U \re Q_0 =0$ by Assumption \ref{ass.q.non-sec}.
We set $\widetilde \chi_r:=1-\chi_r$ where $\chi_r$ is the characteristic function of $B_r(0)$.
Using Assumptions \ref{ass.q.non-sec.reg} and \ref{ass.q.non-sec.W.2},  we obtain that, for arbitrary $\alpha,\,\beta,\,\gamma>0$, 
\begin{equation}\label{WU.est}
\begin{aligned}
2|\langle U f, W f \rangle | & \leq
2|\langle U f, \widetilde \chi_r W f \rangle | + 2|\langle \chi_r U f, W  f \rangle | 
\\[1mm]
&\leq 
2 M_W \|U f\| \|f\|  + 2 \| W f\| \|\chi_r U \|_{\infty} \|f\|
\\
& \leq
\alpha
\left(
\|U f\|^2 + \| W f\|^2
\right) 
+ 
\frac{1}{\alpha}
\left(
M_W^2 + \|\chi_r U \|_{\infty}^2
\right) 
\|f\|^2 ,
\\[1mm]
2 |\langle \Delta f, U f \rangle| & \leq
2 |\langle \widetilde \chi_r \Delta f, U f \rangle| + 2 |\langle \chi_r \Delta f, U f \rangle|
\\
&\leq
\beta \|\widetilde \chi_r \Delta f \|^2 + \frac{1}{\beta} \|U f\|^2 
+ \alpha \|\chi_r \Delta f\|^2 
+ \frac{1}{\alpha} \|\chi_{r}  U \|_{\infty}^2 \|f\|^2, 
\end{aligned} \vspace{-2mm}
\end{equation}
and, \vspace{-1mm} analogously,
\begin{equation}\label{WU.est2}
\begin{aligned}
2 | \langle \Delta f, W f \rangle| &  \leq
\gamma \|\chi_r \Delta f \|^2 + \frac{1}{\gamma} \|Wf\|^2 + 
\alpha \|\widetilde \chi_r \Delta f\|^2 + \frac{1}{\alpha} M_W^2 \|f\|^2,
\\[-1mm]
2|\langle Q_0 f, W f \rangle| & \leq 
\alpha
\left(
\|Q_0 f \|^2 + \|W f\|^2
\right)
+
\frac{1}{\alpha}
\left(
M_W^2 + \|Q_0 \chi_r\|_{\infty}^2
\right)
\|f\|^2. 
\hspace{15mm}
\end{aligned}
\end{equation}
Inserting these estimates into~\eqref{tau.est} and applying Lemma \ref{lem.tau0.def} with arbitrary $\eps_1>0$, we conclude that there exists $C_3(\eps_1, \alpha)\geq 0$ such that
\begin{equation}\label{eq.Tmf.est}
\begin{aligned}
 \|\Tm f\|^2
&=\| (\Trm  - U + W) f \|^2 \\
&\geq 
\left(
1- \epsilon_1
\right)
\left(
\|\Delta f\|^2
+
\|Q_0 f\|^2 
\right)
+
\|Uf\|^2 + \|W f\|^2
\\
&\quad
-
\beta \|\widetilde \chi_r \Delta f\|^2 
- \gamma \|\chi_r \Delta f\|^2 
-\frac{1}{\beta}
\|Uf\|^2 
- \frac{1}{\gamma} \|W f\|^2
\\
& \quad
-
\alpha
\left(
\|\Delta f\|^2 + 2 \|Wf\|^2 + \|Uf\|^2 + \|Q_0 f\|^2 
\right)
-
C_3(\eps_1, \alpha)
\|f\|^2
\\
& \geq 
\left(
1-\varepsilon_1 -\max\{\beta,\gamma\} - \alpha
\right)
\|\Delta f\|^2 + 
\left(
1-\varepsilon_1 - \alpha
\right)
\|Q_0 f\|^2
\\
& \quad 
-
\Big(
\frac{1}{\beta} \!+\! \alpha \!-\! 1
\Big)
\|Uf\|^2
-
\Big(
\frac{1}{\gamma} \!+\! 2\alpha \!-\! 1
\Big)
\|Wf\|^2
- C_3(\eps_1, \alpha)
\|f\|^2.
\end{aligned}
\end{equation}
We choose $\varepsilon_2 >0$ so small that $b_W':=b_W + \varepsilon_2 <1$.
Then, for $\beta$ and $\gamma$ such that
$\bU /(\bU+1) < \beta < 1$
and $\max\{b_W',\beta\} < \gamma <1$, 
we have
\begin{equation}
b:=1-\bU \left(\frac{1}{\beta}-1\right)>0,
\quad
c:=1-\max\{\beta,\gamma\} - b_W' \left(\frac{1}{\gamma}-1\right)>0.
\end{equation}
In order to further estimate $\|Uf\|^2$, $\|Wf\|^2$ in~\eqref{eq.Tmf.est}, we 
note that, since $\beta$, $\gamma<1$, their coefficients satisfy $  \frac{1}{\beta} + \alpha - 1>0$ and $\frac{1}{\gamma} + 2\alpha -1>0$.
Assumption~\ref{ass.q.non-sec.reg} and Lemma~\ref{lem.W.Delta}, applied with the chosen $\eps_2$, 
imply that there exists $C_4(\eps_1, \alpha)\geq 0$~with
\begin{equation}
\begin{aligned}
\hspace{-2mm} \|\Tm f\|^2\!
&\geq\! 
\left(
1-\varepsilon_1 -\max\{\beta,\gamma\} - \alpha
-
b_W'
\Big(
\frac{1}{\gamma} + 2\alpha -1
\Big) 
\right)
\|\Delta f\|^2
\\
& \quad
+
\left(
1-\varepsilon_1 - \alpha
-
\bU \Big(
\frac{1}{\beta} + \alpha - 1
\Big)
\right)
\|Q_0f\|^2
-
C_4(\eps_1, \alpha) \|f\|^2\\
&=\!
\big(c\!-\!\eps_1\!-\!\alpha(1\!+\!2b_W')\big)
\|\Delta f\|^2
\!\!+\!
\big(b\!-\!\eps_1\!-\!\alpha(1+b_U)\big)
\|Q_0f\|^2
\!\!-\!
C_4(\eps_1, \alpha) \|f\|^2
\!.\!
\end{aligned}
\end{equation}
Finally, choosing $\varepsilon_1$ and $\alpha$ sufficiently small, we find that there exists $C\geq 0$ with
\begin{equation}
\begin{aligned}
 \|\Tm f\|^2
&\geq
C
\left(
\|\Delta f\|^2+ \|Q_0 f\|^2
\right)
-
C_4(\eps_1, \alpha) \|f\|^2,
\end{aligned}
\end{equation}
and hence
\begin{align*}
 \big(C_4(\eps_1, \alpha) +1\big) \left(\|\Tm f\|^2+\|f\|^2\right) 
 & \geq \|\Tm f\|^2+ \big( C_4(\eps_1, \alpha) + 1 \big) \|f\|^2 \\
 & \geq C \left( \|\Delta f\|^2 + \|Q_0 f\|^2 \right) + \|f\|^2.
\end{align*}
Now the lower bound in \eqref{norm.equiv} follows with $k:= \min\{C,1\}/(C_4(\eps_1, \alpha) + 1)$.
%
\end{proof}

\begin{proof}[Proof of Proposition {\rm \ref{prop.T.def}}]
 i) Since $\re Q_0 \geq 0$, the operator $\Trm$ is closable and its closure $T_0$
has the domain
\begin{equation}
\dom(T_0)=\left\{f\in \WotRd: \, (-\Delta + Q_0) f\in \LRd \right\}, 
 \end{equation}
\cf~\cite[Cor.\ VII.2.7]{edmundsevans}. 
Lemma~\ref{lem.norm.eq} applied to $\Tm$ and $\Trm$ (which is $\Tm$ with $U = W =0$) yields the existence of $k,\, K, \,k_0,\,K_0>0$ so that, for every $f\in \dom(\Tm)$,
\begin{equation}\label{taus.equiv}
\begin{aligned}
\frac{k}{K_0}
\left(
\|\Trm f\|^2+\|f\|^2
\right)
 \leq 
\|\Tm f\|^2+\|f\|^2
 \leq 
\frac{K}{k_0}
\left(
\|\Trm f\|^2+\|f\|^2
\right).
\end{aligned}
\end{equation}
Hence $\Tm$ is closable as well and its closure $T$ satisfies $\dom(T)=\dom(T_0)$.
The in\-clusion $\WttRd \cap \{ f \in \LRd: Q_0 f \in \LRd \} \subset \dom(T)$ is obvious. 
It remains to prove the opposite inclusion. 
Since $\dom(\Tm)\!=\!\CcRdC$ is a core of~$T$, Lemma~\ref{lem.norm.eq} 
and the equivalence of $(\|\Delta \cdot\|^2 + \|\cdot\|^2)^{1/2}$ with $\|\cdot\|_{\WttRd}$	
imply~that 
\begin{equation}
\begin{aligned}
\dom(T) &= \overline{\CcRdC}^{\, (\|T\cdot\|^2 + \|\cdot\|^2)^\frac 12} 
= \overline{\CcRdC}^{\, (\|\cdot\|_{\WttRd}^2 + \|Q_0 \cdot\|^2)^\frac 12}
\\
& \subset \WttRd \cap \{ f \!\in\! \LRd: Q_0 f \!\in\! \LRd \}. 
\end{aligned}
\end{equation}


ii) The claim follows from Lemma~\ref{lem.norm.eq} and the fact that $\dom(\Tm)$ is a core of~$T$.

iii)
The embedding $\left(\dom(T),(\|T\cdot\|^2 + \|\cdot\|^2)^{1/2}\right) \hookrightarrow \LRd$ is compact due to~\eqref{norm.equiv.2} and Rellich's criterion \cite[Thm.\ XIII.65]{reedsimon}. 

iv)
The compactness of the resolvent follows from claim iii) if we know that $\rho(T)\neq\emptyset$.
This will follow from the remaining claims in iv) since  $\mathcal R(b')\ne \emptyset$.

To prove that $\mathcal R(b') \subset\rho(T)$ for every $b'\in (\max\{b_U,b_W\},1)$,
we first observe that, for every $f \in \dom(\Trm)=\CcRdC$, the first estimate in \eqref{WU.est} yields 
\begin{equation}
\begin{aligned}
\|(W - U)f \|^2 & \leq 
(1+\alpha)
\left(
\|U f\|^2  + \|Wf\|^2
\right)
+  C(\alpha)\|f\|^2
\end{aligned}
\end{equation}
where $\alpha>0$ is arbitrary and $ C(\alpha)\geq 0$. 
Assumption~\ref{ass.q.non-sec.reg},  Lemma~\ref{lem.W.Delta} applied with $\eps_2=\alpha b / (1+\alpha)$,  and Lemma~\ref{lem.tau0.def} applied with $\eps_1=\alpha/(1+3\alpha)$, 
imply the existence of 
$\widetilde C_1(\alpha)$, $\widetilde C_2(\alpha)\geq 0$ such that, for all $f \in \dom(\Tm)$ and with $b:=\max\{b_U,b_W\}$,
\begin{equation}
\begin{aligned}
\|(W - U)f \|^2 & \leq 
b(1+2\alpha) 
\left(
\|\Delta f\|^2 + \| Q_0 f\|^2
\right)
+ \widetilde C_1(\alpha) \|f\|^2
\\
& \leq
b(1+3\alpha)
\|T_0 f\|^2 +
\widetilde C_2(\alpha)
\|f\|^2.
\end{aligned}
\end{equation}
The latter remains valid  for all $f \in \dom(T_0)$ since $\dom(\Trm)$ is a core of $T_0$. 

If $b'\in (\max\{b_U,b_W\},1)=(b,1)$ is arbitrary, we choose $\alpha$ such that $b' = b(1+3\alpha)$
and so there exists $a_{W\!-U}(b') \geq 0$ such that, for every $f\in\dom(T_0)$,
\begin{equation}\label{W-U}
\|(W - U)f \| 
\leq 
a_{W\!-U}(b') \|f\| + \sqrt{b'} \|T_0 f\|.
\end{equation}

Now let $\lm \in {\mathcal R}(b')$. We verify the assumptions of \cite[Thm.\ IV.3.17]{kato} with the unperturbed operator chosen as $T_0$, the perturbation as $W-U$, and $\zeta=\lm$. 
Because $T_0$ is $m$-accretive and $\lm\in{\mathcal R}(b')$ satisfies $\re \lm<0$, we have
$\re \lm\in\rho(T_0)$, \linebreak 
$\|(T_0 - \lm )^{-1}\| \leq |\re \lambda|^{-1}$, and $\|T_0(T_0-\re\lm)^{-1}\| \leq 1$, \cf~\cite[Sec.\ V.10, Prob.\ V.3.31]{kato}.
Notice that the first resolvent identity yields
\begin{align}
\|T_0(T_0-\lambda)^{-1}\| 
& = 
\|T_0(T_0-\re \lambda)^{-1}(I + \I \im \lm(T_0- \lambda)^{-1} ) \|
\leq
1+ \frac{|\im \lm|}{|\re \lm|}.  
\end{align}
Hence, for all $\lm \in {\mathcal R}(b')$,
\begin{equation}
\begin{aligned}
& 
a_{W\!-U}(b') \|(T_0\!-\!\lm)^{-\!1} \| \!+\! \sqrt{b'} \|T_0(T_0\!-\! \lm)^{-\!1}\| 
\!\leq\! 
\frac{a_{W\!-U}(b') + \sqrt{b'}  |\im \lm| }{|\re\lm|}  \!+\! \sqrt{b'}
\!\!<\! 1,
\end{aligned}
\end{equation}
and so the inequality \cite[IV.(3.12)]{kato} holds. 
Thus \cite[Thm.\ IV.3.17]{kato} implies both $\lm \in \rho(T_0- U + W )$ and the estimate \eqref{norm.equiv.2}. 
\end{proof}

\section{Approximating operators in $\Omega_n \subset \R^d$}
\label{sec.Tn}

In this section, we define an approximating sequence $\{T_n\}_n$ of operators $T_n$ in $\Llam$ where $\Omega_n \subset \Rd$ are bounded domains, \ie~open and connected subsets, that exhaust~$\Rd$ eventually. 
In order to work with operators with non-empty resolvent sets, we need to specify boundary conditions. 

If the aim is to approximate $T$ with simple operators $T_n$, then one can choose $\Omega_n$ for instance as expanding balls and impose Dirichlet boundary conditions. 
If the aim is to compare, or optimize, the convergence rate for the approximate eigenvalues, it may be necessary to consider other, more general, boundary conditions such as Robin conditions or 
mixed Dirichlet-Robin conditions.

Our approximation results cover both situations. For Dirichlet conditions only, we do not require any regularity of the boundary $\partial \Omega_n$. For mixed Dirichlet-Robin conditions on
$\partial \Omega_n = \partial \Omega_n^{\rm D} \, \dot\cup \,  \partial \Omega_n^{\rm R}$, formally given by
\[
  f \restriction \partial \Omega_n^{\rm D} =0, \quad (\partial_\nu f + a_n f) \restriction \partial \Omega_n^{\rm R} =0
\]
where $\partial_\nu$ is the normal derivative on $\partial \Omega_n^{\rm R}$,
we assume $\partial \Omega_n$ is Lipschitz and the functions $a_n\!:\!\partial \Omega_n^{\rm R} \!\to\! \C$ are suitably~bounded, \cf~Assumption \ref{ass.lam.n}. 

\begin{ass}
\label{ass.lam.n}
Let $\{\Omega_n\}_n \subset \Rd$ be a sequence of 
bounded domains satisfying
 %
 \begin{equation}\label{partial.Om.split}
 \partial \Omega_n = \partial \Omega_n^{\rm D} \, \dot\cup \,  \partial \Omega_n^{\rm R}
 \vspace{-1mm}
 \end{equation}
where $\partial \Omega_n^{\rm D}$ is closed and the following hold.
\begin{enumerate}[label=(\ref{ass.lam.n}.\rm{\roman{*})}]
\item \label{ass.lam.n.ex} 
\emph{exhausting property}:
with the radius $R>0$  used in Assumption \ref{ass.q.sec} or \ref{ass.q.non-sec},
there exists $\{r_n\}_n \subset \R$, $r_1>R$,  such that 
\begin{equation}
\overline{B_{r_n+1}(0)} \subset \Omega_n, \quad r_{n+1} > r_n, \quad r_n \rightarrow  \infty.
\end{equation}
\end{enumerate}
If $\partial \Omega_n^{\rm R} \neq \emptyset$ and $d\geq 2$, we additionally assume
\begin{enumerate}[label=(\ref{ass.lam.n}.\rm{\roman{*})}]
\setcounter{enumi}{1}	
\item \label{ass.lam.n.div}
\emph{regularity of $\partial\Omega_n$}: 
$\Omega_n$ is Lipschitz.
\end{enumerate}
%
If $a_n\neq0$, $n \in \N$,  we further assume
\begin{enumerate}[label=(\ref{ass.lam.n}.\rm{\roman{*})}]
\setcounter{enumi}{2}	
\item \label{ass.lam.n.a}
\emph{control of Robin boundary terms}: $a_n \in L^{\infty}(\partial \Omega_n^{\rm R},\C)$, $n\in\N$, and 
\begin{equation}\label{Tr.contr}
M_{\rm Tr}:=\sup_n \|a_n\|_{\infty} K_n <  \infty
\end{equation}
where $K_n>0$ are the constants in the trace embedding 
\begin{equation}\label{emb.est.n}
\int_{\partial \Omega_n} |f|^p \d \sigma 
\leq 
K_n 
\left(
\varepsilon^{1-\frac{1}{p}}
\|\nabla f \|^p_{L^p(\Omega_n,\C)}  +
\varepsilon^{-\frac{1}{p}}
\| f \|^p_{L^p(\Omega_n,\C)}
\right)
\end{equation}
valid for all $f \!\in\! W^{1,p}(\Omega_n,\C)$, $\varepsilon \!\in\! (0,1)$, and $p \!\geq\! 1$, \cf~\cite[Thm.~1.5.1.10]{Grisvard-1985-24}. 
\end{enumerate}
\end{ass}



\begin{rem}
i) For balls or boxes, it can be shown that the constants $K_n$ are uniformly bounded;
then the condition \eqref{Tr.contr} reduces to $ \sup_n \|a_n\|_{\infty} < \infty$. 

ii) 
Sometimes, \eg~in Propositions~\ref{prop.Tn.def.sec},~\ref{prop.comp.res} below, we indicate the dependence of the constants on the constant $M_{\rm Tr}$ in \eqref{Tr.contr}.
\end{rem}

The operators $T_n$ are introduced in several steps, analogously to the definition of~$T$ in the previous section. The main difference is in the first step, \cf~Lemma \ref{lem.s0.def}, where 
we first introduce a Dirichlet-Robin Laplacian  $S_{0,n}:=-\Delta_n^{\rm DR}$ in $\Llam$ via its quadratic form, see \eg~\cite[Sec.\ 7]{Davies-1995} for more details on this approach.

We remark that if $Q$ satisfies both Assumptions \ref{ass.q.sec} and \ref{ass.q.non-sec}, then also the approximating operators $T_n$ introduced in the two different ways coincide. 

\begin{lemma}\label{lem.s0.def}
Let Assumption~\ref{ass.lam.n} be satisfied. Then, for every $n \in \N$, the form
\begin{equation}\label{s0n.def}
\begin{aligned}
s_{0,n}&:= \|\nabla \cdot \|_n^2 + \int_{\partial \Omega_n^{\rm R}} a_n |\cdot|^2 \d \sigma, 
\quad 
\dom(s_{0,n}):=\overline{\DD_n}^{\, \|\cdot\|_\Wotlam},
\end{aligned}
\end{equation}
with 
\begin{equation}\label{Dn.def}
\begin{aligned}
\DD_n&\!:=\!\big \{
f \!\in\! C^{\infty}(\Omega_n,\C) : \exists f_0 \! \in \! \CcRdC, f \!=\! f_0 \restriction \Omega_n, 
 \, \supp f \cap \partial \Omega_n^{\rm D} \!=\! \emptyset 
\big
\}
\end{aligned}
\end{equation}
is densely defined, closed and sectorial and it uniquely determines an $m$-sectorial operator $S_{0,n} = -\Delta_n^{\rm DR}$ which has compact resolvent.
\end{lemma}
\begin{proof}
First observe that 	
\begin{equation}\label{Dom.s0n.encl}
W_0^{1,2}(\Omega_n, \C) \subset 
\overline{\DD_n}^{\, \|\cdot\|_\Wotlam}
\subset \Wotlam.
\end{equation}
The symmetric form $\| \nabla \cdot\|_n^2$ defined on $\dom(s_{0,n})$ is densely defined and closed since $(\dom(s_{0,n}),\|\cdot\|_\Wotlam)$ is complete, \cf~\cite[Thm.\ VI.1.11]{kato}.
The boundary trace embedding~\eqref{emb.est.n}, applied with $p=2$ and arbitrarily small $\eps>0$, together with~\eqref{Tr.contr} implies that the boundary term in~\eqref{s0n.def} is a relatively bounded perturbation of $\| \nabla\cdot \|_n^2$ defined on $\dom(s_{0,n})$ with relative bound $0$. By \cite[Thm.\ VI.1.33]{kato}, the form $s_{0,n}$ is densely defined, closed, and sectorial, hence it uniquely determines an $m$-sectorial operator $S_{0,n}$, \cf~\cite[Thm.\ VI.2.1]{kato}.

Moreover, for sufficiently large $c>0$, the norm $(\re s_{0,n}[\cdot] + c \|\cdot\|_n^2)^{1/2}$ is equivalent to $\|\cdot\|_\Wotlam$. Then, by the Rellich-Kondrachov theorem \cite[Thm.~6.3]{Adams-2003}, 
$(\dom(s_{0,n}), (\re s_{0,n}[\cdot] + c \|\cdot\|_n^2)^{1/2})$ is compactly embedded in $\Llam$. Thus
$\re S_{0,n}$ has compact resolvent and hence so does $S_{0,n}$, \cf~\cite[Thm.~VI.3.3]{kato}.
\end{proof}

\begin{rem}
i) If $\Omega_n$ are sufficiently regular, \eg~$\partial \Omega_n$ is of class $C^2$, and either 
$\partial \Omega_n^{\rm R} = \emptyset$ or $\partial \Omega_n^{\rm D} = \emptyset$ where,  
in the latter case, 
$a_n \in W^{1,\infty}(\partial \Omega_n,\C)$, then in Lemma~\ref{lem.s0.def} the usual domains of Dirichlet or Robin Laplacian are \vspace{-1mm} recovered,
\begin{equation}
\begin{aligned}
\dom(-\Delta_n^{\rm D}) &= W^{2,2}(\Omega_n,\C) \cap W_0^{1,2}(\Omega_n,\C), 
\\
\dom(-\Delta_n^{\rm R})&= \{ f \in W^{2,2}(\Omega_n,\C) \, : \, 
(\partial_\nu f + a_n f) \restriction \partial \Omega_n =0 \},
\end{aligned}
\end{equation}
where $\partial_\nu$ denotes the normal derivative on $\partial \Omega_n=\partial \Omega_n^{\rm R}$.

ii) If the splitting $\partial \Omega_n = \partial \Omega_n^{\rm D} \, \dot\cup \,  \partial \Omega_n^{\rm R}$ satisfies additional, 
very technical, regularity assumptions, \cf~\cite[Prop.\ 3.1]{Kriz-2003-PhD}, \vspace{-1mm}  then 
\begin{equation}
\overline{\DD_n}^{\, \|\cdot\|_\Wotlam}
= 
\left\{
f \in \Wotlam : f \restriction \partial \Omega_n^{\rm D} = 0 \ { \rm a.e.}
\right\}.
\end{equation}
\end{rem}

\subsection{Semi-angle $\theta < \pi/2$}
\label{subsec.Tn.sec}

In this case, the operator $T_n$ is introduced in one step by perturbation arguments using quadratic forms. 

\begin{prop}\label{prop.Tn.def.sec}
Let Assumptions~\ref{ass.q.sec},~\ref{ass.lam.n} be satisfied and let $s_{0,n}$, $q_0$ be the forms defined in \eqref{s0n.def}, \eqref{q0.def}, respectively. 
Then 
\begin{enumerate}[\upshape i)]
\item  for every $n \in \N$,  the form
\begin{equation}\label{tn.def}
t_n := s_{0,n} + q_0 + w, 
\quad
\dom(t_n) := \dom(s_{0,n}) \cap \dom(q_0),
\end{equation}
is densely defined, closed, and sectorial, and it uniquely determines an $m$-sectorial operator $T_n$ which has compact resolvent.
\vspace{1mm}
\item the sequence $\{T_n\}_n$ is uniformly quasi-sectorial with semi-angle $<\pi/2$, \ie~there exist $\mu_0(M_{\rm Tr}) \in \C$ and $\theta_0(M_{\rm Tr}) \in[0,\pi/2)$  such that 
the numerical ranges and spectra of all $T_n$ are contained in the uniform sector
\begin{equation}\label{Tn.sec.S}
\sigma(T_n) \subset
\Num ( T_n) \subset \mathcal S(M_{\rm Tr}) :=\{ z \in \C: |\arg (z-\mu_0(M_{\rm Tr}))| \leq \theta_0(M_{\rm Tr}) \}.
\end{equation}
\end{enumerate}
\end{prop}
\begin{proof}
i)
The form 
\begin{equation}\label{T0n.def.sec}
t_{0,n}:= s_{0,n}+ q_0, \quad \dom(t_{0,n}):=\dom( s_{0,n})\cap\dom(q_0),
\end{equation}
 is the sum of two closed sectorial forms, hence it is closed and sectorial as well, \cf~\cite[Thm.\ VI.1.31]{kato}.
So it  uniquely determines an $m$-sectorial operator $T_{0,n}$. 
Notice that  $\re s_{0,n}[f] \leq \re t_{0,n}[f]$ for all $f \in \dom(t_{0,n})$ and, 
with sufficiently large $c>0$, $(\re s_{0,n}[\cdot] + c \|\cdot\|_n^2)^{1/2}$ is compactly embedded in $\Llam$, \cf~the proof of Lemma~\ref{lem.s0.def} for details. Hence $(\dom(t_{0,n}), (\re t_{0,n}[\cdot] + 
c
\|\cdot\|_n^2)^{1/2})$ is compactly embedded in $\Llam$ and consequently the resolvent of $\re T_{0,n}$ is compact.

By the trace embedding~\eqref{emb.est.n} and Assumption~\ref{ass.lam.n.a}, the boundary term in~\eqref{s0n.def} is relatively bounded with respect to $\|\nabla \cdot\|_n^2$ with relative bound~$0$.

For the form $w$ we first note that, by assumption~\eqref{eq.splitting.w} in Assumption~\ref{ass.q.sec}, for every $f\in \dom(t_n)\subset \Wotlam$ 
we have $\sqrt{\zeta}f\in W_0^{1,2}(B_R(0),\C)\subset\dom(w_1)$ and thus
$w[f] = w_1[\sqrt{\zeta}f] + w_2[\sqrt{1-\zeta}f]$ is well-defined.
Using analogous arguments as in the proof of Proposition \ref{prop.T.sec.def}, 
one can verify that the form $w$ is relatively bounded with respect to 
$\re t_{0,n}+c$  with relative bound smaller than $1$. 
Hence $t_n$ uniquely determines an $m$-sectorial operator $T_n$, \cf~\cite[Thm.\ VI.3.4]{kato}. The latter has compact resolvent since the resolvent of $\re T_n$ is compact, \cf~\cite[Thm.\ VI.3.3]{kato}.

ii)
Using the trace embedding~\eqref{emb.est.n}, Assumptions~\ref{ass.lam.n.a},~\ref{ass.q.sec.i} and the estimate~ \eqref{w.est} on $|w|$, we obtain
\begin{equation}
\begin{aligned}
|\im t_n[f]| 
& \leq
 \bigg|\int_{\partial \Omega_n^{\rm R}} a_n |f|^2 \d  \sigma\bigg|+|\im q_0[f]|+|w[f]|
\\[-1mm]
&\leq 
 \|a_n\|_{\infty} \int_{\partial \Omega_n^{\rm R}} |f|^2 \d  \sigma + \tan \theta \re q_0[f] 
+ b_w(1+\varepsilon) \| \nabla f\|_n^2 + C_{w,\varepsilon} \|f\|_n^2
\\[-1mm]
& \leq (\sqrt{\varepsilon} M_{\rm Tr} + b_w (1+\varepsilon) ) \| \nabla f \|_n^2 +
\tan \theta \re q_0[f] 
+ 
\left(
\frac{M_{\rm Tr}}{\sqrt{\varepsilon}} + C_{w,\varepsilon}
\right)
\|f\|_n^2
\\
&\leq 
\widetilde C_1(\eps,M_{\rm Tr}) 
\left(
\| \nabla f \|_n^2 + \re  q_0[f]
\right)
+\widetilde C_2(\eps,M_{\rm Tr}) \|f\|_n^2.
\end{aligned}
\end{equation}
\vspace{-1mm} Similarly,
\begin{equation}\label{eq.retn.vs.nablaf}
\begin{aligned}
\re t_n[f] 
& \geq 
 (1 \!-\! \sqrt{\varepsilon} M_{\rm Tr} \!-\! b_w (1+\varepsilon) ) \| \nabla f \|_n^2 \!+\!
\re q_0[f]
\!-\! 
\left(
\frac{M_{\rm Tr}}{\sqrt{\varepsilon}} \!+\! C_{w,\varepsilon}
\right)
\|f\|_n^2
\\
&\geq 
\widetilde C_3(\eps,M_{\rm Tr}) 
\left(
\| \nabla f \|_n^2 + \re q_0[f]
\right)
- \widetilde C_4(\eps, M_{\rm Tr}) \|f\|_n^2.
\end{aligned}
\end{equation}
Since $\widetilde C_{i}(\eps,M_{\rm Tr})$, $i=1,\dots,4$, 
are independent of $n$ and positive for $\eps>0$ sufficiently small
because $b_w\in [0,1)$, it follows that, for all $f \in \dom(t_n)$,
\begin{align*}
|\im  t_n[f]| &\leq \frac{\widetilde C_1(\eps,M_{\rm Tr})}{\widetilde C_3(\eps,M_{\rm Tr})} 
\left( 
\re  t_n[f] +\widetilde C_4(\eps,M_{\rm Tr}) \|f\|_n^2
\right)
+\widetilde C_2(\eps,M_{\rm Tr}) \|f\|_n^2.
\end{align*}
Now the inclusion \eqref{Tn.sec.S} for the numerical range follows easily.
The inclusion for the spectrum follows \eg~ since $T_n$ has compact resolvent and so $\rho(T_n) \cap (\C\setminus {\mathcal S}(M_{\rm Tr})) \!\ne\! \emptyset$, \cf~\cite[Thm.~V.3.2]{kato}. 
\end{proof}

\subsection{Semi-angle $\theta \geq \pi/2$}
\label{subsec.Tn.non-sec}

Since $Q_0$ is assumed to be locally bounded, \cf~Assumption~\ref{ass.q.non-sec.reg},
$Q_0f$  is well-defined for all functions  $f\in L^2(\Omega_n,\C)$.
We define $T_{0,n}$  as the operator sum
\begin{equation}\label{T0n.def.non.sec}
\begin{aligned}
T_{0,n} := S_{0,n}+Q_0,
\quad 
\dom (T_{0,n}) := \dom (S_{0,n}), \quad  n \in \N.
\end{aligned}
\end{equation}
%
The operator $T_n$ is introduced in the following proposition by further adding the locally bounded non-positive part $-U$ and the ``singular'' part
$W$ which turns out to be $\DnDR$-bounded with relative bound smaller than 1, \cf~Lemma~\ref{lem.T0n.W}. 
%
\begin{prop}\label{prop.comp.res}
Let Assumptions~\ref{ass.q.non-sec},~\ref{ass.lam.n} be satisfied and let $T_{0,n}$ be as in \eqref{T0n.def.non.sec}.
Then, for every $n \in \N$, 
\begin{enumerate}[\upshape i)]
\item the operator
\begin{equation}
T_n:=T_{0,n} - U + W, \quad \dom(T_n):=\dom(T_{0,n}),
\end{equation}
is closed and has compact resolvent;
\item 
there exist $\widetilde k(M_{\rm Tr}),\, \widetilde K(M_{\rm Tr})> 0$, independent of $n$, such that, for every \mbox{$f\in\dom(T_n)$},
\begin{equation}\label{Tn.graph.norm}
\begin{aligned}
& \widetilde k(M_{\rm Tr})
\left(
\|\DnDR f\|_n^2+ \|Q_0 f\|_n^2 +  \|f\|_n^2
\right) 
\\
&\leq \|T_n f\|_n^2+\|f\|_n^2
\\
&\leq \widetilde K(M_{\rm Tr})
\left(
\|\DnDR f\|_n^2+ \|Q_0 f\|_n^2 + \|f\|_n^2
\right);
\end{aligned}
\end{equation}
\item \label{last!} if \,$\bU\!<\!1$, then
the sequence $\{T_n\}_n$ is uniformly quasi-sectorial, more precisely,~for every $b' \!\in\! (\max \{\bU,b_W\}, 1)$, there is $a_{W\!-U}(b'\!,M_{\rm Tr}) \ge 0$, 
independent of~$n$, such that
the sector
\begin{equation}\label{hyperbolaset2}
\begin{aligned}
\!\!\!\!\mathcal R(b'\!,M_{\rm Tr})\!:=\!
\bigg\{ \!
\lambda \!\in\! \C \!: & 
\re \lambda < \!-\! M_{\rm Tr}^2 - \frac{a_{W\!-U}(b'\!,M_{\rm Tr})}{1\!\!-\!\!\sqrt{b'}}\!, \\
& |\im \lm| \!<\! \frac{1\!\!-\!\!\sqrt{b'}}{\sqrt{b'}} |\re \lm \!+\! M_{\rm Tr}^2| \!-\! \frac{a_{W\!-U}(b'\!,M_{\rm Tr})}{\sqrt{b'}}
\!\bigg\} \hspace{-9mm} 
\end{aligned}
\end{equation}
is a subset of $\rho(T_n)$ for all $n \in \N$ and, for all $\lambda \in \mathcal R(b',M_{\rm Tr})$,
\begin{equation}\label{Tn.res.bound}
\|(T_n\!-\!\lambda)^{-1}\| \!\leq\! \frac{1}{(1\!\!-\!\!\sqrt{b'})|\re \lm\!+\!M_{\rm Tr}^2| \!-\! \sqrt{b'}|\im \lm| \!-\! a_{W\!-U}(b'\!,M_{\rm Tr}) 
}. 
\hspace{-8mm}
\end{equation}
\end{enumerate}
\end{prop}

We mention that, formally, for $M_{\rm Tr}=0$ the set $\mathcal R(b'\!,M_{\rm Tr})$ and the resolvent estimate in Proposition \ref{prop.comp.res} iii) coincide with the set 
$\mathcal R(b')$ and the resolvent estimate in Proposition \ref{prop.T.def} iv). 

Before we prove Proposition~\ref{prop.comp.res}, we establish analogues of Lemmas~\ref{lem.tau0.def},~\ref{lem.W.Delta} for the approximating operators~$T_n$ 
where we need to account for the boundary~terms; here we omit the dependence of the constants on $M_{\rm Tr}$.

\begin{lemma}\label{lem.T0n.g.norm}
Let Assumptions~\ref{ass.q.non-sec},~\ref{ass.lam.n} be satisfied and let $T_{0,n}$ be as in \eqref{T0n.def.non.sec}. 
Then, for every $\varepsilon_3 >0$, there exists $C_3(\varepsilon_3) \geq 0$, independent of $n$, such that, 
for every $f\in \dom(T_{0,n})$,
\begin{equation}\label{T0n.norm.est.2}
\begin{aligned}
\|T_{0,n} f \|_n^2 
&\geq
(1-\varepsilon_3) \left(\|\Delta_n^{\rm DR} f\|_n^2+ \|Q_0 f\|_n^2  \right)  
-C_3(\varepsilon_3)
\|f\|_n^2. 
\end{aligned}
\end{equation}
\end{lemma}
\begin{proof}
We adapt the proof of Lemma \ref{lem.tau0.def}.
Let $\eps_3>0$. 
For $f\in\dom(T_{0,n})$,
\begin{equation}\label{eq.T0n-lm}
\|T_{0,n} f\|_n^2 = 
\|\DnDR f\|_n^2 + \|Q_0 f\|_n^2  
+2 \re \langle -\DnDR f , Q_0 f \rangle_n.
\end{equation}
Before we estimate the individual terms, we prove two estimates that are used later on.
First, for arbitrary $\alpha, \beta >0$, by  the trace embedding \eqref{emb.est.n} with 
$\sqrt \eps=\frac{\beta}{M_{\rm Tr}}$ and Assumption~\ref{ass.lam.n.a}, we obtain
\begin{equation}
\begin{aligned}
\|\nabla f\|_n^2 & = 
\langle  -\DnDR f, f \rangle_n - \int_{\partial \Omega_n^{\rm R}} a_n |f|^2 \d  \sigma
\\
&\leq  \alpha\|\DnDR f\|_n^2 + \frac{1}{4 \alpha} \|f\|_n^2 +  
\beta \|\nabla f\|_n ^2 + \frac{M_{\rm Tr}^2}{\beta} \|f\|_n^2;
\end{aligned}
\end{equation}
hence, for \vspace{-2mm} $\beta:=1/2$,
\begin{equation}\label{nabla.n.est}
\begin{aligned}
\|\nabla f\|_n^2 & 
\leq  
2{\alpha} \|\DnDR f\|_n^2 
+ 
\left(
\frac{1}{2\alpha} + 4 M_{\rm Tr}^2
\right)
\|f\|_n^2.  
\end{aligned}
\end{equation}
Secondly, Assumption~\ref{ass.q.non-sec.reg} implies
\begin{equation}\label{eq.fnablaQ0}
 \| f \nabla Q_0 \|_n^2 \leq a_{\nabla}\|f\|_n^2 + b_{\nabla}\|Q_0 f\|_n^2.
\end{equation}
Now we estimate the last term on the right-hand side of~\eqref{eq.T0n-lm}. 

First we verify that $Q_0f \in \dom (s_{0,n})$. 
Since $Q_0 \in W^{1,\infty}_{\rm \loc}(\Rd,\C)$,  \cf~Assumption~\ref{ass.q.non-sec.reg}, 
and $f \in \dom(T_{0,n}) \subset \dom(s_{0,n}) \subset \Wotlam$, \cf~\eqref{T0n.def.non.sec} and \eqref{Dom.s0n.encl}, we have $Q_0f \in \Wotlam$. 
Moreover, using the definition of $\dom(s_{0,n})$, \cf~\eqref{s0n.def}, we find $\{f_k\}_k \subset \CcRdC$ such that
\begin{align}
& \dist\, (\supp f_k, \partial \Omega^{\rm D}) \ := \!\!\! \inf_{x_1 \in \supp f_k \atop x_2 \in \partial \Omega^{\rm D}} \|x_1-x_2\| >0, \quad k\ \in \N, \label{supp.dist}
\\
& \|f_k \restriction \Omega_n - f\|_{\Wotlam} \to 0, \quad k \to \infty \label{W12.lim}.
\end{align}
Since $Q_0 \in W^{1,\infty}_{\rm \loc}(\Rd,\C)$, we have $\{Q_0 f_k\}_k \subset \WotRd$. 
Let $J_\eps$, $\eps>0$, be the standard mollifier, \cf~\cite[Par.\ 2.28]{Adams-2003}, and let $k\in \N$. Due to $\eqref{supp.dist}$ and properties of mollifiers, \cf~\cite[Par.\ 2.28,\ Lem.\ 3.16]{Adams-2003}, there exists $\eps_k >0$ such that
\begin{equation}
\begin{aligned}
&(J_{\eps} * Q_0 f_k) \restriction \Omega_n \in \DD_n, \quad 0 < \eps < \eps_k, 
\\
& \| (J_{\eps} * Q_0 f_k) \restriction \Omega_n - Q_0 f_k \restriction \Omega_n\| _{\Wotlam} \to  0, \quad \eps \searrow 0.
\end{aligned}
\end{equation}
Thus $\{Q_0 f_k \restriction \Omega_n\}_k \subset \dom(s_{0,n})$, hence $Q_0 f \in \dom(s_{0,n})$ by \eqref{W12.lim} and $Q_0 \in W^{1,\infty}_{\rm \loc}(\Rd,\C)$.

Now we continue the estimates. For every $f \in \dom(T_{0,n})$,
\begin{equation}\label{mixed.est}
\begin{aligned}
 2\re \langle -\DnDR f, Q_0 f \rangle_n 
&= 
2\re \langle \nabla f, f \nabla Q_0 \!+\! Q_0 \nabla f \rangle_n  
+ 
2\re \int_{\partial \Omega_n^{\rm R}} \!\!\!a_n \overline{Q_0} |f|^2 \d  \sigma
\\
&\geq 
2\re \langle \nabla f, f \nabla Q_0 \rangle_n  
+
2\re \int_{\partial \Omega_n^{\rm R}} a_n\overline{Q_0} |f|^2 \d  \sigma
\\
&\geq 
-\gamma  \| f \nabla Q_0 \|_n^2 -\frac{1}{\gamma}\|\nabla f\|_n^2-\bigg|
2\int_{\partial \Omega_n^{\rm R}} a_n\overline{Q_0} |f|^2 \d  \sigma
\bigg| 
\end{aligned}
\end{equation}
for arbitrary $\gamma>0$.
We have ${Q_0} f^2 \in W^{1,1}{(\Omega_n,\C)}$, hence the trace embedding~\eqref{emb.est.n} with $p=1$, $\varepsilon=1$, and Assumption~\ref{ass.lam.n.a}  yield
\begin{equation}\label{bt.est}
\begin{aligned}
\bigg|
2\int_{\partial \Omega_n^{\rm R}} \!\!a_n \overline{Q_0} |f|^2 \d  \sigma
\bigg| 
& \leq
2\|a_n\|_{\infty} 
\int_{\partial \Omega_n^{\rm R}}
\big|{Q_0}f^2\big| \d  \sigma
\\
&
\leq 
2\|a_n\|_{\infty} K_n 
\int_{\Omega_n} 
\Big(
\left|
\nabla (Q_0f^2)
\right| 
+\big|Q_0f^2\big| 
\Big)
\d x
\\
&\leq
2 M_{\rm Tr} \Big(\|f\nabla Q_0\|_n \|f\|_n+\|Q_0 f\|_n(2 \|\nabla f\|_n+\|f\|_n) \Big)
\\
& \leq
\delta \| f \nabla Q_0 \|_n^2 + 
2 \delta \|Q_0 f\|_n^2  + \frac{2 M_{\rm Tr}^2}{\delta} (2 \| \nabla f\|_n^2 + \|f\|_n^2)
\end{aligned}
\end{equation}
where 
$\delta>0$ is arbitrary. 
Using \eqref{bt.est} and~\eqref{nabla.n.est},~\eqref{eq.fnablaQ0} to  estimate \eqref{mixed.est}, and by choosing $\gamma$, $\delta$ and then $\alpha$ sufficiently small, we arrive at
\begin{equation} \label{august1}
 2\re \langle -\DnDR f, Q_0 f \rangle_n
\geq 
\eps_3 \left(\|\Delta_n^{\rm DR} f\|_n^2+ \|Q_0 f\|_n^2  \right)  
-C_3(\varepsilon_3)
\|f\|_n^2
\end{equation}
for some $C_3(\eps_3)\geq 0$, so the claim follows from~\eqref{eq.T0n-lm}.
\end{proof}

\begin{lemma}\label{lem.T0n.W}
Let Assumptions~\ref{ass.q.non-sec},~\ref{ass.lam.n} be satisfied and let $T_{0,n}$ be as in \eqref{T0n.def.non.sec}. 
Then, for every $\varepsilon_4 \!>\!0$, there exists $C_4(\varepsilon_4) \!\geq\! 0$, independent of $n$, such that, for every \vspace{-1mm} $f \!\in\! \dom(T_{0,n})$,
\begin{equation}
\begin{aligned}
\|Wf\|_n^2 &\leq 
(b_W + \varepsilon_4) \|\DnDR f\|_n^2 + C_4 (\varepsilon_4) \|f\|_n^2.
\end{aligned}
\end{equation}
\end{lemma}
\begin{proof}
Let $\eta \in C_0^{\infty}(B_R(0),\R)$ be as in the proof of Lemma \ref{lem.W.Delta}. 
For  $f \in \dom(T_{0,n}) \subset \dom(s_{0,n}) \subset \Wotlam$, we have $\eta f\in W_0^{1,2}(B_R(0),\C)\subset\dom(s_{0,n})$ since $B_R(0)\!\subset\!\Omega_n$, 
\cf~\eqref{Dom.s0n.encl}. Let $\psi \in \DD_n$, \cf~\eqref{Dn.def}. 
Then $\eta\psi\in C_0^{\infty}(B_R(0), \C)$ and, integrating by parts, we can verify that
%
\begin{equation}
s_{0,n}(\eta f,\psi)
=
\langle -2 \nabla \eta . \nabla f - f \Delta \eta- \eta \DnDR f, \psi \rangle_n.
\end{equation}
Since $\DD_n$ is a core of $s_{0,n}$, 
 the first representation theorem~\cite[Thm.\ VI.2.1]{kato} implies that $\eta f\in\dom(-\DnDR)$ and  $\DnDR (\eta f)=f\Delta\eta + 2\nabla\eta . \nabla f + \eta\DnDR f$.
With the help of \eqref{nabla.n.est} instead of \eqref{nab.est}, the proof can be finished in the same way as the proof of Lemma \ref{lem.W.Delta}.
\end{proof}
\begin{proof}[Proof of Proposition {\rm\ref{prop.comp.res}}]
i)
Since $Q_0,$ $U \in L^{\infty}_{\rm loc}(\Rd,\C)$ by Assumption~\ref{ass.q.non-sec.reg}, and $W$ is 
$S_{0,n}$-bounded with  relative bound smaller than 1 by~Lemma \ref{lem.T0n.W},   
the operator $T_n$ is closed, \cf~\cite[Thm.\ IV.1.1]{kato}. 
Moreover, by \cite[Thm.\ IV.1.16]{kato}, it has compact resolvent since $S_{0,n}$ is $m$-sectorial with compact resolvent, \cf~Lemma \ref{lem.s0.def}. 

ii)
The equivalence of the norms 
can be proved by a straightforward adaptation of the arguments in the proof of Lemma \ref{lem.norm.eq}; note that Lemma \ref{lem.T0n.W} is used instead of Lemma \ref{lem.W.Delta}. 

iii)
By the trace embedding \eqref{emb.est.n} and Assumption~\ref{ass.lam.n.a}, we have
\begin{equation}
\begin{aligned}
\re \langle  T_{0,n} f, f \rangle_n 
&
\geq 
(1-M_{\rm Tr} \sqrt \eps ) \|\nabla f\|_n^2 - \frac{M_{\rm Tr}}{\sqrt \eps} \|f\|_n^2
= 
-M_{\rm Tr}^2 \|f\|_n^2, 
\end{aligned}
\end{equation}
where we have chosen $\sqrt \eps = 1 /M_{\rm Tr}$ in the last step. 
Hence the numerical range of $T_{0,n}$ lies in the half-plane $\{\lm\in\C: \re \lambda\geq -M_{\rm Tr}^2\}$ and so does the spectrum since $T_{0,n}$ has compact resolvent, \cf~\cite[Thm.~V.3.2]{kato}; 
moreover, $\|(T_{0,n} -\lm)^{-1}\|\leq |\re \lm+M_{\rm Tr}^2|^{-1}$ if $\re \lm <-M_{\rm Tr}^2$. 

Since $b_U<1$, the claims in \eqref{hyperbolaset2}--\eqref{Tn.res.bound} are now obtained by an argument based on \cite[Thm.\ IV.3.17]{kato}, similarly as in the proof of Proposition \ref{prop.T.def}; here Lemmas \ref{lem.T0n.g.norm}, \ref{lem.T0n.W} are used, and $a_{W\!-U}(b'\!,M_{\rm Tr})$ is the constant in the relative boundedness inequality of $W-U$ with respect to $T_{0,n}$, in analogy to \eqref{W-U}.
\end{proof}

\section{Convergence of $T_n$ to $T$}
\label{sec.conv}

In this section, we prove that the operators $T_n$ converge to $T$ in generalized norm resolvent sense, \cf~Theorem \ref{thm.norm.res.conv}.
The proof relies on two ingredients.

First, in Lemma \ref{lem.s.res.conv}, we show generalized \emph{strong} resolvent convergence of $T_n$ to~$T$. Here, for semi-angle $\theta \geq \pi/2$, we employ the so-called common core property of approximations, \cf~\cite[Thm.\ 1]{Weidmann-1997-34}.  
For semi-angle $\theta < \pi/2$, where it is not even guaranteed that $\CcRdC \subset \dom(T)$, we use form techniques inspired by the approach in \cite[Prop.\ 5.4]{Krejcirik-2010-94} for a selfadjoint Laplacian in twisted tubes.  

Secondly, in Lemma \ref{lem.discr.comp}, we establish discrete compactness, \cf~\cite[Def.\ 3.1.(k)]{stummel1}, of the sequence of \vspace{-1mm} embeddings 
\begin{equation}\label{dis.comp}
\left(\dom(T_n),\left(\|T_n\cdot\|_n^2 + \|\cdot\|_n^2\right)^\frac 12 \right) \hookrightarrow L^2(\Omega_n,\C), \quad n \in \N.
\end{equation}


\begin{theorem}\label{thm.norm.res.conv}
Let Assumption~\ref{ass.lam.n} be satisfied and assume that
\begin{enumerate}[\upshape I.]
\item in the case of semi-angle $\theta < \pi/2$, Assumption~\ref{ass.q.sec} holds and $T$, $T_n,$ $n\in \N$, are the operators defined in Propositions~{\rm\ref{prop.T.sec.def},~\ref{prop.Tn.def.sec}}, respectively;
\item in the case of semi-angle $\theta \geq \pi/2$, Assumption~\ref{ass.q.non-sec} holds with $\bU\!<\!1$ and  $T$,~$T_n,$ $n\!\in\!\N$, are the operators defined in Propositions~{\rm\ref{prop.T.def},~\ref{prop.comp.res}}, respectively. 
\end{enumerate}	
Then, for every $\lm\!\in\!\rho(T)$, 
there exists $n_\lm \!\in\! \N$ such that, for all $n\geq n_{\lm}$, 
$\lm \!\in\! \rho(T_n)$~and
\begin{equation}\label{norm.conv}
\left\|
(T_n-\lm)^{-1}\chi_{\Omega_n}- (T-\lm )^{-1}
\right\| 
\to 0, \quad n \to \infty.
\end{equation}
\end{theorem}
%

\begin{rem}\label{rem.loc.unif}
The generalized norm resolvent convergence in \eqref{norm.conv} is even locally uniform, \ie~for all $\lm\in \rho(T)$, there exist $r_{\lm}>0$ and $n_\lm \in \N$ such that $B_{r_\lm}(\lm) \subset \bigcap_{n \geq n_\lm}\rho(T_n) \cap \rho(T)$ and the convergence is uniform in $B_{r_\lm}(\lm)$.

To see the latter, let $n_{\lm}\in\N$ be so large that 
$\|(T_n-\lm)^{-1}\chi_{\Omega_n} - (T-\lm)^{-1}\|\leq 2 \|(T-\lm)^{-1}\|$ for all $n\geq n_\lm$. If we choose $r_\lm:=\|(T-\lm)^{-1}\|^{-1}/4$, then a Neumann series argument yields that, for every $\mu\in B_{r_{\lm}}(\lm)$, the resolvents $(T-\mu)^{-1}$, $(T_n-\mu)^{-1}$, $n\geq n_{\lm}$, exist and are uniformly bounded (in $n$ and $\mu$). Then the uniform convergence of the resolvents follows from \eqref{eq.normconv.lm} below with $\lm_0, \lm$ replaced by $\lm, \mu$.
\end{rem}	

To prove Theorem \ref{thm.norm.res.conv}, we first show the two lemmas described above. 

\begin{lemma}\label{lem.s.res.conv}
Let the assumptions of Theorem~{\upshape\ref{thm.norm.res.conv}} be satisfied. 
Then there exists $\gamma >0$ such that $(-\infty,-\gamma) \subset \bigcap_n \rho(T_n) \cap \rho(T) \neq \emptyset$
and for all $\lm_0 \in (-\infty,-\gamma)$ and for every $f \in \LRd$,
\begin{equation}
\left\|
\left(
(T_n-\lm_0)^{-1} \chi_{\Omega_n}-(T-\lm_0)^{-1}
\right)
f
\right\| \to 0, \quad n \to \infty.
\end{equation}
\end{lemma}


\begin{proof}
	I. \emph{Semi-angle $\theta < \pi/2$}: 
	Since $T$, $T_n$ are $m$-sectorial and the numerical ranges of $T_n$ satisfy \eqref{Tn.sec.S},
	the set $\mathcal S(M_{\rm Tr}) \cup \overline{\Num(T)}$ contains all spectra and is itself contained in some right half-plane.
	In particular, $(-\infty,\delta_1) \subset \bigcap_n \rho(T_n) \cap \rho(T)$ for some $\delta_1\in\R$.

	Now let $\lm_0 \notin \mathcal S(M_{\rm Tr}) \cup \overline{\Num(T)}$ be arbitrary; 
	without loss of generality, we assume that $\lm_0=0$, \ie~$0 \notin \mathcal S(M_{\rm Tr}) \cup \overline{\Num(T)}$; otherwise we replace $T$, $T_n$ by $T-\lm_0$, $T_n-\lm_0$, respectively.
	Then there exists $d_0>0$ such that $\dist(0, \overline{\Num(T_n)}) \geq d_0$ and  $\dist(0, \overline{\Num(T)}) \geq d_0$.

	We prove the claim by contradiction. Suppose that $T_n^{-1} \chi_{\Omega_n} f \rightarrow T^{-1} f$ in $L^2(\R^d,\C)$ does not hold. Then there exist  $\delta >0$ and an infinite subset $I \subset \N$ such that $\|T_n^{-1} \chi_{\Omega_n} f-T^{-1} f\|\geq \delta$ for all $n \in I$. We will show that $\{T_n^{-1} \chi_{\Omega_n} f\}_{n \in I}$ contains a subsequence converging to $T^{-1} f$, a contradiction.

		To simplify the notation, we set $f_n := \chi_{\Omega_n}f$, $\phi_n :=T_n^{-1} f_n$. Note that
	\begin{equation}
	\|\phi_n\| = \|T_n^{-1} f_n \| \leq \|T_n^{-1}\| \|f\| \leq \frac{ \|f\| }{d_0}
	\end{equation}
	and rewrite $T_n\phi_n=f_n $ in terms of forms,
	\begin{equation}\label{eq.tn.phin.phi}
	\forall \, \phi\in \dom(t_n):\quad t_n(\phi_n,\varphi)= \langle f_n, \phi \rangle_n.
	\end{equation}
	If we insert $\phi\!=\!\phi_n$ and take real parts in the equation in \eqref{eq.tn.phin.phi}, we obtain 
	\begin{equation}\label{retn.est.1}
	\|\nabla \phi_n \|_n^2 + \re \int_{\partial \Omega_n^{\rm R}} a_n |\phi_n|^2 \d  \sigma   
	+ \re  q_0[\phi_n] + \re  w[\phi_n] 
	= \re \langle f_n, \phi_n \rangle_n.
	\end{equation}
	Taking absolute values on both sides and using the relative $\|\nabla \cdot \|^2$-bounds of the boundary term and of $w$, 
	\cf~the trace embedding~\eqref{emb.est.n}, Assumption~\ref{ass.lam.n.a}, and~\eqref{w.est}, we arrive at
	\begin{equation}
	\begin{aligned}
	&(1 - 
	\sqrt \eps M_{\rm Tr} - b_w(1+\varepsilon)  ) \|\nabla \phi_n\|_n^2 
	+ \re q_0[\phi_n]
	\\
	&
	\leq 
	\|f\| \|\phi_n\| + \left( 
	\frac{M_{\rm Tr}}{\sqrt \eps}
	+ C_{w, \varepsilon}\right) \|\phi_n\|^2 
	\\
	& \leq 
	\left(
	\frac{1}{d_0} + \left(
	\frac{M_{\rm Tr}}{\sqrt \eps}
	+ C_{w, \varepsilon} \right) \frac{1}{d_0^2}
	\right)
	\|f\|^2
	=: K_1(\eps),
	\end{aligned}
	\end{equation}
	where $K_1(\eps)>0$ is independent of $n$.
	If we choose $\eps>0$ sufficiently small, we find that there exists $K_2\geq 0$, independent of $n$, such that 
	\begin{equation}\label{phin.est}
	\|\nabla\phi_n\|_{n} \leq K_2, \qquad \re q_0[\phi_n] \leq K_2.
	\end{equation}

	Let $\zeta_n \in C_0^{\infty}(B_{r_n+1}(0),\R) \subset C_0^{\infty}(\Omega_n,\R)$ be such that 
	\begin{equation}\label{omegan.def}
	\begin{aligned}
	& 
	0\leq \zeta_n \leq 1, \quad \zeta_n \restriction B_{r_n}(0) =1, 
	\quad 
	\| \zeta_n \|_{\infty} + \|\nabla \zeta_n \|_{\infty} \leq M_1,
	\end{aligned}
	\end{equation}
	where $r_n$ are the radii used in Assumption~\ref{ass.lam.n} and $M_1 >    0$ is independent of $n$.
	We define $\psi_n := \zeta_n \phi_n$. 
	Note that $\psi_n$ coincides with $\phi_n$ on $B_{r_n}(0)$ and its support is contained in $\Omega_n$. It is easy to see that $\psi_n \in \WotRd$ and that,
	by \eqref{phin.est} and \eqref{omegan.def}, there exists $K\geq 0$ such that 
	\begin{equation}\label{psin.est}
	\|\psi_n \| \leq \|\phi_n\|\leq\frac{\|f\|}{d_0}, \quad 
	\|\nabla\psi_n\| \leq K, \quad \re q_0[\psi_n] \leq K.
	\end{equation}
	Hence $\{\psi_n \}_{n \in I}$ is a bounded sequence in 
	$
	\mathcal H_1:=
	\WotRd\cap\dom(q_0)$
	equipped with the norm $(\|\nabla\cdot  \|^2 + \re q_0[\cdot]  + \|\cdot\|^2 )^{1/2}$.
	Therefore there exists a subsequence $\{n_k\}_k \subset I$ such that $\{\psi_{n_k}\}_k$ 
	 converges weakly in $\mathcal H_1$ to some $\psi \in \mathcal H_1$. Since $\mathcal H_1$ is compactly embedded in $\LRd$, \cf~Rellich's criterion \cite[Thm.\ XIII.65]{reedsimon}, the sequence $\{\psi_{n_k}\}_k$ converges to $\psi$ in $\LRd$. 
	
	Now we prove that  $\{\phi_{n_k}\}_k$ converges to $\psi$.
	The properties of $\zeta_n$, \cf\,\eqref{omegan.def}, imply
	\begin{equation}\label{phin.psin.conv.sec.1}
	\| \phi_n - \psi_n\|^2 \leq \int_{|x|\geq r_n} |\phi_n|^2 \d  x.
	\end{equation}
	Using $\re Q_0 >0$ and $\re Q_0(x) \to  \infty $ as $|x| \to  \infty$, we obtain
	\begin{equation}\label{phin.psin.conv.sec}
	K_2 \geq \re  q_0[\phi_n] \geq \int_{|x| \geq r_n} \re  Q_0 |\phi_n|^2 \d  x \geq \left( 
	\mathop{\rm ess \, inf}\limits_{|x| \geq r_n} \re  Q_0 
	\right)
	\int_{|x| \geq r_n} |\phi_n|^2 \d  x;
	\end{equation}
	hence, since $r_n\to \infty$, 
	\begin{equation}\label{eq.phin.psin.diff}
	\| \phi_{n} - \psi_{n}\| \to 0, \quad n \to  \infty,
	\end{equation}
	and thus $\{\phi_{n_k}\}_k$ converges to $\psi$.
	
	Finally, to obtain the contradiction, we prove that $\psi=T^{-1}f$, \ie 
	\begin{equation}\label{eq.F.and.tildeT.psi}
	\psi \in \dom (T), \quad f =T \psi.
	\end{equation} 
	To this end, we show that
	\begin{equation}\label{eq.t.psi.phi}
	\forall\, \varphi \in  \CcRdC:\quad  t_{n_k}(\phi_{n_k},\varphi) \to t(\psi,\varphi), \quad  k \to \infty;
	\end{equation}
	then $\langle f,\varphi\rangle=\lim_{k\to \infty}\langle f_{n_k},\varphi\rangle_{n_k}=t(\psi,\varphi)$ by~\eqref{eq.tn.phin.phi}, 
	and hence~\eqref{eq.F.and.tildeT.psi} follows from the representation theorem \cite[Thm.\ VI.2.1]{kato} and the fact that $\CcRdC$ is a core of $t$, \cf~Proposition \ref{prop.T.sec.def}. 
	
	Let $\varphi \in \CcRdC$. 
	There exists $n(\varphi) \in\N$ such that, for all $n \geq  n(\varphi)$, we have $\supp \varphi \subset B_{r_n}(0) \subset \Omega_{n}$
	and  therefore $\zeta_{n} \varphi = \varphi$, $\zeta_{n} \nabla \varphi = \nabla \varphi$ by \eqref{omegan.def} and
	\begin{equation}
	\begin{aligned}
	\langle \nabla \phi_{n}, \nabla \varphi \rangle_{n} &= \langle \nabla \psi_{n}, \nabla \varphi \rangle, 
	& 
	\int_{\Omega_{n}}   Q_0 \, \phi_{n} \overline \varphi \, \d  x  &= \int_{\Rd}  Q_0\,  \psi_{n} \overline \varphi \, \d  x,
	\\
	\int_{\partial \Omega_{n}^{\rm R}} a_{n} \, \phi_{n} \overline \varphi  \, \d  \sigma &=0, 
	& \langle f_{n}, \varphi \rangle_{n} &= \langle f, \varphi \rangle . 
	\end{aligned}
	\end{equation}
	Since $w$ is relatively bounded with respect to $\|\nabla \cdot \|^2$, it is a bounded form on~$\mathcal H_1$. Therefore, $w(\psi_{n_k},g) \to w(\psi,g)$ for any $g \in \mathcal H_1$. 
	Recall that the function $\zeta$ in Assumption~\ref{ass.q.sec} satisfies ${\rm supp}\,\zeta\subset B_R(0)$.
	Hence, since $\psi_{n_k}$ and $\phi_{n_k}$ coincide on $B_{r_n}(0)\supset B_R(0)$, we have $(\phi_{n_k}-\psi_{n_k})\sqrt{\zeta}=0$.
	Thus the splitting property~\eqref{eq.splitting.w} of~$w$ in Assumption~\ref{ass.q.sec} and the polarization identity~\cite[Eq.\ VI.(1.1)]{kato} imply 
	\begin{equation}\label{w.phi.psi}
	w(\phi_{n_k},\varphi) - w(\psi_{n_k},\varphi) = w_2(\sqrt{1 -\zeta}\, (\phi_{n_k} - \psi_{n_k}),\sqrt{1 -\zeta}\, \varphi).
	\end{equation}
	Because the form $w_2$ is bounded, \cf~Assumption~\ref{ass.q.sec.w.2},
	by~\eqref{eq.phin.psin.diff}, we have $w(\phi_{n_k},\varphi) - w(\psi_{n_k},\varphi) \to 0$.
	Altogether, we obtain 
			\begin{equation}
		\lim_{k\to \infty}t_{n_k}(\phi_{n_k},\varphi) 
		= \lim_{k \to \infty}t(\psi_{n_k},\varphi)+(w(\phi_{n_k},\varphi) - w(\psi_{n_k},\varphi)) 
		= \lim_{k \to \infty}t(\psi_{n_k},\varphi),
		\end{equation}
	and the latter equals $t(\psi,\varphi)$ since $\psi$ is the weak limit of $\{\psi_{n_k}\}_{k}$ in $\mathcal H_1$.
	
	%
	II. \emph{Semi-angle $\theta \geq \pi/2$}:
	The intersection of the resolvent sets is non-empty since there exists $\delta_2\in\R$ such that 
	$(-\infty,\delta_2) \subset {\mathcal R}(b') \cap{\mathcal R}(b',M_{\rm Tr}) \neq \emptyset$, \cf~\eqref{hyperbolaset}, \eqref{hyperbolaset2}. 
	
	Let $\lm_0 \in  {\mathcal R(b')} \cap {\mathcal R(b',M_{\rm Tr})} \supset (-\infty,\delta_2)$ be arbitrary.
	We can follow \cite[Thm.\ 1]{Weidmann-1997-34} which can be straightforwardly generalized to the non-selfadjoint case if a uniform bound on $\|(T_n-\lm_0)^{-1}\|$ is available; such a bound  is given by~\eqref{Tn.res.bound}.
	In order to check the assumptions of  \cite[Thm.\ 1]{Weidmann-1997-34}, recall that $ \dom(\Tm)=\CcRdC$ is a  core of $T$, \cf~Proposition \ref{prop.T.def}. 
	For every $f\in \dom(\Tm)$ there exists $n_0(f) \in\N$ such that  $\supp f \subset \Omega_n$ for all $n \geq n_0(f)$.
	Then, for $n\geq n_0(f)$, we have $f_n:=\chi_{\Omega_n}f \in \dom(T_n)$ and $Tf = T_n f_n$. 
	Notice that,  since $\dom(T_n)$ is not described explicitly,  we use the first representation theorem \cite[Thm.\ VI.2.1]{kato} to verify the latter.
\end{proof}

\begin{lemma}\label{lem.discr.comp}
Let the assumptions of Theorem~{\upshape\ref{thm.norm.res.conv}} be satisfied and let $I\subset \N$ be an arbitrary infinite subset. 
Then every sequence of elements  $\phi_n \in \dom(T_n)$, $n \in I$, such that $\left\{\|T_n \phi_n\|_n^2+\|\phi_n\|_n^2\right\}_n$ is bounded has a convergent subsequence in $\LRd$.
\end{lemma}
\begin{proof}
Let $\phi_n \in \dom(T_n)$, $n \in I$, and $M\geq 0$ be such that
\begin{equation}\label{phin.Tng}
\|T_n \phi_n \|^2_{n} + \|\phi_n\|^2_{n} \leq  M.
\end{equation}

I. \emph{Semi-angle $\theta < \pi/2$}: 
The bound \eqref{phin.Tng} implies that 
$|t_n[\phi_n]| \leq M/2$. 
Define $f_n:=T_n\phi_n$ and note that $\|f_n\|_n^2\leq M$.
Now we proceed analogously as in~\eqref{eq.tn.phin.phi}--\eqref{eq.phin.psin.diff} to find a convergent subsequence  of \vspace{1mm} $\{\phi_n\}_n$.

II. \emph{Semi-angle $\theta \geq \pi/2$}:
Let $\{\zeta_n\}_n$ be the family of functions defined in \eqref{omegan.def}. We set $\psi_n := \zeta_n \phi_n$. Using the inequality \eqref{nabla.n.est} and the equivalence of norms in \eqref{Tn.graph.norm}, we obtain the existence of $\widetilde M\geq 0$ such that
$\|\nabla \psi_n \| + \|Q_0 \psi_n\| \leq \widetilde M$. Hence, it follows from Rellich's criterion \cite[Thm.\ XIII.65]{reedsimon} that $\{\psi_n\}_n$ is contained in a compact subset of $L^2(\R^d,\C)$, thus it has a convergent subsequence $\{\psi_{n_k}\}_k$. Finally, using an analogous argument as \eqref{phin.psin.conv.sec.1}--\eqref{phin.psin.conv.sec} with \eqref{phin.psin.conv.sec} replaced by
\begin{equation}\label{phi.psi.cnv.2}
\begin{aligned}
\widetilde M^2  & 
\geq 
\int_{|x|\geq r_n} |Q_0|^2 |\phi_n|^2 \d  x
\geq 
\left(\mathop{\rm ess \, inf}\limits_{|x|\geq r_n} |Q_0(x)|^2\right) \int_{|x|\geq r_n} |\phi_n|^2 \d  x,
\end{aligned}
\end{equation}
one may show that $\{\phi_{n_k}\}_k$ has the same limit as $\{\psi_{n_k}\}_k$.
\end{proof}

\begin{proof}[Proof of Theorem \upshape \ref{thm.norm.res.conv}]
By Lemma~\ref{lem.s.res.conv}, there exists $\gamma>0$ such that we have generalized strong resolvent convergence at all $\lambda_0 \in (-\infty,-\gamma)$.
One may verify that Assumptions~\ref{ass.q.sec},~\ref{ass.q.non-sec},~\ref{ass.lam.n} remain valid under complex conjugation of $q$, $Q$, $a_n$, $W$, 
and with $w$ replaced by the adjoint form $w^*$. 
Hence Propositions \ref{prop.T.sec.def}, \ref{prop.T.def}, \ref{prop.Tn.def.sec}, \ref{prop.comp.res} define closed operators $\widehat T, \widehat T_n$, and the latter coincide with the adjoints $T^*\!$, $T_n^*$, \cf~\cite[Thm.\ VI.2.5]{kato} (for semi-angle $\theta \!<\! \pi/2$) and \cite[Thm.\ VII.2.5,\ 2.6,\ Cor.~2.7]{edmundsevans}, \cite[Cor.\ 1]{Hess-1970-45} (for semi-angle $\theta \!\geq\! \pi/2$). 
Moreover, Lemma \ref{lem.s.res.conv} implies that $(T_n^*-\lm_0)^{-1}\chi_{\Omega_n}$ converges strongly to $(T^*-\lm_0 )^{-1}$.
Then \cite[Thm.\ 3.4]{anselonepalmer} yields that the resolvents converge even in norm provided we verify that $\{(T_n-\lm_0)^{-1} \chi_{\Omega_n}:\,n\in\N\}$ and $\{(T_n^*-\lm_0)^{-1} \chi_{\Omega_n}:\,n\in\N\}$ are collectively compact sets. The claim for the former set follows from~\cite[Prop.\ 2.1]{anselonepalmer} since every $(T_n-\lm_0)^{-1}$ is compact and the sequence of embeddings \eqref{dis.comp} is discretely compact, \cf~Lemma~\ref{lem.discr.comp}; the reasoning  for the set of adjoint operators is~analogous.

Now let $\lm \in \rho(T)$ with $\lm\neq \lm_0$ be arbitrary. By the spectral mapping theorem we have $\mu:=(\lm-\lm_0)^{-1} \in \rho\big((T_n-\lm_0)^{-1}\big)$.
By~\cite[Thm.\ IV.2.25]{kato}, together with \eqref{norm.conv} for $\lm=\lm_0$, there exists $n_\lm\in\N$ such that, for all $n \geq n_\lm$,
we have $\mu \in \rho\big((T_n-\lm_0)^{-1}\big)$ and so $\lm\!\in\!\rho(T_n)$.
A straightforward application of the first resolvent identity yields
\begin{equation}\label{eq.normconv.lm}
\begin{aligned}
&\left(
(T_n-\lm)^{-1}\chi_{\Omega_n} - (T-\lm)^{-1} 
\right) 
S_n
\\
& 
= 
\left(
I + (\lm-\lm_0) (T-\lm)^{-1}
\right)
\left( 
(T_n-\lm_0)^{-1}\chi_{\Omega_n} - (T-\lm_0)^{-1} 
\right)
\end{aligned}  
\end{equation}
with $S_n=I - (\lm - \lm_0) (T_n-\lm_0)^{-1}\chi_{\Omega_n}$.
Since $S:=\lim_{n \to \infty} S_n$ has a bounded inverse, the operator $S_n$ is boundedly invertible for all sufficiently large $n$
and $\|S_n^{-1}\|$ is uniformly bounded, \cf~\cite[Thm.\ IV.1.16]{kato}.
Now the convergence \eqref{norm.conv} follows from the convergence at $\lm_0$ and~\eqref{eq.normconv.lm}.
\end{proof}

\section{Convergence of spectra and pseudospectra}
\label{sec.sp.conv}

In the following theorem, we prove that $\{T_n\}_n$ is a spectrally exact approximation of $T$, i.e.\ \emph{all} eigenvalues of $T$ are approximated and \emph{no} spectral pollution occurs.
In addition, we prove norm convergence of the spectral projections.

\begin{theorem}\label{thm.conv.spectrum}
Let Assumption~\ref{ass.lam.n} be satisfied and assume that
\begin{enumerate}[\upshape I.]
	\item in the case of semi-angle $\theta < \pi/2$, Assumption~\ref{ass.q.sec} holds and $T$, $T_n,$ $n\in \N$, are the operators defined in Propositions~{\rm\ref{prop.T.sec.def},~\ref{prop.Tn.def.sec}}, respectively;
	\item in the case of semi-angle $\theta \geq \pi/2$, Assumption~\ref{ass.q.non-sec} holds with $\bU<1$ and $T$,~$T_n,$ $n\!\in\! \N$, are the operators defined in Propositions~{\rm\ref{prop.T.def},~\ref{prop.comp.res}}, respectively. 
\end{enumerate}	
 
Then the following hold:
\begin{enumerate}[\upshape i)]
\item \emph{Spectral inclusion with preservation of algebraic multiplicity:}  
If $\lm \in \C$ is an eigenvalue of $T$ of algebraic multiplicity $m$,
then, for $n$ large enough,  $T_n$ has exactly $m$ eigenvalues {\rm(}repeated according to their algebraic multiplicities{\rm)} in a neighbourhood of $\lm$ which converge to $\lm$ as $n\to \infty$ and the corresponding spectral projections converge in norm.
\item \emph{No spectral pollution:} 
If $\{\lm_n\}_n \subset \C$ is a sequence of eigenvalues $\lm_n \in \sigma(T_n)$, $n \in \N$, such that there exists $\lm \in\C$ and a subsequence $\{\lm_{n_k}\}_k \subset \{\lm_n\}_n$ with $\lm_{n_k} \to \lm$ as $k \to \infty$, then $\lm$ is an eigenvalue of \,$T$.
\end{enumerate}
\end{theorem}

\begin{proof}
i)
Since $T$ has compact resolvent, every eigenvalue $\lm\in\sigma(T)$ is isolated, \ie~there exists $\epsilon>0$ such that $\overline{B_{\epsilon}(\lm)}\backslash\{\lm\}\subset\rho(T)$.
By Theorem~\ref{thm.norm.res.conv}, we have  $\|(T-z)^{-1}-(T_n-z)^{-1}\chi_{\Omega_n}\|\to 0$ for every $z\in\partial B_{\epsilon}(\lm)$
and the convergence is uniform in $\partial B_{\epsilon}(\lm)$ since $\partial B_{\epsilon}(\lm)$ is compact,  
\cf~Remark~\ref{rem.loc.unif}. 
Hence the spectral projections 
\begin{equation}
 \begin{aligned}
  E:=-\frac{1}{2\pi\I}\int_{\partial B_{\epsilon}(\lm)} (T-z)^{-1} \d z, \qquad 
  E_n:=-\frac{1}{2\pi\I}\int_{\partial B_{\epsilon}(\lm)} (T_n-z)^{-1} \d z
 \end{aligned}
\end{equation}
satisfy $\|E-E_n\chi_{\Omega_n}\|\to 0$ and therefore there exists $n_0 \in\N$ such that, for all $n\geq n_0$, 
\vspace{1mm} $\text{rank}\, E_n=\text{rank}\, (E_n\chi_{\Omega_n})=\text{rank}\, E=m.$

ii)
Spectral pollution cannot occur since it would contradict the locally uniform convergence of the resolvents, \cf~Remark~\ref{rem.loc.unif}. 
%
\end{proof}

Based on \cite[Thm.\ 2]{Osborn-1975-29}, we prove an estimate on the convergence rate of the arithmetic mean of the eigenvalues in terms of the decay rate of the functions in the corresponding algebraic eigenspace. 
An analogous result can be obtained, using \cite[Thm.\ 6]{Osborn-1975-29}, for the individual eigenvalues instead of their arithmetic mean;
if, however, $\lm$ is not semi-simple, \ie~$\lm$ has ascent greater than one, then the convergence of the individual eigenvalues is slower than the one of their arithmetic~mean.

\begin{theorem}\label{thm.conv.rate}
Let Assumption \ref{ass.lam.n} be satisfied and assume that
\begin{enumerate}[\upshape I.]
	\item in the case of semi-angle $\theta \!<\! \pi/2$, Assumption~\ref{ass.q.sec} holds and 
	for every $\varphi\!\in\! \CcRdC$ and $f,g\in\dom(w)$, 
	\begin{equation}\label{ass.w.loc}
	w(\varphi f,g)=w(f,\overline\varphi g),
	\end{equation}
	and $T$, $T_n$, $n\!\in\! \N$, are the operators defined in Propositions~{\rm\ref{prop.T.sec.def}, \ref{prop.Tn.def.sec}}, respectively;
	\item in the case of semi-angle $\theta \geq \pi/2$, Assumption~\ref{ass.q.non-sec} holds with $\bU<1$ and $T$,~$T_n,$ $n\!\in\! \N$, are the operators defined in Propositions~{\rm\ref{prop.T.def},~\ref{prop.comp.res}}, respectively. 
\end{enumerate}	
Let $\lm\in\sigma(T)$ be an eigenvalue of algebraic multiplicity $m$,  let $\Lla$ be the corresponding algebraic eigenspace and let $\{\lm_{1;n},\dots,\lm_{m;n}\} \subset \sigma(T_n)$ be the eigenvalues of $T_n$ converging to $\lm$ as $n \to \infty$, \cf~Theorem~{\rm\ref{thm.conv.spectrum}}.
Then there exists $C\geq 0$, independent of $n$, such that
\begin{equation}\label{conv.rate.of.lambda}
\begin{aligned}
\bigg|
\lm-\frac 1m \sum_{j=1}^m\lm_{j;n}
\bigg|
\leq C\, \max_{\substack{\phi\in \Lla \\ \|\phi\|=1}}\big\|\phi\restriction \R^d\backslash B_{r_n}(0)\big\|
\end{aligned}
\end{equation}
where $r_n$ are the radii used in Assumption~{\rm \ref{ass.lam.n.ex}}.
\end{theorem}

\begin{rem}\label{rem.dec}
The decay rate of $\phi\in \Lla$ can be further estimated as 
\begin{equation}\label{decay.rate.of.phi.Q_0}
\max_{\substack{\phi\in \Lla \\ \|\phi\|=1}}
\big\|\phi\restriction \R^d\backslash B_{r_n}(0)\big\|
\leq \frac{D}{\mathop{\rm ess \, inf}\limits_{|x|\geq r_n}|Q_0(x)|^{\iota}},
\end{equation}
where $D \geq 0$ is independent of $n$ and $\iota =1/2$  if $\theta < \pi/2$  and  $\iota=1$  if $\theta \geq \pi/2$, respectively, \cf~\eqref{phin.psin.conv.sec} and \eqref{phi.psi.cnv.2}.
However, the decay rate of $\phi\in \Lla$ may be much faster than the growth of $|Q_0|$, 
\eg~exponential, \cf~\cite{Agmon-1982-29, Simon-1982-7,Cappiello-2010-111} or \cf~\cite{Sibuya-1975} for complex polynomial potentials.
\end{rem}


\begin{proof}[Proof of Theorem \upshape \ref{thm.conv.rate}]
Let $\mu\in\rho(T)$.
Theorem \ref{thm.norm.res.conv} implies that  $\mu\in\rho(T_n)$ for all sufficiently large~$n$ and $\|(T_n-\mu)^{-1}\chi_{\Omega_n}- (T-\mu)^{-1}\|\to 0$.
The spectral mapping theorem yields $\nu\!:=\!(\lm-\mu)^{-1}\!\in\!\sigma((T-\mu)^{-1})$ and
the eigenvalues $\nu_{j;n}\!:=\!(\lm_{j;n}-\mu)^{-1}\!\in\!\sigma((T_n-\mu)^{-1})\subset\sigma((T_n-\mu)^{-1}\chi_{\Omega_n})$ satisfy $\nu_{j;n} \to \nu$ as $n \to \infty$. 
Now the identity $|\lm-\lm_{j;n}|=|\nu  \nu_{j;n}|^{-1} |\nu-\nu_{j;n}|$ implies that it suffices to study the convergence rate for $\nu_{j;n}$.

By \cite[Thm.\ 2]{Osborn-1975-29}, there exists $C_1\geq 0$, independent of $n$, such that, 
\begin{equation}\label{estimate.osborn}
\bigg|
\nu-\frac 1m \sum_{j=1}^m\nu_{j;n}
\bigg|
\leq C_1 
\left\|
\left((T-\mu)^{-1}-(T_n-\mu)^{-1}\chi_{\Omega_n}\right) \restriction \Lla
\right\|.
\end{equation}
Below we show that there exists $\widetilde C \!>\!0$, independent of $n$, such that, for every $\phi\! \in\! \Lla$, 	
\begin{equation}\label{rate.phi.psi}
\begin{aligned}
\left\|\left((T-\mu)^{-1}-(T_n-\mu)^{-1}\chi_{\Omega_n}\right)\phi\right\|
\leq \widetilde C
\big(&
\big\|\phi\restriction \R^d\backslash B_{r_n}(0)\big\|
\\&
+ \big\|((T-\mu)^{-1} \phi)\restriction \R^d\backslash B_{r_n}(0)\big\|  
\big).
\end{aligned}
\end{equation}
Since $\Lla$ is an invariant subspace of $T$, we have 
\begin{equation}
\begin{aligned}
\!\!\max_{\substack{\phi\in \Lla \\ \|\phi\|=1}} \!
\big\|((T\!-\!\mu)^{-1} \phi)\restriction \R^d\backslash B_{r_n}(0)\big\|  
&\!\leq\!
\|(T-\mu)^{-1}\| \!\!
\max_{\substack{\phi\in \Lla \\ \|\phi\|=1}} \!\!\!\!
\frac{\big\|((T\!-\!\mu)^{-1} \phi)\restriction \R^d\backslash B_{r_n}(0)\big\|  }
{\big\|(T\!-\!\mu)^{-1} \phi\big\|} 
\\
& \!\leq\!
\|(T\!-\!\mu)^{-1}\| \!
\max_{\substack{\psi\in \Lla \\ \|\psi\|=1}} \!
\|\psi\restriction \R^d\backslash B_{r_n}(0)\|,
\end{aligned}
\end{equation}
hence the estimate \eqref{conv.rate.of.lambda} in the claim follows.

To prove \eqref{rate.phi.psi}, let $\{\zeta_n\}_n \subset \CcRdR$ be  such that, with $\widetilde\zeta_n:=1-\zeta_n$, 
\begin{equation}\label{omegan.and.tilde.omegan}
 \begin{aligned}
  & 0\leq \zeta_n \leq 1, \ \ \zeta_n \restriction B_{r_n}(0) =1, \quad \supp \zeta_n \subset B_{r_n+1}(0),  
\\
& \|\nabla\zeta_n \|_{\infty}+\|\Delta\zeta_n \|_{\infty} = \|\nabla\widetilde\zeta_n \|_{\infty}+\|\Delta\widetilde\zeta_n \|_{\infty} \leq C_2
\end{aligned}
\end{equation}
where $C_2>0$ is independent of $n$.
Let $\phi \in \Lla$ and set $\psi:=(T-\mu)^{-1}\phi$. 
First we adapt the approach of \cite{Friedlander-2009-170} or \cite[Prop.\ 5.3]{Krejcirik-2012-24} based on 
\begin{equation}\label{estimate.sect.3}
\|g\| =\sup_{f\neq 0}\frac{|\langle g, f \rangle |}{\|f\|}.
\end{equation}
Let $f\in L^2(\R^d,\C)$, $f \neq 0$. Then, with $\widetilde \chi_{\Omega_n}:=1-\chi_{\Omega_n}$, we write
\begin{equation}\label{eq.est.res.diff.1}
\begin{aligned}
&\langle  \left((T-\mu)^{-1}-(T_n-\mu)^{-1}\chi_{\Omega_n}\right)\phi, f \rangle 
\\ 
&
= \langle (T-\mu)^{-1}\phi, \chi_{\Omega_n} f \rangle 
+ \langle (T-\mu)^{-1}\phi, \widetilde\chi_{\Omega_n} f \rangle 
- \langle (T_n-\mu)^{-1}\chi_{\Omega_n}\phi, \chi_{\Omega_n} f \rangle, 
\end{aligned}
\end{equation}
and the second term satisfies
\begin{equation} \label{august3}
 \left|\langle (T-\mu)^{-1}\phi, \widetilde\chi_{\Omega_n} f \rangle\right|
 =
 \left|
 \langle  \widetilde\chi_{\Omega_n}\psi, f \rangle
 \right|
 \leq 
 \|f\| \|\psi\restriction \R^d\backslash B_{r_n}(0) \|. 
\end{equation}
Since $\mu\in\rho(T_n)$, we have $\lbar\mu\in\rho(T_n^*)$. Define $g_n:=(T_n^*-\lbar\mu)^{-1}\chi_{\Omega_n} f\in\dom(T_n^*)$.
Note that the functions $g_n$ are uniformly bounded, $\|g_n\|\leq \sup_n \|(T_n-\mu)^{-1}\| \,\|f\|$.
Now the remaining terms on the right-hand side of~\eqref{eq.est.res.diff.1} can be written as
%
\begin{equation}\label{eq.est.res.diff.2}
 \begin{aligned}
  &\langle (T-\mu)^{-1}\phi, \chi_{\Omega_n} f \rangle 
  - \langle (T_n-\mu)^{-1}\chi_{\Omega_n}\phi, \chi_{\Omega_n} f \rangle
\\
&=
  \langle \psi, (T_n^*-\lbar\mu)g_n \rangle 
- \langle \chi_{\Omega_n} (T-\mu)\psi, g_n \rangle
\\
&=\langle \zeta_n\psi, (T_n^*-\lbar\mu)g_n \rangle
+ \langle \widetilde\zeta_n\psi, (T_n^*-\lbar\mu)g_n \rangle
\\
&\quad 
- \langle \chi_{\Omega_n} (T-\mu)\psi, \zeta_n g_n \rangle
- \langle \chi_{\Omega_n} (T-\mu)\psi, \widetilde\zeta_n g_n \rangle
\\
&=
 \langle \zeta_n\psi, T_n^*g_n \rangle
-\langle T\psi, \zeta_n g_n \rangle
+\langle \widetilde\zeta_n\psi, \chi_{\Omega_n}f \rangle 
-\langle \chi_{\Omega_n} \phi, \widetilde\zeta_n g_n \rangle.
 \end{aligned}
\end{equation}
The last two terms can be estimated easily,
\begin{equation}\label{estimate.sect.4}
\begin{aligned}
&\left|
  \langle \widetilde\zeta_n\psi, \chi_{\Omega_n}f \rangle
- \langle \chi_{\Omega_n} \phi, \widetilde\zeta_n g_n \rangle
 \right|
\leq \|\widetilde \zeta_n \psi\|  \|f\| + \|\widetilde \zeta_n \phi \| \|g_n\|
\\
& \leq \|f\|\big( \big\|\psi\restriction \R^d\backslash B_{r_n}(0)\big\|+\sup_n \big \|(T_n-\mu)^{-1} \big \| \big\|\phi\restriction \R^d\backslash B_{r_n}(0)\big\|\big).
\end{aligned}
\end{equation}
Below we show, separately for the cases $\theta < \pi/2$ and $\theta \geq \pi/2$, that the first two terms of~\eqref{eq.est.res.diff.2} satisfy
\begin{equation}\label{dif.id.1}
\begin{aligned}
\langle \zeta_n\psi, T_n^*g_n \rangle
-\langle T\psi, \zeta_n g_n \rangle 
=
\langle \psi \Delta\zeta_n, g_n \rangle 
+ 2 \langle \psi \nabla\zeta_n, \nabla g_n \rangle
\end{aligned}
\end{equation}
and there exists $C_3\geq 0$ such that 
\begin{equation}\label{est.2}
\begin{aligned}
 \|\nabla g_n\|_n^2&
\leq
C_3 \|f\|^2.
\end{aligned}
\end{equation}
Then, since $\nabla\zeta_n\restriction B_{r_n}(0)=0$, $\Delta\zeta_n\restriction B_{r_n}(0)=0$, it follows that there is $C_4\geq 0$ with
\begin{equation}\label{estimate.sect.5}
\begin{aligned}
&\left|
 \langle \zeta_n\psi, T_n^*g_n \rangle
 - \langle  T\psi, \zeta_ng_n \rangle
 \right|
\\
& \leq (\|g_n\|_n \|\Delta\zeta_n\|_{\infty} +2\|\nabla g_n\|_n \|\nabla\zeta_n\|_{\infty})\big\|\psi\restriction \R^d\backslash B_{r_n}(0)\big\|
\\
& \leq 
C_4 \|f\|\big\|\psi\restriction \R^d\backslash B_{r_n}(0)\big\|.
\end{aligned}
\end{equation}
Thus summarizing \eqref{august3}--\eqref{estimate.sect.5} 
we obtain \eqref{rate.phi.psi}.

It remains to prove \eqref{dif.id.1} and \eqref{est.2}.
First we study the case $\theta < \pi/2$.
Since $\psi\in\dom(T)\subset \WotRd \cap\dom(q_0)$ and $\zeta_n \in \CcRdR$ with $\supp\zeta_n\subset B_{r_n+1}(0)\subset\Omega_n$, we conclude that $\zeta_n\psi\in\dom(t_n)=\dom(t_n^*)$ and
\begin{equation}
\langle \zeta_n\psi, T_n^*g_n \rangle
=\overline{\langle T_n^*g_n, \zeta_n\psi \rangle}
=\overline{t_n^*(g_n, \zeta_n\psi)}
=t_n(\zeta_n\psi,g_n).
\end{equation}
Moreover, it follows from $g_n\in\dom(T_n^*)\subset \Wotlam\cap\dom(q_0)$ and the properties of $\zeta_n$ that $\zeta_n g_n\in\dom(t)$.
Hence, using $\supp \zeta_n \subset \Omega_n$, assumption \eqref{ass.w.loc}, and integration by parts, we obtain
\begin{equation}\label{estimate.sect.1}
\begin{aligned}
&\langle \zeta_n\psi, T_n^*g_n \rangle
-\langle T\psi, \zeta_n g_n \rangle 
=
t_n(\zeta_n\psi, g_n )-t(\psi, \zeta_n g_n)
\\
&\quad =
\langle \nabla(\zeta_n\psi), \nabla g_n \rangle
+ \int_{\partial\Omega_n^{\rm R}} a_n \zeta_n \psi \, \overline{g_n}  \,\rd \sigma + q_0(\zeta_n\psi, g_n ) + w(\zeta_n\psi,g_n)
\\
& \quad \quad
- \langle \nabla\psi, \nabla(\zeta_ng_n) \rangle
- q_0(\psi, \zeta_n g_n) - w(\psi, \zeta_n g_n)
\\
& \quad = 
\langle \nabla(\zeta_n\psi), \nabla g_n \rangle
- \langle \nabla\psi, \nabla(\zeta_n g_n) \rangle
\\
& \quad=
\langle \psi \nabla\zeta_n, \nabla g_n \rangle 
- \langle \nabla\psi, g_n \nabla\zeta_n \rangle 
\\
&\quad  =\langle \psi \Delta\zeta_n, g_n \rangle 
+ 2 \langle \psi \nabla\zeta_n, \nabla g_n \rangle,
\end{aligned}
\end{equation}
which proves \eqref{dif.id.1}.
The estimate~\eqref{eq.retn.vs.nablaf} implies that there exist $C_5, C_6\geq 0$ with
\begin{equation}
 \begin{aligned}
  |t_n^*[g_n]|=|t_n[g_n]|\geq \re t_n[g_n]\geq C_5 \left(\|\nabla g_n\|_n^2+\re q_0[g_n]\right)-C_6 \|g_n\|_n^2.
 \end{aligned}
\end{equation}
Since $\re q_0[g_n]\geq 0$ and $T_n^*g_n=\overline{\mu} g_n+\chi_{\Omega_n} f$, there exists $C_3\geq 0$ such that
\begin{equation}\label{estimate.sect.2}
\begin{aligned}
 \|\nabla g_n\|_n^2
 \leq 
\frac{1}{C_5}\left(|t_n^*[g_n]|+C_6 \|g_n\|_n^2\right)
=\frac{1}{C_5}\left(|\langle T_n^*g_n, g_n\rangle_n|+C_6 \|g_n\|_n^2\right)
\leq C_3 \|f\|^2.
\end{aligned}
\end{equation}

For $\theta \geq \pi/2$, we first note that $T_n^*=-(\Delta_n^{\rm DR})^*+(Q_0-U+W)^*$ where the adjoint of the potential is simply its complex conjugate, \cf~the proof of Theorem~\ref{thm.norm.res.conv} for details.
Hence, with $\zeta_n \psi \in \dom(s_{0,n})= \dom(s_{0,n}^*)$ and integrating by parts in the last step, we get
\begin{equation}\label{estimate.nonsect}
\begin{aligned}
\langle \zeta_n\psi, T_n^*g_n \rangle
-\langle T\psi, \zeta_n g_n \rangle 
&=
\langle \zeta_n\psi, -(\Delta_n^{\rm DR})^*g_n \rangle 
+ \langle (Q_0-U+W)\zeta_n\psi, g_n \rangle 
\\
&\quad
-\langle -\Delta\psi, \zeta_n g_n \rangle - 
\langle (Q_0-U+W)\zeta_n\psi, g_n \rangle 
\\
& =\langle \psi \Delta\zeta_n, g_n \rangle 
+ 2 \langle \psi \nabla\zeta_n, \nabla g_n \rangle.
\end{aligned}
\end{equation}
Finally, adapting \eqref{nabla.n.est} and \eqref{Tn.graph.norm} for $T_n^*$, we see that there are $C_7, C_3\geq 0$ with
\begin{align*}\label{estimate.nonsect.2}
 \|\nabla g_n\|_n^2&
\leq 
 C_7 (\|T_n^*g_n\|^2 + \|g_n\|^2)
\leq
C_3 \|f\|^2.
\qedhere
\end{align*}

\end{proof}


To conclude this section, we study the convergence properties of the pseudospectra of the operators $T_n$; here we use the following definition,
\cf~\cite{Trefethen-2005} for an overview.
\begin{definition}
Let $\eps>0$. The \emph{$\eps$-pseudospectrum} $\sigma_\eps(A)$ of a closed operator $A$ is 
\begin{equation}
 \sigma_{\eps}(A):=\sigma(A) \cup \left\{\lm\in\rho(A):\,\|(A-\lm)^{-1}\|>\frac{1}{\eps}\right\}.
\end{equation}
\end{definition}
To study convergence of the sequence of $\sigma_{\varepsilon}(T_n)$, $n\in\N$, we need a suitable metric for closed unbounded subsets of $\C$.
We use convergence in \emph{Attouch-Wets metric}  $d_{\rm AW}$ which is a generalization of convergence in Hausdorff metric for unbounded sets,
\cf~\cite[Chap.\ 3]{Beer-1993-268} for details and further discussions.
We refrain from giving the definition of  $d_{\rm AW}$ here since we only use
the following, equivalent, characterization,  \cf~\cite[Cor.\ 3.1.8]{Beer-1993-268}. 

Let $\Lambda, \Lambda_n  \subset \C$ be closed non-empty subsets. Then
the sequence $\{\Lambda_n\}_n$ converges to $\Lambda$ in Attouch-Wets metric,
%
$d_{\rm AW}(\Lambda_n,\Lambda ) \to 0$, $n \to \infty$,
%
if and only if for all closed balls $\overline{B_\rho(0)}$, $\rho >0$,
\begin{equation}\label{dAW.crit}
\max \left\{ 
\sup_{w \in \Lambda_n \cap \overline{B_\rho(0)}} \dist(w, \Lambda ), 
\sup_{z \in \Lambda \cap \overline{B_\rho(0)} } \dist(z, \Lambda_n )  \right\} \to 0, \quad n \to \infty.
\end{equation}
%

\begin{theorem}\label{thm.pseudo}
Let Assumption~\ref{ass.lam.n} be satisfied and assume that
\begin{enumerate}[\upshape I.]
	\item in the case of semi-angle $\theta < \pi/2$, Assumption~\ref{ass.q.sec} holds and $T$, $T_n,$ $n\in \N$, are the operators defined in Propositions~{\rm\ref{prop.T.sec.def},~\ref{prop.Tn.def.sec}}, respectively;
	\item in the case of semi-angle $\theta \geq \pi/2$, Assumption~\ref{ass.q.non-sec} holds with $\bU<1$ and  $T$,~$T_n,$ $n\in \N$, are the operators defined in Propositions~{\rm\ref{prop.T.def},~\ref{prop.comp.res}}, respectively. 
\end{enumerate}%
\noindent
Then, for any \vspace{-1mm} $\eps>0$,
\begin{equation}
d_{\rm AW}\big( \hspace{0.3mm} \overline{\sigma_{\varepsilon}(T_n)},\overline{\sigma_{\varepsilon}(T)} \hspace{0.5mm} \big) \to 0, \quad n \to \infty. 
\end{equation}	

\end{theorem}
\begin{proof}
Since $T$ has compact resolvent, \cf~Proposition \ref{prop.T.sec.def} (for semi-angle $\theta < \pi/2$) and Proposition~\ref{prop.T.def} (for semi-angle $\theta \geq \pi/2$), its resolvent norm is not constant on any open subset of $\rho(T)$, \cf~\cite[Thm.\ 2.2]{Davies-2014}. 
Then the claim follows from the generalized norm resolvent convergence of $T_n$ to $T$, similarly as in \cite[Thm.\ 2.1]{Boegli-2014-80}.
In fact, without assuming that condition (ii) in \cite[Thm.\ 2.1]{Boegli-2014-80} holds, the claim of \cite[Lem.\ 4.3]{Boegli-2014-80} can be modified as follows. 
For every $\delta>0$ and $\mathcal K \subset \C$ compact, there exists $n_0 \in \N$ such that 
\begin{equation}\label{eq.delta.nbh}
\overline{\sigma_{\varepsilon}(T_n)} \cap \mathcal K \subset B_\delta(\overline{\sigma_{\varepsilon}(T)}), \quad
\overline{\sigma_{\varepsilon}(T)} \cap \mathcal K \subset B_\delta(\overline{\sigma_{\varepsilon}(T_n)}), \quad n\geq n_0,
\end{equation}
where $B_\delta(\Lambda)$ denotes the open $\delta$-neighbourhood of the set $\Lambda$ (called $\omega_\delta(\Lambda)$ in \cite{Boegli-2014-80}). 
Now~\eqref{eq.delta.nbh} yields the convergence~\eqref{dAW.crit} which, in turn, implies convergence in Attouch-Wets metric.
\end{proof}

\section{Remarks on exterior domains}
\label{sec.ext.dom}
In this section, we extend our results  to the situation of a Schr\"odinger operator $T_{\Omega}$ acting in $L^2(\Omega,\C)$ where $\Omega \subset \Rd$ is an exterior domain, \ie~$\Rd\setminus\Omega$ is compact.
We focus on dimension $d\geq 2$ since in $d=1$ an exterior domain is not connected, although we can also treat the case when $\Omega \subset \R$ is a half-line. 
The generalization is almost straightforward and the proofs are analogous. 
Therefore we only mention major differences and additional ingredients. 

For an exterior domain $\Omega$ we define the corresponding operator $T_\Omega$ in an analogous way as in Section \ref{sec.op.Rd}.
Since $\partial \Omega$ is non-empty, we now have to impose boundary conditions ensuring that $T_\Omega$ has non-empty resolvent set.
While for the $m$-sectorial case $\theta < \pi/2$ we can allow a combination of Dirichlet and Robin conditions, determined by a function $a_{\rm in}: \partial \Omega^{\rm R} \to \C$,
for the case of semi-angle $\theta \ge \pi/2$ only Dirichlet conditions are allowed; this restriction is due to Kato's Theorem~\cite[Thm.\ VII.2.5]{edmundsevans} which we use to define $T_\Omega$ as a perturbation of an $m$-accretive operator.

\begin{ass}\label{ass.ext}
Let $d\geq 2$ and let $\Omega \subset \Rd$ be an exterior domain, \ie~$\R^d\backslash\Omega\neq\emptyset$ is compact, 
\begin{equation}\label{eq.splitting.Omega}
\partial \Omega = \partial \Omega^{\rm D} \, \dot\cup \,  \partial \Omega^{\rm R}
\end{equation} 
with $\partial \Omega^{\rm D}$ 
closed and with $\partial\Omega^{\rm R}=\emptyset$ if $\theta\geq \pi/2$. 
If $\theta< \pi/2$ and $\partial \Omega^{\rm R} \neq \emptyset$, we assume 
\begin{enumerate}[label=(\ref{ass.ext}.\rm{\roman{*})}] 
\item \emph{regularity of $\partial\Omega$}: 
$\Omega$ is Lipschitz; 
\item \emph{control of Robin boundary term}:
$a_{\rm in} \in L^{\infty}(\partial \Omega^{\rm R},\C)$.
\end{enumerate}
\end{ass}

The main results in Sections~\ref{sec.conv} and~\ref{sec.sp.conv} generalize in a straightforward way to the situation of an exterior domain if analogues of the claims in Sections \ref{sec.op.Rd} and~\ref{sec.Tn} are available. In Subsections \ref{subsec.sec.ext}, \ref{subsec.nonsec.ext} below, we provide proof ideas of the latter and indicate additional modifications in order to prove the following theorem. 
Here, by $\|\cdot\|_{\Omega}$ we denote the norm of $L^2(\Omega,\C)$.

\begin{assOm}
	\label{ass.q.sec.Om}
	The sesquilinear form $q$ decomposes as $q=q_0 + w$ where $q_0$ and $w$ have the following properties.
	The form $q_0$ is generated by $Q_0 \in L^{1}_{\rm \loc}(\overline \Omega,\C)$, \ie~
	\begin{equation}\label{q0.def.Om}
	q_0[\cdot]:= \int_{\Omega} Q_0 |\cdot|^2\,\rd x, 
	\quad \dom(q_0):=\left\{ f\in L^2(\Omega,\C) \, : \, Q_0 |f|^2\in L^1(\Omega,\C) \right\},
	\end{equation}
	such that
	\begin{enumerate}[label=(\ref{ass.q.sec}.\rm{\roman{*}.$\Omega$)},labelsep=0.15in,leftmargin=0.55in]
		\item \label{ass.q.sec.i.Om}
		\emph{sectoriality of $\,Q_0$ with semi-angle $\theta \!<\! \pi/2$}:
		there exist $c_0\!>\!0$ and $\theta \!\in\! [0,\pi/2)$ with
		\begin{equation}\label{Q0.sec.Om}
		\begin{aligned}
		\re Q_0 \geq c_0,
		\quad 
		|\im Q_0| \leq \tan \theta \, \re Q_0
		;
		\end{aligned}
		\end{equation}
		\item \label{ass.q.sec.unbdd.Om} \emph{unboundedness of $\,Q_0$ at infinity}:
		\begin{equation}
		|Q_0(x)| \to \infty \ \ {\rm as} \ \ |x| \to \infty.
		\end{equation}
	\end{enumerate}
	For the form $w$, there  exist
	$R>r>0$ and $\zeta \in \CcRdR$ with
	\begin{equation}
	\Rd\setminus \Omega \subset B_r(0), \quad \supp \zeta \subset B_R(0), \quad 0 \leq \zeta \leq 1, \quad \zeta \restriction B_r(0)=1,
	\end{equation}
	sesquilinear forms $w_1\!$, $w_2$ with $\overline{\DD_{\Omega,R}}^{\|\cdot\|_{W^{\!1\!,2}(\Omega,\C)}} \!\!\!\subset\!\! \dom (w_1)$, $\dom(w_2)\!\!=\!\!L^2(\Omega,\!\C)$
	where 
	$\DD_{\Omega,R}:= \{ f \!\in\! C^{\infty}(B_R(0) \cap \Omega,\C) : \exists f_0 \in C_0^\infty(B_R(0),\C), f \!=\! f_0 \restriction \Omega, 
	\, \supp f \cap \partial \Omega^{\rm D} \!=\! \emptyset\}$~with
	\begin{equation}\label{eq.splitting.w.Om}
	 \forall\, f \in \dom(w) : \quad
\sqrt{\zeta}f\in \overline{\DD_{\Omega,R}}^{\|\cdot\|_{W^{1,2}(\Omega,\C)}}, 
\quad 
 w[f]=w_1[\sqrt{\zeta}f]+w_2[\sqrt{1-\zeta}f],
	\end{equation}
	such that
	\begin{enumerate}[label=(\ref{ass.q.sec}.\rm{\roman{*}.$\Omega$)},labelsep=0.15in,leftmargin=0.55in]
		\setcounter{enumi}{2}
		\item\label{ass.q.sec.w.Om}
		\emph{$\|\nabla\cdot\|_\Omega^2$-boundedness of $w_1$ in $L^2(B_R(0) \cap \Omega,\C)$}: 
		there are  $a_w \!\geq\! 0$, $b_w \!\in\! [0,1)$ so that, for every $f \in \overline{\DD_{\Omega,R}}^{\|\cdot\|_{W^{1,2}(\Omega,\C)}}$,
		\begin{equation}
		\begin{aligned}
		\quad |w_1[f] | \leq  a_w \|f \|_\Omega^2 + b_w \| \nabla f \|_\Omega^2;
		\end{aligned}
		\end{equation}
		\item\label{ass.q.sec.w.2.Om} \emph{boundedness of $w_2$ outside $B_r(0)$}:
		there exists $M_w \geq 0$ so that, for every $f\in L^2(\Omega,\C)$,
		\begin{equation}\label{ass.w.eq.Om}
		|w_2[(1-\chi_{r,\Omega}) f]| \leq M_w \|f\|_\Omega^2,
		\end{equation}
		where $\chi_{r,\Omega}$ is the characteristic function of $B_r(0) \cap \Omega$.
	\end{enumerate}
\end{assOm}

\begin{assOm}
	\label{ass.q.non-sec.Om}
	The function $Q \in L^2_{\rm loc}(\Omega,\C)$ decomposes as
	\begin{equation}
	Q = Q_0 - U + W
	\end{equation}
	where $\re Q_0 \geq 0$, $U \geq 0$, $U \re Q_0 =0$, $W \in L^2_{\rm loc}(\overline \Omega,\C)$, and the following hold. 
	\begin{enumerate}[label=(\ref{ass.q.non-sec}.\rm{\roman{*}.$\Omega$)},labelsep=0.15in,leftmargin=0.55in]
		\item \label{ass.q.non-sec.reg.Om}
		\emph{regularity of $Q_0$ and $U$}:
		$Q_0 \!\in\! W^{1,\infty}_{\rm \loc}(\overline{\Omega},\C)$, $U \!\in\! L^{\infty}_{\rm \loc}(\overline \Omega,\R)$,
		and there exist $\an$, $\bn$, $\aU$, $\bU$  $\geq 0$ such that 
		\begin{equation}\label{ass.q.non-sec.rb.Om} 
		\begin{aligned}
		|\nabla Q_0 |^2 \leq \an + \bn |Q_0|^2,  \quad 
		U^2 \leq \aU  +  \bU |\im Q_0|^2;
		\end{aligned}
		\end{equation}
		\item \label{ass.q.non-sec.unbdd.Om} 
		\emph{unboundedness of $Q_0$ at infinity}:
		\begin{equation}
		|Q_0(x)| \rightarrow \infty \ \  {\rm as} \ \  \ {|x| \rightarrow \infty}.
		\end{equation}
	\end{enumerate}
	There exist $R>r>0$ such that $\Rd\setminus \Omega \subset B_r(0)$ and
	\begin{enumerate}[label=(\ref{ass.q.non-sec}.\rm{\roman{*}.$\Omega$)},labelsep=0.15in,leftmargin=0.55in]
		\setcounter{enumi}{2}
		\item\label{ass.q.non-sec.W.Om}
		\emph{$\Delta$-boundedness of \,$W$ in $L^2(B_R(0) \cap \Omega,\C)$}: 
		there exist $a_W \geq 0$, $b_W \in [0,1)$ such that, for every $f \in \{ f \!\in\! W_0^{1,2}(B_R(0)\cap \Omega,\C) : \Delta f \in L^2(B_R(0) \cap \Omega,\C) \}$,
		\begin{equation}
		\|W f \|_\Omega^2 \leq a_W \|f \|_\Omega^2 + b_W \|\Delta f\|_\Omega^2;   
		\end{equation}
		
		\item\label{ass.q.non-sec.W.2.Om}
		\emph{boundedness of \,$W$ outside $B_r(0)$}: 
		there exists $M_W \geq 0$ such that
		\begin{equation}
		\|(1-\chi_{r,\Omega}) W \|_{\infty} \leq M_{W},
		\end{equation}
		where $\chi_{r,\Omega}$ is the characteristic function of $B_r(0) \cap \Omega$.
	\end{enumerate}
	
\end{assOm}

\begin{theorem}\label{thm.ext}
	Let Assumptions~\ref{ass.ext} and \ref{ass.lam.n} hold with $\Rd$ replaced by $\Omega$ and, if applicable, 
	the regularity assumption {\rm\ref{ass.lam.n.div}} with $\partial \Omega_n$ replaced by the smaller set 
	$\partial \Omega_n \setminus \partial \Omega$. 
	Further assume that 
	\begin{enumerate}[\upshape I.]
		\item  for semi-angle $\theta \!<\! \pi/2$, Assumption~\ref{ass.q.sec.Om} holds;
		\item  for semi-angle $\theta \!\geq\! \pi/2$, Assumption~\ref{ass.q.non-sec.Om} holds with $b_U<1$. 
			\end{enumerate}
	Then the statements of Theorems~{\rm\ref{thm.norm.res.conv},~\ref{thm.conv.spectrum},~\ref{thm.conv.rate}}, and~{\rm\ref{thm.pseudo}} continue to hold with $T$, $T_n$ replaced by $T_\Omega$, $(T_\Omega)_n$, respectively .
\end{theorem}

\noindent
\emph{Proof of Theorem {\rm \ref{thm.ext}}:} Sketch of modifications of the proofs. 

\subsection{Semi-angle $\theta < \pi/2$}
\label{subsec.sec.ext}
The analogue of Proposition \ref{prop.T.sec.def} holds for the $m$-sectorial operator $T_{\Omega}$
which is uniquely determined by the closed sectorial form
\begin{equation}\label{t.def.ext}
\begin{aligned}
t_{\Omega}[f] & := \| \nabla f \|_\Omega^2 + 
 \int_{\partial \Omega^{\rm R}} 
a_{\rm in} |f|^2 \d  \sigma 
+ q_0[f] + w[f],\\
\dom(t_\Omega) & := \overline{\DD_\Omega}^{\, (\|\cdot\|^2_{W^{1,2}(\Omega,\C)} + \re q_0[\cdot])^\frac12},
\end{aligned}
\end{equation}
	where
	\begin{equation}\label{D.Omega.def}
	\begin{aligned}
	\DD_\Omega&:=\big \{
	f \!\in\! C^{\infty}(\Omega,\C) : \exists f_0 \in \CcRdC, f \!=\! f_0 \restriction \Omega, 
	\, \supp f \cap \partial \Omega^{\rm D} \!=\! \emptyset 
	\big
	\}.
	\end{aligned}
	\end{equation}
To show that $T_\Omega$ has compact resolvent, first observe that
\begin{equation}\label{Dom.tOm.encl}
\dom (t_\Omega) \subset W^{1,2}(\Omega,\C) \cap \dom(q_0)
\end{equation}
and that, by Assumptions \ref{ass.q.sec.w.Om}, \ref{ass.q.sec.w.2.Om}, and a trace embedding analogous to~\eqref{emb.est.n},
there is a constant $c>0$ so that $(\re t_\Omega[\cdot] + c \|\cdot\|_\Omega^2)^{1/2}$ is equivalent to $(\| \cdot \|^2_{W^{1,2}(\Omega,\C)} \!+\! \re q_0[\cdot])^{1/2}$.
Next, similarly as in \eqref{phin.psin.conv.sec}, there is $C>0$ such that,  for all sufficiently large $n\in \N$ and all $f \in \dom(t_\Omega)$, 
\begin{equation}\label{dom.Om.dec}
\int_{x \in\Omega, |x| \geq n} |f|^2 \d  x
\leq 
C
\frac{\re t_\Omega[f] + c \|f\|^2_{\Omega}}{
	\mathop{\rm ess \, inf}\limits_{x \in \Omega, |x| \geq n} \re  Q_0}.
\end{equation}
Therefore \cite[Thm.\ 2.33]{Adams-2003} and Assumption \ref{ass.q.sec.unbdd.Om} imply that
the embedding $(\dom(t_\Omega)$, $(\re t_\Omega[\cdot] + c\|\cdot\|_{\Omega}^2)^{1/2}) \hookrightarrow L^2(\Omega,\C)$ is compact.

The approximating operators $(T_{\Omega})_n$ are introduced analogously as in Section~\ref{sec.Tn}.
In fact, under Assumption~\ref{ass.lam.n} with $\Rd$ replaced by $\Omega$, Lemma~\ref{lem.s0.def} and Proposition~\ref{prop.Tn.def.sec} are generalized in a straightforward way; there appears an additional boundary term as in \eqref{t.def.ext} which is a harmless relatively bounded perturbation with relative bound $0$. 

In order to prove an analogue of the generalized strong resolvent convergence in Lemma~\ref{lem.s.res.conv}, 
we use that~$\DD_\Omega$ in \eqref{D.Omega.def} is a core of $t_{\Omega}$
 and rely on the estimate \eqref{dom.Om.dec}; the latter is also used to prove the analogue of  Lemma \ref{lem.discr.comp},
\ie~the discrete compactness of the embeddings. 

Norm resolvent convergence and convergence of spectra and pseudospectra then follow in a straightforward way from these analogues of Lemmas~\ref{lem.s.res.conv},~\ref{lem.discr.comp}.

\subsection{Semi-angle $\theta \geq \pi/2$}
\label{subsec.nonsec.ext}
In order to prove an analogue of Proposition~\ref{prop.T.def}, take ${R_\Omega}>0$ sufficiently large and $\xi \in \CcRdR$ such that 
$$\Rd \setminus \Omega \subset B_{R_\Omega-2}(0), \quad \xi\restriction B_{R_\Omega}(0)  = 1, \quad \supp \xi \subset B_{{R_\Omega}+1}(0).$$ 
Then the closure of 
\begin{align}
(T_{\Omega})_{\min}&:=-\Delta+Q,\\
\dom((T_{\Omega})_{\min})&:=\{f \in W^{1,2}_0(\Omega,\C) \, : \, (-\Delta + Q_0)f \in L^2(\Omega,\C),\, (1-\xi)f \in C_0^\infty(\Omega,\C)\},
\\[-8mm]
\end{align}
is given by
\begin{equation}\label{eq.def.TOmega}
T_{\Omega}=-\Delta + Q, \quad \dom(T_{\Omega})=\{ f \in W^{1,2}_0(\Omega,\C) \, : \, (-\Delta + Q_0)f \in L^2(\Omega,\C) \}.
\end{equation}
In fact, we may proceed similarly as in the proof of Proposition~\ref{prop.T.def}, starting with
\begin{equation}
(T_{\Omega})_{0,\min}:=-\Delta + Q_0, \quad \dom((T_{\Omega})_{0,\min}):= \dom((T_{\Omega})_{\min}).
\end{equation}
Then  Lemmas \ref{lem.tau0.def}--\ref{lem.norm.eq} may be generalized
for $(T_{\Omega})_{0,\min}$ and $(T_{\Omega})_{\min}$, with similar arguments as for the generalizations for the operators $T_{0,n}$, \cf~Lemmas~\ref{lem.T0n.g.norm} and~\ref{lem.T0n.W}.
To this end, we use that every function $f\in\dom((T_{\Omega})_{0,\min})$ has compact support and belongs to the domain of the Laplacian defined in $L^2(\Omega,\C)$
(because $\xi f$ and $(1-\xi)f$ both belong to the latter domain);
note that  the quadratic form has  no boundary term because $\partial \Omega = \partial \Omega^{\rm D}$ by the assumptions. 
We thus arrive at the analogue of the estimate~\eqref{nab.est} and at a norm equivalence similar to~\eqref{norm.equiv},
\ie~ there exist $\beta_{\Omega}, k_{\Omega}, K_{\Omega}>0$ such that, for all $f\in\dom((T_{\Omega})_{0,\min})=\dom((T_{\Omega})_{\min})$,
\begin{equation}\label{nab.est.ext}
\| \nabla f \|^2_{\Omega}
\leq
\frac{\beta_{\Omega}}{2}\left\|\Delta f\right\|^2_{\Omega}+\frac{1}{2\beta_{\Omega}}\left\|f\right\|^2_{\Omega}
\end{equation}
and
\begin{equation}\label{norm.equiv.ext}
\begin{aligned}
& 
k_{\Omega}
\left(
\|\Delta f\|^2_{\Omega}+ \|Q_0f\|^2_{\Omega} + \|f\|^2_{\Omega}
\right) 
\\
&
\leq 
\|(T_{\Omega})_{\min} f\|^2_{\Omega}+\|f\|^2_{\Omega}
\\
&
\leq 
K_{\Omega}
\left(
\|\Delta f\|^2_{\Omega}+ \|Q_0 f\|^2_{\Omega} + \|f\|^2_{\Omega}
\right).
\end{aligned}
\end{equation}
The latter continues to hold for the closure of $(T_{\Omega})_{\min}$.
Below we prove that the closure of $(T_{\Omega})_{0,\min}$ is $(T_{\Omega})_0$ where
\begin{equation}
(T_{\Omega})_0:=-\Delta + Q_0, \quad \dom((T_{\Omega})_0):=\{ f \in W^{1,2}_0(\Omega,\C) \, : \, (-\Delta + Q_0)f \in L^2(\Omega,\C) \};
\end{equation}
then~\eqref{eq.def.TOmega} follows from~\eqref{norm.equiv.ext}.

To justify that $(T_\Omega)_{0}$ is the closure of $(T_{\Omega})_{0,\min}$, we employ two cut-off functions $\zeta_i \in C^\infty(\Omega,\R)$, $i=0,1$, that satisfy
\begin{alignat*}{5}
0 &\leq \zeta_0 \leq 1, \quad &\zeta_0\restriction \Omega \setminus B_{R_\Omega-1}(0) &=1, \quad &\zeta_0\restriction  B_{R_\Omega-2}(0) \cap \Omega &= 0,\\
0 &\leq \zeta_1 \leq 1, \quad &\zeta_1\restriction \Omega \setminus B_{R_\Omega}(0) &=1, \quad &\zeta_1 \restriction B_{R_\Omega-1}(0) \cap \Omega &= 0.
\end{alignat*}
Note that these properties yield
\begin{equation}\label{eq.cutoff}
\zeta_0\zeta_1=\zeta_1, \quad \zeta_i(1-\xi)=1-\xi, \quad i=0,1.
\end{equation}
The potential $\widetilde Q_0:=\zeta_0 Q_0$ satisfies Assumption~\ref{ass.q.non-sec}.
Thus Proposition \ref{prop.T.def} and its proof imply that
\begin{equation}
T_0: = -\Delta + \widetilde Q_0, 
\quad 
\dom(T_0):= \big\{f\in \WotRd: \, (-\Delta + \widetilde Q_0) f\in \LRd \big\},
\end{equation}
has the separation property, \ie~$\dom(T_0) = W^{2,2}(\R^d,\C) \cap \{ f \in \LRd: \widetilde Q_0 f \in \LRd \}$, 
and $\CcRdC$ is a core of $T_0$.
Let  $-\Delta^{\rm D}_{B_{R_{\Omega}+1}(0) \cap \Omega}$
be the Dirichlet Laplacian in $L^2(B_{{R_\Omega}+1}(0)\cap\Omega,\C)$ defined via its quadratic form.
Observe that if $f \in \dom((T_{\Omega})_0)$, then $\xi f\in \dom(-\Delta^{\rm D}_{B_{R_{\Omega}+1}(0) \cap \Omega})$ and $(1-\xi) f\in\dom(T_0)$. 
Since  $\CcRdC$ is a core of $T_0$, there exists a sequence $\{f_n\}_n \subset \CcRdC$ that converges to the function $(1-\xi) f \in \dom(T_0)$ in the graph norm of $T_0$. 
Using~\eqref{eq.cutoff} and~\eqref{nab.est.ext},~\eqref{norm.equiv.ext}, one may verify that the same holds for the sequence $\{\zeta_1 f_n\}_n\subset C_0^{\infty}(\Omega,\C)$.
Then 
$$(T_\Omega)_0 (1-\xi)f = T_0 (1-\xi)f, \quad (T_\Omega)_0 \zeta_1 f_n =T_0 \zeta_1 f_n$$
implies that $\{\xi f + \zeta_1 f_n\}_n \subset \dom((T_\Omega)_{0,{\rm min}})$ approximates $f$ in the graph norm of~$(T_\Omega)_0$, 
and so the claim follows.

The operator $(T_{\Omega})_0$ is $m$-accretive, \cf~Kato's theorem \cite[Thm.\ VII.2.5]{edmundsevans}.
Using Assumption~\ref{ass.q.non-sec.unbdd.Om}, we obtain that the resolvent of $(T_\Omega)_0$ is compact. The same holds for $T_\Omega$ by a perturbation argument as in the proof of Proposition \ref{prop.T.def}.
In the same way, we prove a resolvent estimate similar to~\eqref{Res.dec}.

The approximating operators are introduced similarly as in Section \ref{subsec.Tn.non-sec}.

\indent
Generalized strong resolvent convergence, \cf~Lemma~\ref{lem.s.res.conv}, 
 and discrete compactness of the embeddings, \cf~Lemma \ref{lem.discr.comp}, can be verified by straightforward generalizations of the given proofs; here we make use of the analogue of~\eqref{phi.psi.cnv.2} which follows from~\eqref{norm.equiv.ext}.

As in the case $\theta<\pi/2$, the claims in Theorem~\ref{thm.ext} then follow from these analogues of Lemmas~\ref{lem.s.res.conv},~\ref{lem.discr.comp}.

\section{Examples}
\label{sec.ex}

In this section, we present numerical examples for dimensions $d=1,2,3$. 
All numerical computations arising in this section were performed on a standard dual-core Linux machine with the use of the software Wolfram Mathematica 9. The differential equations on finite intervals were solved numerically by implementing a shooting method in Mathematica.
\subsection{\boldmath Potentials $Q(x)=\I x$ and  $Q(x)=\I x^3$ in $\R$}
\label{subsec.ex.d1}
The sets $\Omega_n$ are intervals $(-s_n,s_n)$ with $s_n\nearrow \infty$ as $n\to \infty$ 
and we impose various boundary conditions at the endpoints $\pm s_n$. 
\subsubsection*{Potential $Q(x)=\I x$}
Here  
the resolvent of $T$ is compact and the spectrum of $T$~is empty, \cf~for example \cite{Sibuya-1975, Almog-2008-40}, 
whereas the spectrum of $T_n$ in $L^2((-s_n,s_n),\C)$ is~\emph{not},
\begin{equation}\label{last}
 \sigma(T)=\emptyset, \quad \sigma(T_n) \ne \emptyset, \ n\in \N,
\end{equation}%
since $T_n$ is a bounded perturbation of $-\rd^2/\rd x^2$ with separated  boundary conditions;
moreover, the system of eigenfunctions and associated functions of the operator $T_n$ forms a Riesz basis, 
\cf~\cite{Mikhajlov-1962-3}. 
The pseudospectra of $T$ are also well-studied, \cf~\cite{BordeauxMontrieux-2013}. 
For  $T_n$ with Dirichlet conditions, a detailed analysis of the bottom of the spectrum 
can be found in \cite{Beauchard-2013}:  it was proved that
\begin{equation}\label{ev.ix.lim}
\lim_{n \to \infty} (\inf \re \sigma(T_n)) = \frac{|\mu_1|}{2}
\end{equation}
where $0>\mu_1 \approx -2.338$ is the first zero of the Airy function, \cf~\cite[Thm.\ 3.1]{Beauchard-2013}.

We illustrate the behaviour of the eigenvalues of $T_n$ with Dirichlet conditions  in Figure~\ref{plot_ir}; for Robin conditions, the plots look similar.
The parameter $s_n$ is chosen as $s_n:=0.05 n$ for $n=1,\dots,200$, hence $s_n\in [0,10]$. 
\begin{figure}[h]
	\includegraphics[width=0.49\textwidth,height=4cm]{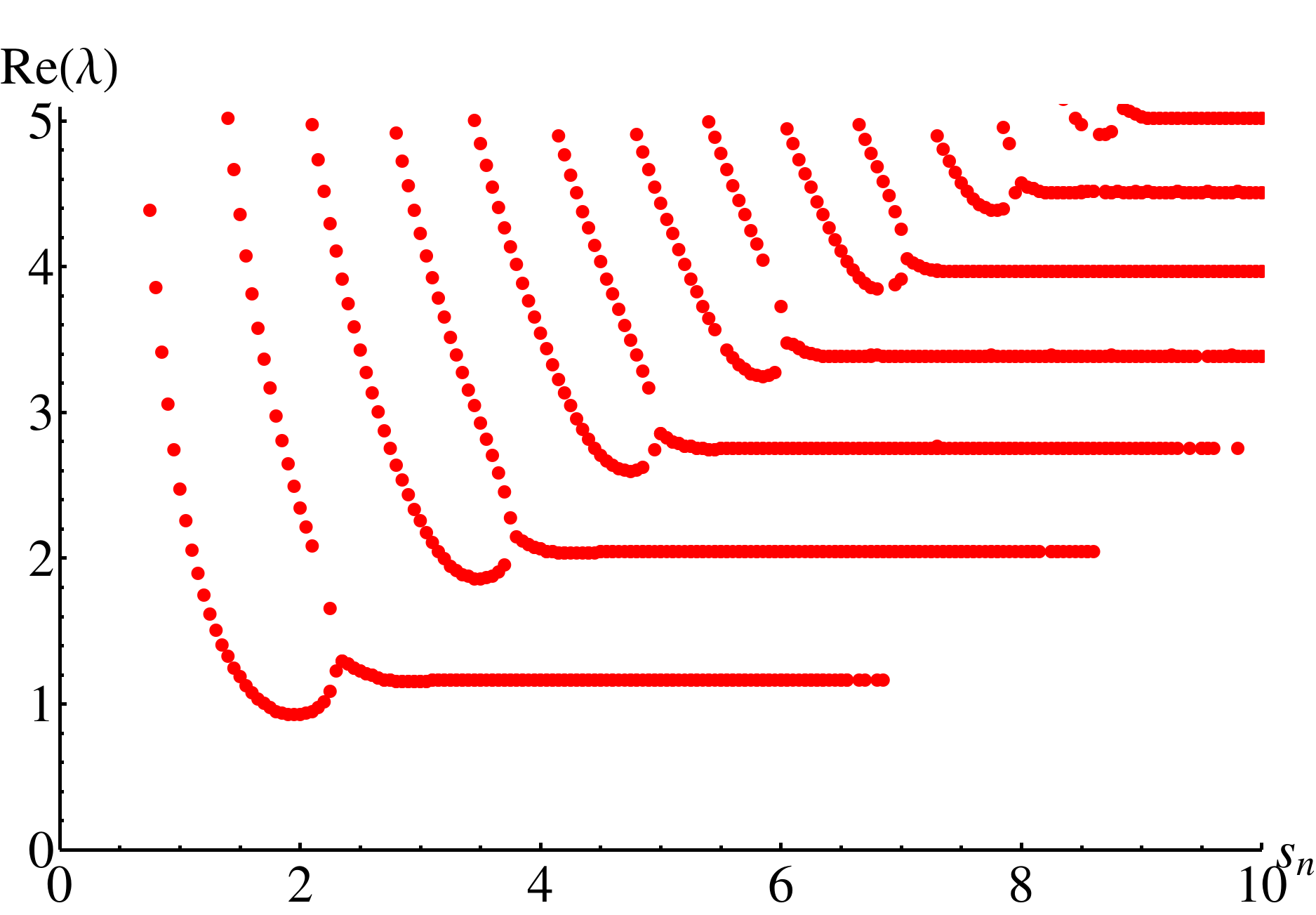} \hfill
	\includegraphics[width=0.49\textwidth,height=4cm]{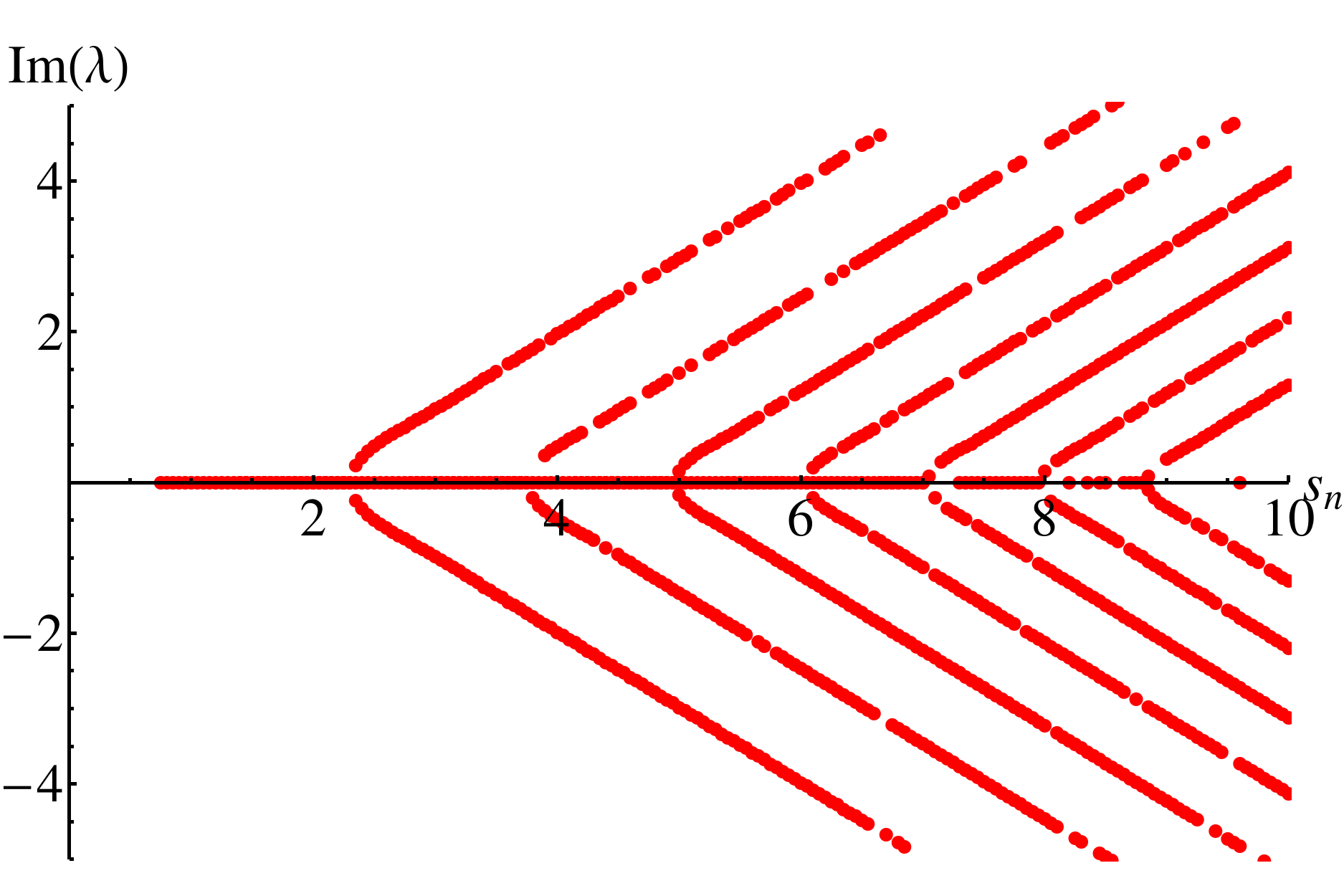}
	\caption{\label{plot_ir}
		$Q(x)=\I x$: Real (left) and imaginary (right) part of the eigenvalues of $T_n$ on $(-s_n,s_n)$  and Dirichlet conditions.}
\vspace{-1mm}		
\end{figure}
We see that, as $n$ increases, every eigenvalue $\lm_n$ of $T_n$ meets another, and they form a complex conjugate pair.
The real parts of this pair seem to converge, whereas the imaginary parts diverge to $\pm \infty$ in almost straight lines. Hence in the limit $n \to \infty$ there are no eigenvalues left, 
which is in agreement with \eqref{last} and our result on spectral exactness, \cf~Theorem~\ref{thm.conv.spectrum}. 
The behaviour of the eigenvalue with the lowest real part is in agreement with the result \eqref{ev.ix.lim}.

\subsubsection*{Potential $Q(x)=\I x^3$}

The imaginary cubic oscillator and related operators have been studied extensively, \cf~\cite{Caliceti-1980-75, Bender-1998-80, Shin-2002-229, Siegl-2012-86, Grecchi-2013-319, Henry-2013b, Giordanelli-2015-16}. 
It is known that all eigenvalues $\lambda^{(k)}$ of $T$ are real and simple, behave asymptotically as $k^{6/5}$, and the 
system of eigenfunctions of $T$ is complete in $L^2(\R,\C)$ but does not form a basis, \cf~\cite{Siegl-2012-86, Henry-2013b} for the two latter. 
As in the previous example, 
the system of eigenfunctions of $T_n$ forms a Riesz basis.

The plot in~Figure~\ref{plot_ir3} shows the behaviour of the eigenvalues of $T_n$ for increasing $\Omega_n=(-s_n,s_n)$ with Dirichlet and Neumann boundary conditions at $\pm s_n$, marked by red balls and blue squares, respectively; for Robin conditions the eigenvalue behaviour is similar.
\begin{figure}[h]
	\includegraphics[width=0.49\textwidth,height=4cm]{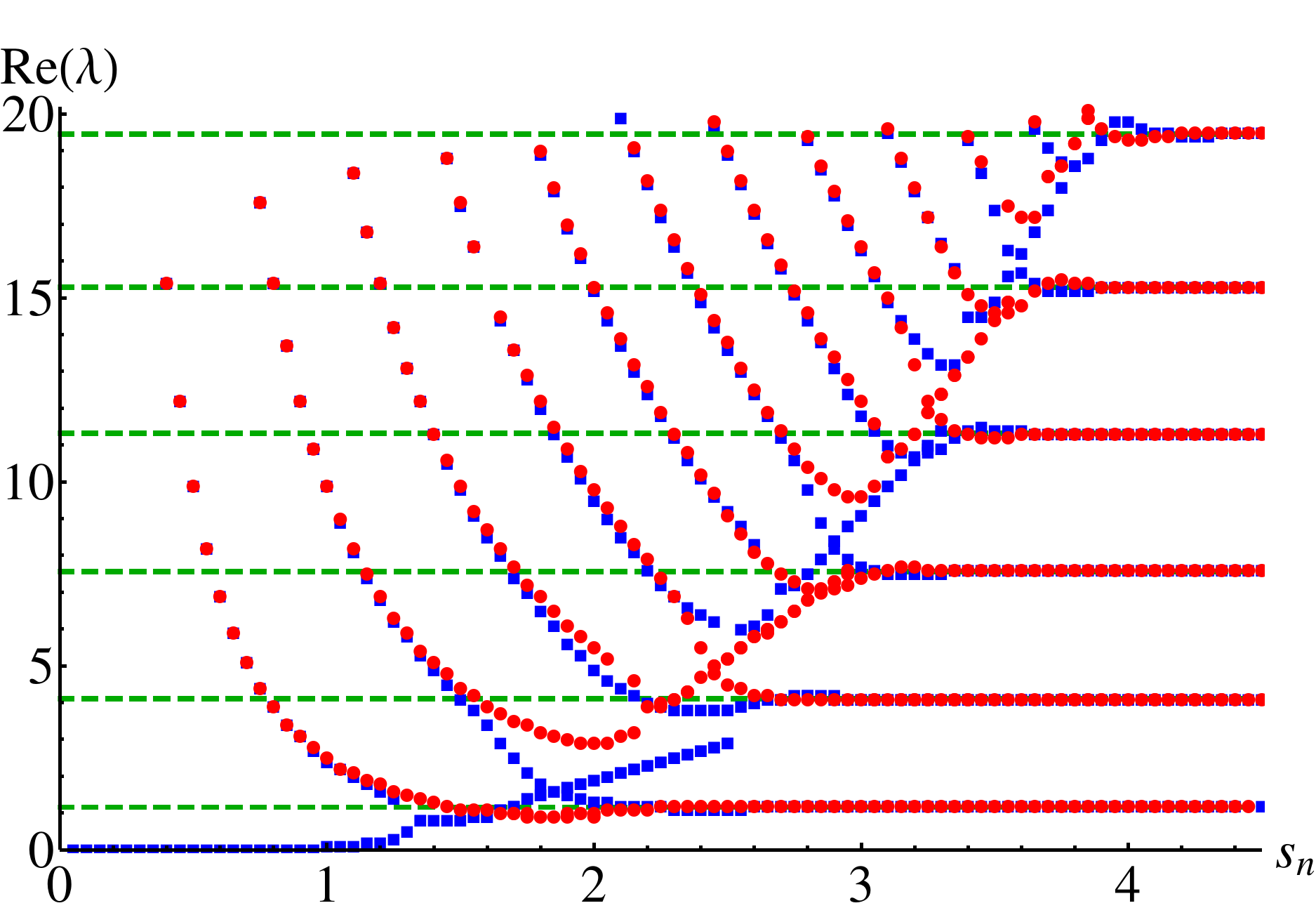} \hfill
	\includegraphics[width=0.49\textwidth,height=4cm]{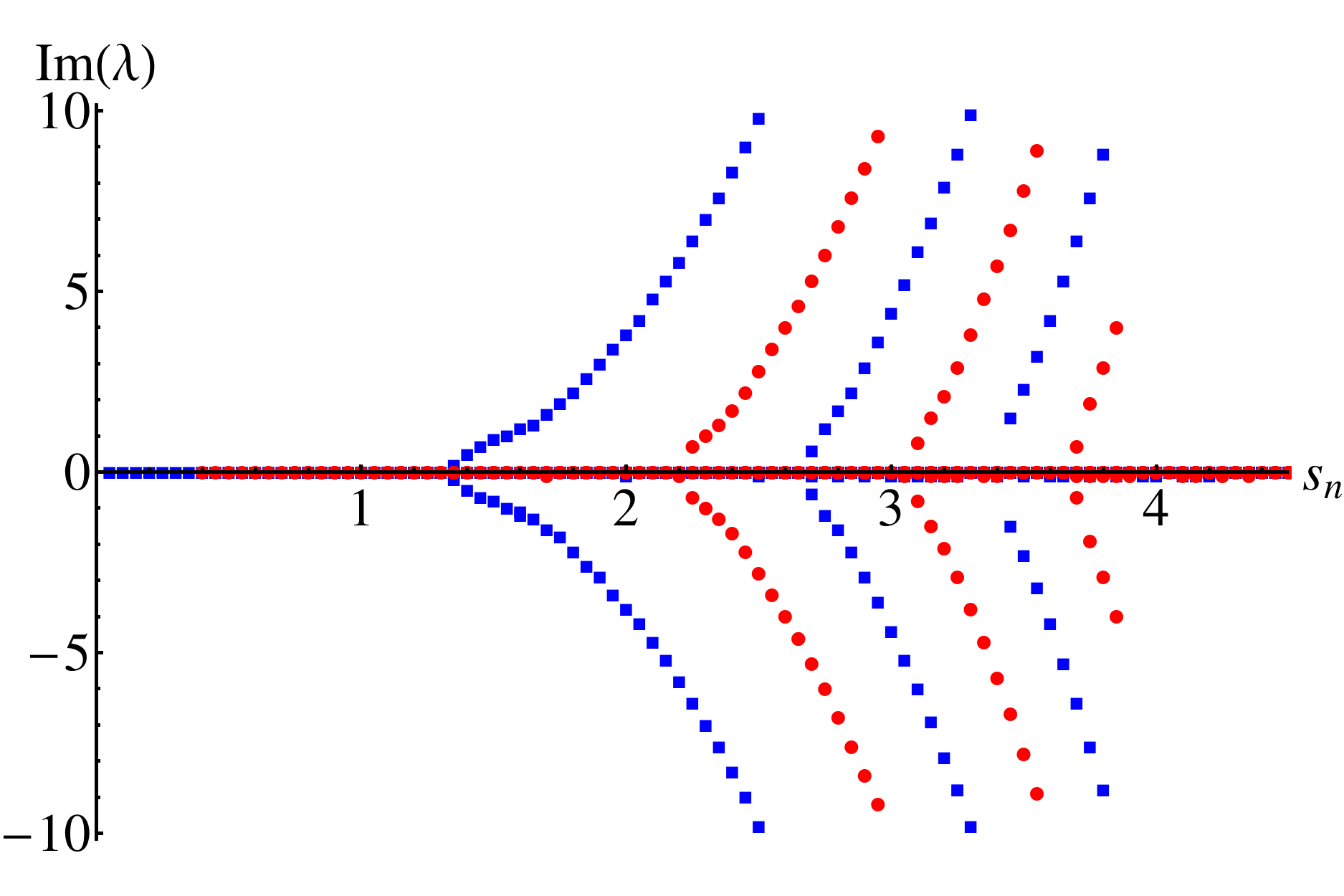}
	\caption{\label{plot_ir3}
		\small $Q(x)\!=\!\I x^3$: Real (left) and imaginary (right) part of the eigenvalues of $T_n$  on $(-s_n,s_n)$  with Dirichlet (red balls)\,/\,Neu\-mann (blue squares) conditions, and lowest 6 (real) eigenvalues of $T$ (dashed horizontal~lines).}
\vspace{-1mm}
\end{figure}

For small $n$, the eigenvalues are all real. 
If $n$ is increased, some eigenvalues 
form complex conjugate pairs and their imaginary parts diverge, so they do not have a limit in $\C$, 
while other eigenvalues do converge to a finite limit $\lm$ which must be an eigenvalue of $T$ due to our spectral exactness Theorem~\ref{thm.conv.spectrum}, 
and all eigenvalues of $T$ are approximated in this way. 
This also confirms that the numerically computed eigenvalues in \cite[Tab.\ 1]{Bender-1998-80} or the following ones computed by M. Tater, \cf~\cite{MT-priv}, 
\begin{equation}
\begin{aligned}
\lm_{\rm M.T.}^{(1)}&=1.1562671, \quad \lm_{\rm M.T.}^{(2)}=4.1092288, \quad \lm_{\rm M.T.}^{(3)}=7.5622739,\\
\lm_{\rm M.T.}^{(4)}&=11.314422,\quad \lm_{\rm M.T.}^{(5)}=15.291554, \quad \lm_{\rm M.T.}^{(6)}=19.451529,
\end{aligned}
\end{equation}
indeed approximate true eigenvalues.

Figure~\ref{plot_ir3} also shows that, for $s_n\geq 4$, at the bottom of the spectrum of $T$ the difference between our numerical approximations and their limits, \ie~the
true eigenvalues, marked by dashed horizontal lines, 
is already very small. 
For Dirichlet conditions at the endpoints $\pm s_n$, the convergence rate of $|\lm^{(1)}-\lm_n^{(1)}|$ is illustrated in  Figure~\ref{plot_ir3_convRate}, where 
\begin{equation}\label{eq.lm.1}
\lm^{(1)}:=\min\sigma(T)\approx \lm_{\rm M.T.}^{(1)}, \quad \lm_n^{(1)}:=\min\sigma(T_n).
\end{equation}
The left plot in  Figure~\ref{plot_ir3_convRate} is a zoom of Figure~\ref{plot_ir3} near the relevant values
which reveals that $\lm_n^{(1)}$ seems to converge to $\lm^{(1)}$ in an oscillatory manner;
the right plot, where the values of $\log |\lm^{(1)}-\lm_n^{(1)}|$ are shown to make the oscillations better visible,
suggests an exponential convergence rate of the eigenvalues as $s_n \nearrow \infty$. 
\begin{figure}[h]
	\includegraphics[width=0.49\textwidth,height=4cm]{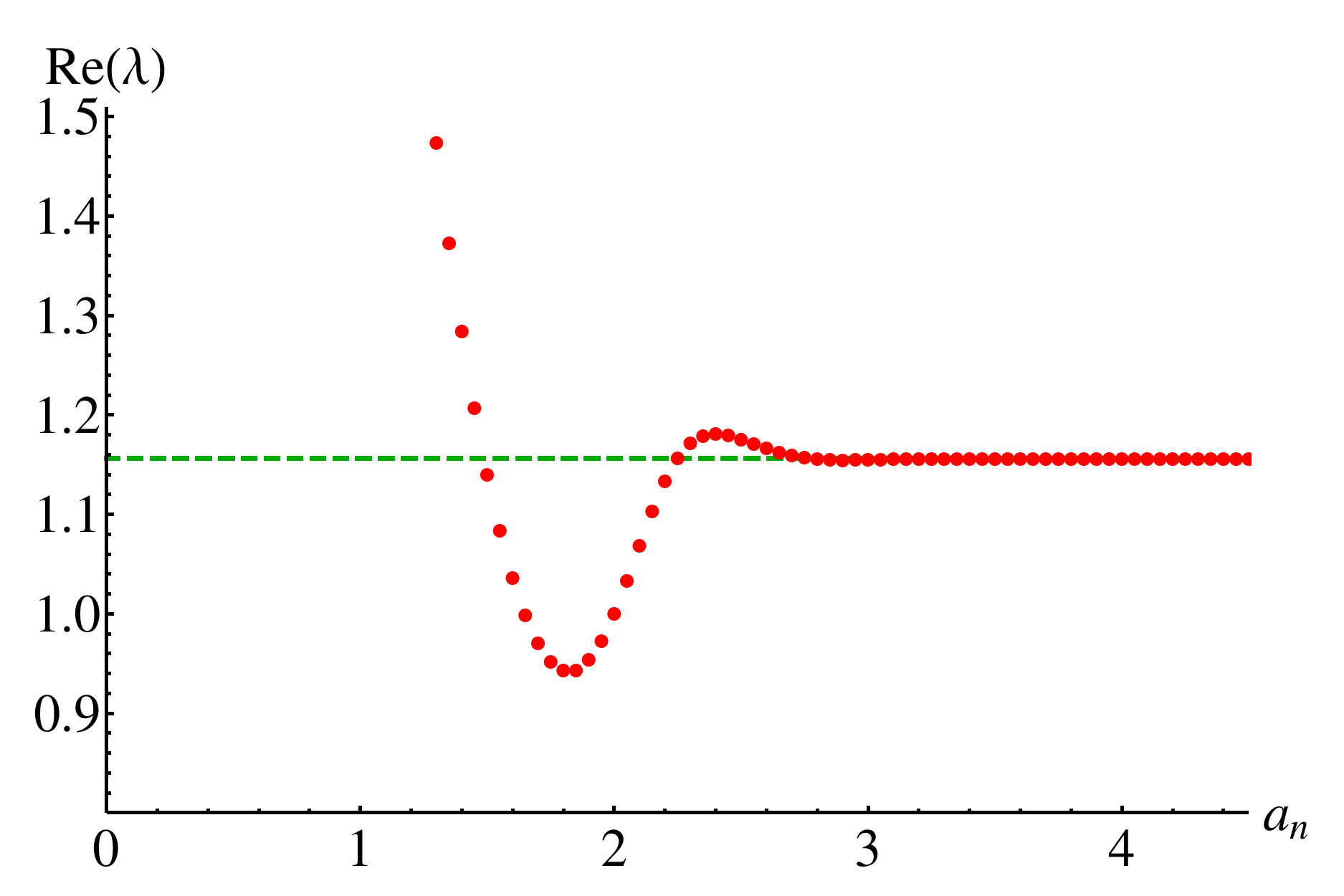}\hfill
	\includegraphics[width=0.49\textwidth,height=4cm]{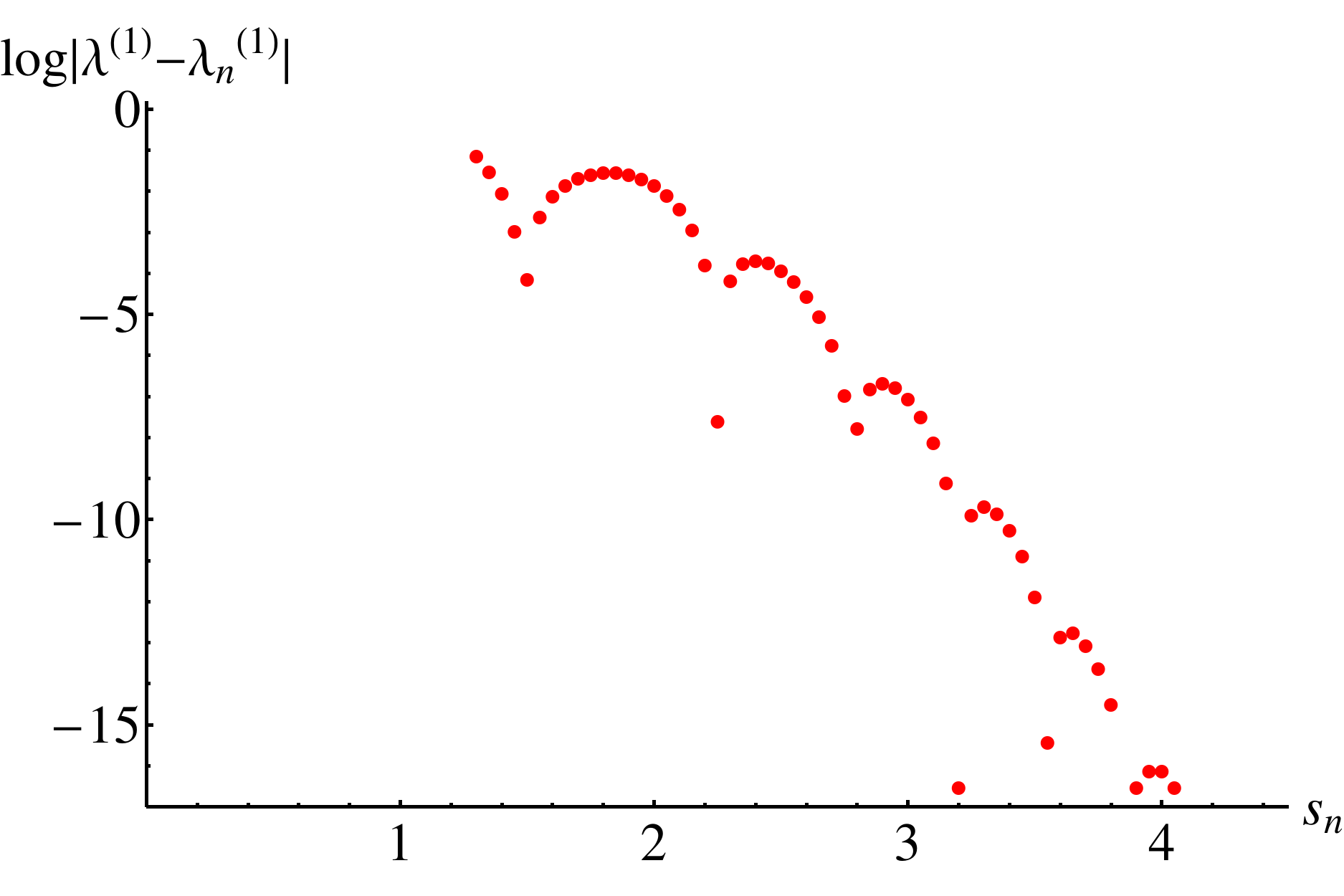}
	\caption{\label{plot_ir3_convRate}
		\small $Q(x)=\I x^3$: Approximation of lowest eigenvalue $\lm^{(1)}$ (dashed horizontal line on the left) of $T$  by lowest eigenvalue $\lm_n^{(1)}$ of $T_n$ with Dirichlet conditions.}
\vspace{-1mm}
\end{figure}

\subsection{\boldmath Harmonic oscillator in $\R^3$}
\noindent
In several dimensions, the simplest choice of $\Omega_n$ are cubes and balls. 
While the former are natural for potentials allowing for separation in Cartesian coordinates,  the latter are suitable for radial potentials, \ie~$Q(x)=Q(|x|)$. 
The harmonic oscillator, \ie~$Q(x):=|x|^2$, allows for both separations. In the following, we compare the approximations for cubes and balls.

For $\Omega_n=(-s_n,s_n)^3$ with $s_n\nearrow \infty$, 
the eigenvalue problem for $T_n$ is reduced to the one-dimensional problem
\begin{equation}\label{eq.harm.osc.separated}
-f''(x)+x^2 f(x)=\mu f(x), \quad x \in (-s_n,s_n),
\end{equation}
subject to Robin conditions at the artificial endpoints $\pm s_n$.
Any eigenvalue $\lambda$ of $T_n$ can be expressed as $\lm=\mu^{(i)}+\mu^{(j)} + \mu^{(l)}$ where $\{\mu^{(k)}\}_k$ are the eigenvalues of the one-dimensional problem~\eqref{eq.harm.osc.separated}.

For $\Omega_n=B_{s_n}(0)$ with $s_n\nearrow \infty$, the operator can be written in spherical coordinates. The eigenfunctions of $T_n$ can be factorized as $f(r,\theta,\varphi)=g(r)Y_l^m(\theta,\varphi)$.
Here the spherical harmonics $Y_l^m(\theta,\varphi)$, $m=-l,\dots,l$,\, $l\in\N_0$, satisfy $\Delta Y_l^m=\frac{l(l+1)}{r^2} Y_l^m$, while $g$ is an eigenfunction of the radial problem 
\begin{equation}\label{def.gn}
-g''(r)-\frac{2}{r} g'(r) + \left(\frac{l(l+1)}{r^2} +r^2\right) g(r) = \lm g(r),\quad r\in (0,s_n),
\end{equation}
with some Robin condition at $s_n$.

In  Figure~\ref{plot_harmosc_xyz}, we compare the eigenvalues of $T_n$ for cubes and balls with Dirichlet conditions; for Robin conditions the plots are similar.
The behaviour of the eigenvalue approximations does not differ much for  cubes and balls, both converge to true eigenvalues and all true eigenvalues are approximated in this way, 
\cf~Theorem~\ref{thm.conv.spectrum}.

The degeneracy of the $k$-th eigenvalue $\lm^{(k)}$ (ordered increasingly) is $\frac{k(k+1)}{2}$; 
note that, in Figure~\ref{plot_harmosc_xyz}, some curves represent eigenvalues of higher multiplicity so that the sum of the multiplicities of all curves converging to $\lm^{(k)}$ equals its degeneracy.

\begin{figure}[h]
	\includegraphics[width=0.49\textwidth,height=4cm]{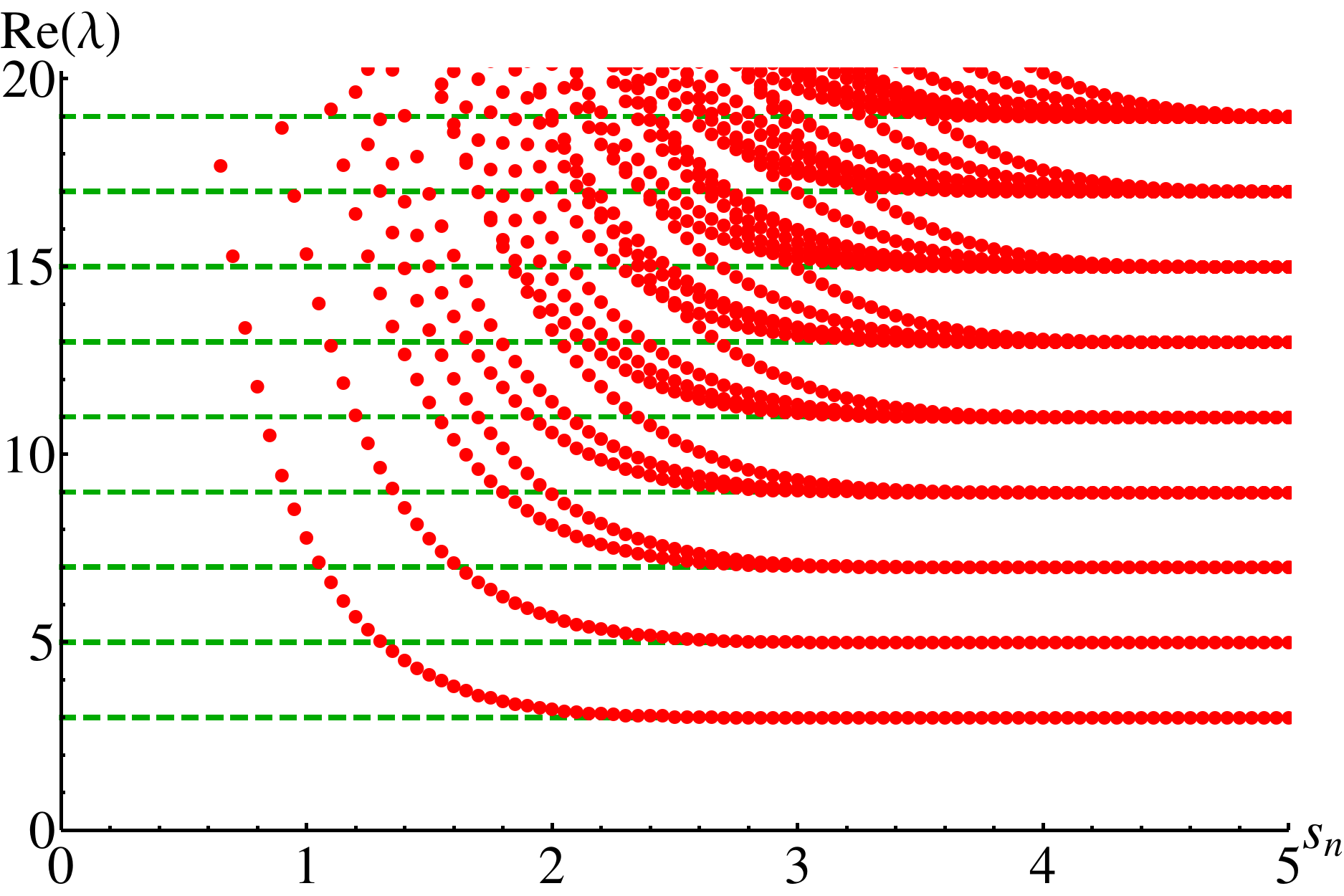} \hfill
	\includegraphics[width=0.49\textwidth,height=4cm]{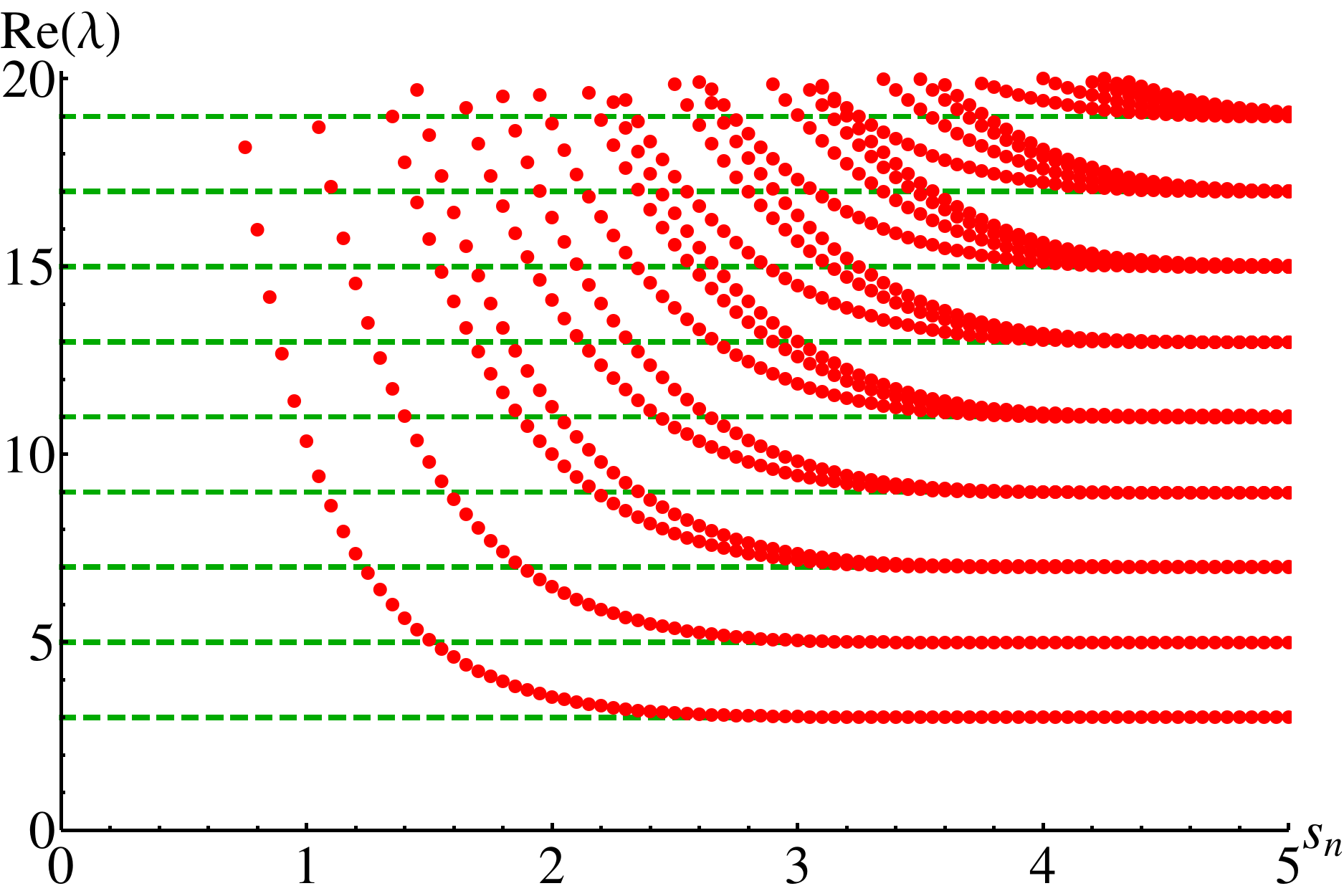}
	\caption{\label{plot_harmosc_xyz}
	\small Eigenvalues of harmonic oscillator in $\R^3\!$ approximated on cubes $\Omega_n\!=\!(-s_n,s_n)^3\!$ (left) and balls $\Omega_n\!=\!B_{s_n}\!(0)$ (right) for Dirichlet conditions.}
\vspace{-1mm}
\end{figure}

\subsection{Rotated oscillator on exterior domain in $\R^2$.}
\label{subsec.ext}
In \cite[Ex.\ 1]{Brown-2004-24}, the \vspace{-1mm} problem 
\begin{equation}
\begin{aligned}
-\Delta f(x,y)+ (1+3\I) (x^2+y^2) f(x,y)&=\lambda f(x,y),&\quad &x^2+y^2\geq 1,\\
f(x,y)&=0,& \quad &x^2+y^2=1,
\end{aligned}
\end{equation}
was studied, but it could not be decided whether domain truncation is spectrally exact. Our new result, \cf~Theorem \ref{thm.conv.spectrum}, yields a definite and positive~result. 

If we truncate the exterior domain $\R^2\backslash B_1(0)$ to $\Omega_n\!:=\!B_{s_n}(0)\backslash \overline{B_1(0)}$ for some sequence $s_n\!\nearrow\! \infty$ as $n\!\to\! \infty$ and impose Dirichlet  conditions on the outer boundary, 
we obtain a spectrally exact approximation, \cf~Theorem \ref{thm.ext}.
In polar coordinates, the truncated problem decouples into an infinite system of problems that depend on $l\in\N_0$ (representing the angular part),
\begin{equation}\label{eq.ext.domain.truncated}
\begin{aligned}
-f''(r)-\frac{1}{r} f'(r)+ \left((1+3\I) r^2+\frac{l^2}{r^2}\right)f(r)&=\lambda f(r),&\quad &r\in (1,s_n),\\
f(r)&=0,& \quad &r\in\{1,s_n\}. 
\end{aligned}
\end{equation}
We performed numerical computations to find and approximate the eigenvalues in the box $[0,20]+ [0,15]\,\I$ for different $l\in\N_0$ and increasing $s_n$. For $l\geq 7$, no eigenvalue was found in this box.
For $l=0,1,\dots,6$, the eigenvalues in the box change very little (less than $10^{-7}$) for $s_n\in [5,10]$. So the numerical approximations for $s_n=10$ shown in 
Table~\ref{table_exterior}\footnote{We mention that the numerical values in \cite{Brown-2004-24} should be modified correspondingly, \cf~\cite{Marco}.}
are already near true \vspace{-1mm} eigenvalues.
\begin{table}[htb!]
	\small
	\begin{tabular}{ll}
		{\small Value of $l$}\hspace{0.5cm} &{\small Approximate eigenvalues $\lm_n$ up to 7 digits}\ \ \\ \hline \\[-1.5ex]
		$l=0$ 	& $\lm_{n}\approx\,8.1962583+9.8951098\,\I$\\[0.5mm]
		$l=1$ 	& $\lm_{n}\approx\,8.5747825+9.9950630\,\I$\\[0.5mm]
		$l=2$ 	& $\lm_{n}\approx\,9.6945118+10.3061585\,\I$\\[0.5mm]
		$l=3$ 	& $\lm_{n}\approx\,11.5061205+10.8625746\,\I$\\[0.5mm]
		$l=4$ 	& $\lm_{n}\approx\,13.9201983+11.7211938 \,\I$\\[0.5mm]
		$l=5$ 	& $\lm_{n}\approx\,16.7923324+12.9529682\,\I$\\[0.5mm]
		$l=6$ 	& $\lm_{n}\approx\,19.9029928+14.6018978 \,\I$\\[2mm]
	\end{tabular}
	\caption{\label{table_exterior}
		\small Eigenvalues $\lm_{n}\in[0,20]+ [0,15]\,\I$ of~\protect\eqref{eq.ext.domain.truncated} for $s_n=10$.}
\end{table}
%
%
\subsection*{Acknowledgements}
{\small
The authors gratefully acknowledge the support of several  funding bodies: 
Swiss National Science Foundation, SNF, grant no.\ 200020\_146477 (S.B., C.T.); 
SNF Early Postdoc.Mobility project P2BEP2\_159007 (S.B.);
SCIEX-NMS Fellowship 11.263 and SNF Ambizione project PZ00P2\_154786 (P.S.).

They also thank A.\ Hansen for drawing their attention to the Attouch-Wets metric and M.\ Tater for numerical values on the imaginary cubic oscillator.
P.S.\ also thanks D.\ Krej\v ci\v r\'ik for fruitful discussions
\eg~on the results in \cite{Krejcirik-2010-94}. 
}

\vspace{-2mm}

\bibliography{mybib}
\bibliographystyle{acm}

\end{document}